\documentclass[hidelinks,onefignum,onetabnum]{siamart220329}

\usepackage{adjustbox} 
\usepackage{tikz} 
\usetikzlibrary{arrows,shapes,positioning,automata}
\usetikzlibrary{calc,decorations.markings}
\usetikzlibrary{decorations.pathreplacing}
\usetikzlibrary{snakes}

\usepackage{stmaryrd} 
\usepackage{bm}
\usepackage{multirow} 
\usepackage{diagbox} 
\usepackage{booktabs} 
\usepackage{makecell} 
\usepackage{enumitem}

\usepackage{algorithm}
\usepackage{algorithmic}

\usepackage{lipsum}
\usepackage{amsfonts}
\usepackage{graphicx}
\usepackage{epstopdf}
\usepackage{algorithmic}
\ifpdf
  \DeclareGraphicsExtensions{.eps,.pdf,.png,.jpg}
\else
  \DeclareGraphicsExtensions{.eps}
\fi

\newsiamremark{remark}{Remark}
\newsiamremark{hypothesis}{Hypothesis}
\crefname{hypothesis}{Hypothesis}{Hypotheses}
\newsiamthm{claim}{Claim}

\headers{Dirichlet-Neumann Learning Algorithm for Interface Problems}{Q. Sun, X. Xu, and H. Yi}

\title{Dirichlet-Neumann learning algorithm for solving elliptic interface problems\thanks{Submitted to the editors DATE.
\funding{Qi Sun is supported in part by the National Natural Science Foundation of China under Grant 12201465.}}}

\author{Qi Sun\thanks{School of Mathematical Sciences, Tongji University, Shanghai 200092, China
(\email{qsun\_irl@tongji.edu.cn}, \email{2111166@tongji.edu.cn}).}
\and Xuejun Xu\footnotemark[2]$\,\,^,$\thanks{Institute of Computational Mathematics, AMSS, Chinese Academy of Sciences, Beijing 100190, China (\email{xxj@lsec.cc.ac.cn}).}
\and Haotian Yi\footnotemark[2]}

\usepackage{amsopn}

\ifpdf
\hypersetup{
  pdftitle={An Example Article},
  pdfauthor={D. Doe, P. T. Frank, and J. E. Smith}
}
\fi

\begin{document}

\maketitle

\begin{abstract}
Non-overlapping domain decomposition methods are natural for solving interface problems arising from various disciplines, however, the numerical simulation requires technical analysis and is often available only with the use of high-quality grids, thereby impeding their use in more complicated situations. To remove the burden of mesh generation and to effectively tackle with the interface jump conditions, a novel mesh-free scheme, \textit{i.e.}, Dirichlet-Neumann learning algorithm, is proposed in this work to solve the benchmark elliptic interface problem with high-contrast coefficients as well as irregular interfaces. By resorting to the variational principle, we carry out a rigorous error analysis to evaluate the discrepancy caused by the boundary penalty treatment for each decomposed subproblem, which paves the way for realizing the Dirichlet-Neumann algorithm using neural network extension operators. The effectiveness and robustness of our proposed methods are demonstrated experimentally through a series of elliptic interface problems, achieving better performance over other alternatives especially in the presence of erroneous flux prediction at interface.
\end{abstract}



\section{Introduction}

Many problems in science and engineering are carried out with domains separated by curves or surfaces, \textit{e.g.}, the abrupt changes in material properties between adjacent regions, from which the interface problems naturally arise. A widely studied benchmark example is the elliptic interface problem with high-contrast coefficients \cite{li2006immersed,chen1998finite,leveque1994immersed}, whose solution lies in the Sobolev space $H^{1+\epsilon}(\Omega)$ with $\epsilon>0$ possibly close to zero \cite{mercier2003minimal}. Due to the low regularity of solution at the interface, classical numerical methods, such as finite difference and finite element methods  \cite{leveque2007finite,brenner2008mathematical}, require the generation of an interface-fitted mesh in the discretization of the computation domain \cite{chen1998finite}, which could be technically involved and time consuming especially when the geometry of interface gets complicated. To ease the burden of manual mesh construction, numerical methods based on unfitted meshes, \textit{e.g.}, the immersed interface method \cite{leveque1994immersed} and many others, have emerged as an attractive alternative for solving the elliptic interface problems \cite{li2006immersed}. However, using unfitted mesh, \textit{e.g}, a uniform Cartesian mesh, often makes it difficult to enforce the jump conditions across subdomain interfaces accurately \cite{chen2017interface,he2022mesh}. Although both types of methods have shown to be effective to some extent, their practical implementation is not an easy task due to the complex geometry of interface and the discontinuity of solution.

\begin{figure}[t!]
\centering
\includegraphics[width=0.835\textwidth]{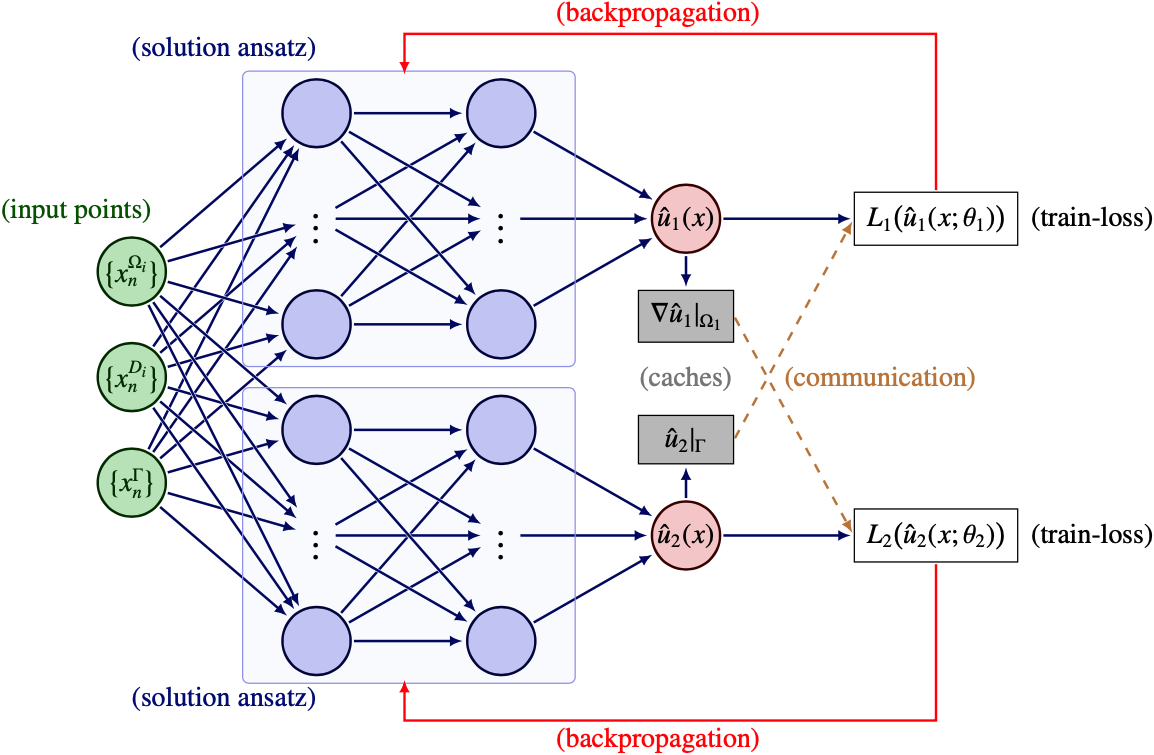}
\vspace{-0.2cm}
\caption{Computational graph of our proposed Dirichlet-Neumann learning algorithm.}
\label{fig-computation-graph}
\vspace{-0.45cm}
\end{figure}

Note that the entire computation domain has already been separated as the union of multiple non-overlapping subdomains, each of which corresponds to a local boundary value problem after endowing the subdomain interface with an appropriate boundary condition \cite{quarteroni1999domain}. As a result, the elliptic interface problem being considered can also be tackled by a non-overlapping Dirichlet-Neumann algorithm at the continuous level \cite{yang1997parallel,yang2000finite,quarteroni1999domain}, where the decomposed subproblems are typically solved using the mesh-based numerical methods \cite{leveque2007finite,brenner2008mathematical,leveque2002finite}. However, the complex geometry of subdomain interfaces remains a major concern during the meshing process. Fortunately, the domain decomposition methods \cite{toselli2004domain,quarteroni1999domain,dolean2015introduction} are essentially continuous schemes that greatly differ from the aforementioned interface-fitted or -unfitted methods, making it computationally feasible to adopt the mesh-free deep learning techniques \cite{goodfellow2016deep,karniadakis2021physics,yu2018deep} as the local problem solvers \cite{heinlein2021combining}. Thanks to the rapid development of artificial intelligence science, much attention has recently been paid to combining deep learning with insight from the domain decomposition methods. The physics-informed neural networks \cite{raissi2019physics,karniadakis2021physics,lagaris2000neural,lagaris1998artificial}, among others \cite{yu2018deep,zang2020weak,sirignano2018dgm}, has been utilized to discretize and solve the Dirichlet and Neumann subproblems within the classical Dirichlet-Neumann algorithm framework \cite{li2020deep}, which is named ``DeepDDM'' and applied to a simple interface problem as a proof of concept. To further enhance its scalability properties, the DeepDDM method is extended with the aid of coarse space correction \cite{mercier2021coarse,quarteroni1999domain}. However, these works fail to address the impact of erroneous flux prediction \cite{dockhorn2019discussion,bajaj2021robust} at interface, which can degrade the overall performance \cite{sun2022domain}. Note that in the degenerate case of homogeneous jump conditions, the continuity of averaged solution between neighbouring subdomains, as well as its first and higher-order derivatives, are explicitly enforced through additional penalty terms in a series of papers \cite{jagtap2020extended,jagtap2020conservative,shukla2021parallel,wu2022improved,hu2021extended}, which also suffer from the issue of erroneous flux predictions. Worse still, these strategies can not be applied to resolving discontinuous solutions, where the higher-order derivatives of network solutions at interface could be meaningless due to the lack of global regularity \cite{mercier2003minimal}. Designing specific network architectures is another way of dealing with the complex geometry and jump condition of elliptic interface problems \cite{lai2022shallow,hu2022discontinuity,tseng2022cusp,hu2022hybrid,dong2021local}, \textit{e.g.}, augmenting an additional coordinate variable as a feature input of the solution ansatz \cite{lai2022shallow}, replacing neural network structures with extreme learning machines \cite{dong2021local,dwivedi2021distributed} or graph neural networks \cite{taghibakhshi2022learning}, while the theoretical study is still in need to gain a deeper understanding.

In contrast to the existing methods that require an explicit computation of the flux data along subdomain interfaces, a novel Dirichlet-Neumann learning algorithm using neural network extension operators (see \autoref{fig-computation-graph}) is proposed in this work for solving the elliptic interface problems with high-contrast coefficients as well as irregular interfaces. By resorting to the variational principle and applying the Green's formula to Dirichlet-Neumann algorithm \cite{quarteroni1999domain,yang1997parallel,yang2000finite}, a rigorous error analysis is first established to estimate the discrepancies caused by the penalty treatment\footnote{The boundary conditions are included as ``soft'' penalty terms in the training loss function.} of boundary conditions \cite{jiao2021error,duan2021analysis,muller2022error}, which paves the way for constructing the training loss functions for each decomposed subproblem. As a direct result, the flux transmission between neighbouring subdomains is realized without explicitly computing the derivatives of network solution at interface (see also \autoref{fig-computation-graph}) and thereby alleviating the deterioration of outer iterations in the presence of erroneous flux predictions. Moreover, a comparison study is presented to demonstrate the robustness and effectiveness of our proposed methods over the DeepDDM scheme \cite{li2020deep}, followed by a series of numerical examples to validate our statements.

The rest of this paper is organized as follows. In \autoref{Section-Preliminaries}, we begin by recalling the Dirichlet-Neumann algorithm for solving elliptic interface problems in the continuous level, then the numerical methods, \textit{i.e.}, mesh-based and mesh-free solvers, are briefly reviewed and compared. To combine domain decomposition methods with ideas from machine learning in a consistent manner, a rigorous error analysis of boundary penalty method for both the Dirichlet and Neumann subproblems is presented in \autoref{Section-Method}, followed by the implementation details of our proposed Dirichlet-Neumann learning algorithm. Then, the numerical results on a series of benchmark problems are reported in \autoref{Section-Experiment}. Finally, we summarize our work in \autoref{Section-Conlcusion}.


\section{Preliminaries}\label{Section-Preliminaries}

\subsection{Elliptic Interface Problem with High-Contrast Coefficients}\label{Section-IntfcProb}

\begin{figure}[t!]
\centering
\includegraphics[width=0.65\textwidth]{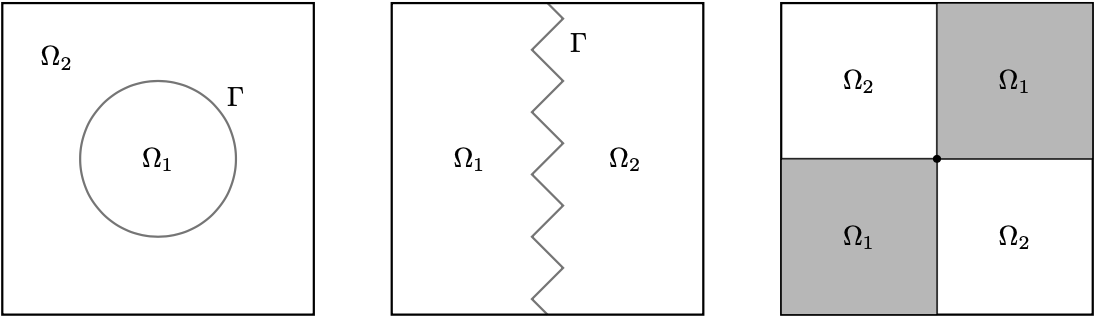}
\vspace{-0.2cm}
\caption{A Lipschitz domain $\Omega\subset\mathbb{R}^2$ that is decomposed into several non-overlapping subdregions.}
\label{fig-domain-decomposition}
\vspace{-0.3cm}
\end{figure}

Let $\Omega\subset\mathbb{R}^d$ be a bounded domain with Lipschitz boundary $\partial\Omega$, which is first assumed to be partitioned into two non-overlapping subdomains as depicted in \autoref{fig-domain-decomposition}, that is,
\begin{equation*}
	\overline{\Omega} = \overline{\Omega_1\cap\Omega_2},\ \ \ \Omega_1\cap\Omega_2=\emptyset,\ \ \ \Gamma=\partial\Omega_1\cap\partial\Omega_2.
\end{equation*}
Then, we consider the elliptic interface problem with high-contrast coefficients and natural jump conditions, which is often formally written as \cite{li2006immersed}
\begingroup
\renewcommand*{\arraystretch}{1.1}
\begin{equation}
\begin{array}{cl}
-\nabla \cdot \left( c(x) \nabla u(x)  \right) + u(x) = f(x)\ \ & \text{in}\ \Omega,\\
u(x) = 0\ \ & \text{on}\ \partial \Omega, \\
\llbracket u(x) \rrbracket = 0\ \ \text{and}\ \ \llbracket c(x) \nabla u(x)\cdot\bm{n} \rrbracket = q(x) \ \ & \text{on}\ \Gamma,
\end{array}
\label{IntfcProb-StrongForm}
\end{equation}
\endgroup
where $f(x)$ represents a given function of $L^2(\Omega)$, $\bm{n}=\bm{n}_2$ ($\bm{n}_1$) the unit outer normal vector for subdomain $\Omega_2$ ($\Omega_1$) and notation $\llbracket \cdot \rrbracket$ the difference of quantity across the interface, namely, for any point $X\in \Gamma$,
\begingroup
\renewcommand*{\arraystretch}{1.3}
\begin{equation}
\begin{array}{c}
\lim\limits_{x\rightarrow X,\, x\in \Omega_1} u(x) = \lim\limits_{x\rightarrow X,\, x\in \Omega_2} u(x), \\
-\lim\limits_{x\rightarrow X,\, x\in \Omega_1} c(x) \nabla u(x)\cdot \bm{n}_1 - \lim\limits_{x\rightarrow X,\, x\in \Omega_2} c(x) \nabla u(x)\cdot\bm{n}_2 = q(X).
\end{array}
\end{equation}
\endgroup
Notably, $c(x)$ is a piecewise constant function that has a finite jump of function value across the interface $\Gamma$, that is,
\begingroup
\renewcommand*{\arraystretch}{1.1}
\begin{equation*}
c(x) = \left\{
\begin{array}{cl}
c_1>0\ \ & \text{in}\ \Omega_1,\\
c_2\gg c_1\ \ & \text{in}\ \Omega_2.
\end{array}\right.
\end{equation*}
\endgroup
which is typically caused by the abrupt changes in material properties or the interaction of fluid dynamics \cite{li2006immersed}. More broadly, the regular coefficients $c_1\approx c_2>0$ or $c_1=c_2$ can be regarded as a degenerate case of problem \eqref{IntfcProb-StrongForm}.

By setting $V_i = \big\{ v_i\in H^1(\Omega_i)\, \big|\, v_i|_{\partial\Omega\cap\partial\Omega_i} = 0 \big\}$, $V_i^0=H_0^1(\Omega_i)$, and defining 
\begin{equation*}
	b_i(u_i, v_i) = \int_{\Omega_i} c_i \nabla u_i \cdot \nabla v_i + u_i v_i \,dx,\ \ (f, v_i)_i = \int_{\Omega_i} f v_i\, dx,\ \ (q,v_2)_{L^2(\Gamma)} = \int_\Gamma qv\,ds,
\end{equation*}
for $i=1$, 2, the Green's formula implies that the weak formulation of \eqref{IntfcProb-StrongForm} reads: find $u_1\in V_1$ and $u_2\in V_2$ such that
\begingroup
\renewcommand*{\arraystretch}{1.1}
\begin{equation}
\begin{array}{cl}
b_1(u_1, v_1) = (f, v_1)_1 & \text{for any}\ v_1\in V_1^0,\\
u_1 = u_2 & \text{on}\ \Gamma, \\
b_2(u_2, v_2) \!=\! (f, v_2)_2 \!+\! (f, R_1\gamma_0 v_2)_1 \!-\! b_1(u_1, R_1\gamma_0 v_2) \!-\! (q,v_2)_{L^2(\Gamma)} & \text{for any}\ v_2\in V_2,
\end{array}
\label{IntfcProb-WeakForm}
\end{equation}
\endgroup
where $\gamma_0 v= v|_\Gamma$ is the restriction of $v\in H^1(\Omega_i)$ on $\Gamma$ and $R_i: H_{00}^{\frac12}(\Gamma)\to V_i$ any differentiable extension operator \cite{quarteroni1999domain}. 

As a direct result, by employing a suitable relaxation parameter $\rho\in (0,\rho_{\textnormal{max}})$ \cite{na2022domain,gander2015optimized}, an iterative scheme (also known as the Dirichlet-Neumann algorithm \cite{yang1997parallel,quarteroni1999domain}) can be developed for solving the elliptic interface problem \eqref{IntfcProb-WeakForm}: given the initial guess of the unknown solution value at interface $u_\Gamma^{[0]}\in H_{00}^{\frac12}(\Gamma)$, then solve for $k\geq 0$, 
\begin{itemize}[leftmargin=*]
\item[1)] $\displaystyle u_1^{[k]} = \operatorname*{arg\,min}_{u_1\in V_1,\, u_1|_\Gamma = u_\Gamma^{[k]}}  \frac12  b_1(u_1, u_1) - (f, u_1)_1$\hfill\refstepcounter{equation}\textup{(\theequation)}\label{IntfcProb-DN-DirichletSubProb-RitzForm}
\item[2)] $\displaystyle u_2^{[k]} = \operatorname*{arg\,min}_{u_2\in V_2} \frac12 b_2(u_2, u_2) - (f, u_2)_2 + b_1(u_1^{[k]}, R_1\gamma_0 u_2) - (f, R_1\gamma_0 u_2)_1 + (q,u_2)_{L^2(\Gamma)}$\hfill\refstepcounter{equation}\textup{(\theequation)}\label{IntfcProb-DN-NeumannSubProb-RitzForm}%
\item[3)] $\displaystyle u_\Gamma^{[k+1]} = \rho u_2^{[k]} + (1-\rho) u_\Gamma^{[k]}\ \ \ \text{on}\ \Gamma$\hfill\refstepcounter{equation}\textup{(\theequation)}\label{IntfcProb-DN-Update}
\vspace{0.4em}
\end{itemize}
until certain stopping criteria are met \cite{toselli2004domain}. It is noteworthy that in contrast to the strong form \eqref{IntfcProb-StrongForm}, the jump conditions are enforced without explicitly computing and exchanging the interface flux of local solutions \cite{sun2022domain}. 

\begin{remark}
Although the solution $u(x)$ of interface problem \eqref{IntfcProb-StrongForm} lies in the space $H^{1+\epsilon}(\Omega)$ with $\epsilon>0$ possibly close to zero \cite{mercier2003minimal}, the decomposed solutions $u_1^{[k]}$, $u_2^{[k]}$ of subproblems \eqref{IntfcProb-DN-DirichletSubProb-RitzForm}, \eqref{IntfcProb-DN-NeumannSubProb-RitzForm} are regular under mild geometric assumptions on the boundary \cite{evans2010partial,gilbarg1977elliptic}. More precisely, it is assumed that the geometries of subdomains $\Omega_1$, $\Omega_2$ are smooth enough to ensure the regularity $u_1^{[k]}\in H^2(\Omega_1)$, $u_2^{[k]}\in H^2(\Omega_2)$ of the decomposed solutions for error estimates in \autoref{Section-Method}.
\end{remark}


\subsection{Related Work}\label{Section-RelatedWork}

Traditional numerical methods for solving the elliptic interface problems \eqref{IntfcProb-StrongForm} can be roughly categorized into two mainstreams by using either an interface-fitted or -unfitted mesh in the discretization of the domain. Provided a mesh that aligns exactly with the interface, the interface jump conditions can be naturally absorbed into a finite element formulation \cite{brenner2008mathematical}, achieving high accuracy with nearly optimal error bounds \cite{chen1998finite}. However, a major drawback is the requirement of an interface-fitted mesh generator, which could be time consuming and technically involved for complicated interface geometries in two and higher dimensions. As such, the second type of methods use interface-unfitted meshes (\textit{e.g.}, the Cartesian mesh) and have attracted enormous attention due to the easiness of mesh generation. A rich literature can be found in this direction including, but not limited to immersed interface methods \cite{leveque1994immersed}, extended finite element methods \cite{fries2010extended}, and many others. However, there can be significant difficulties in enforcing jump conditions across the complicated interfaces \cite{chen2017interface}, as well as a rigorous error analysis that is typically an ad hoc.

On the other hand, the Dirichlet-Neumann algorithm (\ref{IntfcProb-DN-DirichletSubProb-RitzForm}, \ref{IntfcProb-DN-NeumannSubProb-RitzForm}, \ref{IntfcProb-DN-Update}) illustrated in section \ref{Section-IntfcProb} is essentially a continuous method for solving the elliptic interface problem \eqref{IntfcProb-StrongForm}, thereby making it feasible to integrate with techniques from the deep learning community \cite{goodfellow2016deep,karniadakis2021physics}. Consequently, one of the most straightforward learning approach \cite{li2020deep} is based on the Dirichlet-Neumann algorithm written in terms of differential operators \cite{toselli2004domain}, namely, for $k\geq 1$,
\begingroup
\renewcommand*{\arraystretch}{1.3}
\begin{equation}
\begin{array}{cl}
\textnormal{Dirichlet subproblem} \!\!&\!\! \left\{
\begin{array}{cl}
-\nabla \cdot ( c_1 \nabla u_1^{[k]} ) + u_1^{[k]} = f & \text{in}\ \Omega_1,\\
u_1^{[k]} = 0 &  \text{on}\ \partial \Omega_1\setminus\Gamma, \\
\!\! u_1^{[k]} = \rho u_2^{[k-1]} + (1-\rho) u_1^{[k-1]} & \text{on}\ \Gamma,
\end{array}\right.
\\ 
\textnormal{Neumann subproblem} \!\!&\!\! \left\{
\begin{array}{cl}
-\nabla \cdot ( c_2 \nabla u_2^{[k+1]} ) + u_2^{[k+1]} = f & \text{in}\ \Omega_2,\\
u_2^{[k+1]} = 0 & \text{on}\ \partial \Omega_2\setminus\Gamma, \\
\!\! c_2 \nabla u_2^{[k+1]} \cdot \bm{n}_2 = -q - c_1 \nabla u_1^{[k]}\cdot\bm{n}_1 & \text{on}\ \Gamma,
\end{array}\right.
\end{array}
\label{IntfcProb-DN-StrongForm}
\end{equation}
\endgroup
where all the decomposed subproblems in \eqref{IntfcProb-DN-StrongForm} are approximately solved using the physics-informed neural networks \cite{raissi2019physics} or the deep Ritz method \cite{yu2018deep}. However, though the training loss tends to decrease as the iteration proceeds, the trained network solutions are prone to returning erroneous flux prediction along the interfaces \cite{dockhorn2019discussion,wang2021understanding,bajaj2021robust}, thereby hampering the convergence of outer iteration due to the incorrect transmission conditions being posed (see remark \ref{Remark-Err-Intfc} or \autoref{Section-Experiment}). It is also interesting to note that similar idea has been applied in areas of overlapping domain decomposition \cite{li2019d3m,li2020deep,sheng2022pfnn,li2022deep}, which achieves empirical success to a certain extent since the updated interface conditions are taken from the interior of neighbouring subdomains rather than its boundaries. Unfortunately, these overlapping methods face difficulties in handling problems with large jumps in the coefficient or discontinuous solutions.

Besides the meshless method mentioned above, recent advances in scientific machine learning \cite{baker2019workshop} have also led to an increasing interest in solving elliptic interface problems \eqref{IntfcProb-StrongForm} using various neural network models. For instance, by removing the inhomogeneous boundary conditions, a deep Ritz-type approach \cite{yu2018deep} is developed in \cite{wang2020mesh} for solving the interface problem using one single neural network. Based on the observation that the solution of interface problem is typically piecewise-continuous, an improved learning method is to employ a piecewise neural network to approximate the solution of \eqref{IntfcProb-StrongForm} on multiple subdomains, where the training loss function could be built from either the least square principle \cite{he2022mesh} or the variational formulation \cite{guo2021deep} of the underlying system \eqref{IntfcProb-StrongForm}. In addition, setting appropriate penalty weights among various loss terms is a crucial but painstaking process, which can be mitigated by assigning adaptive weights to different loss terms \cite{wu2022inn,wang2021understanding}. Note that in the case of homogeneous interface jump conditions, the continuity of averaged solution between neighbouring subdomains, as well as its first and higher-order derivatives \cite{jagtap2020extended,jagtap2020conservative,shukla2021parallel,wu2022improved,hu2021extended}, can be explicitly enforced through additional penalty terms posed on the interfaces. This family of methods are quite general and parallelizable, however, high-order derivatives at interface are meaningless for low-regularity solutions and no domain knowledge of the underlying equation is used during the training mode, leading to large errors near the interface that degrade the overall performance \cite{hu2022augmented}. Unfortunately, all these algorithms fail to address the issue of erroneous flux estimation at interface, which is the focus of this work. Another way of dealing with the jump conditions is to design specific network architectures that can capture the jump discontinuity across subdomain interfaces \cite{lai2022shallow,hu2022discontinuity,tseng2022cusp,hu2022hybrid}, while the theoretical work will be needed to obtain a deeper understanding.
 

\section{Method}\label{Section-Method}

In this section, we start with the convergence analysis of our proposed Dirichlet-Neumann learning algorithm for solving the elliptic interface problem \eqref{IntfcProb-StrongForm} in the ``continuum-level'', which offers a novel theoretical understanding of how to combine deep learning models \cite{raissi2019physics,yu2018deep,sirignano2018dgm} with ideas from domain decomposition methods \cite{quarteroni1999domain,toselli2004domain}. Then, the implementation details for realizing our proposed method is illustrated and summarized (see also \autoref{fig-computation-graph}).

Note that the convergence results of the iterative sequence of subdomain solutions (\ref{IntfcProb-DN-DirichletSubProb-RitzForm}, \ref{IntfcProb-DN-NeumannSubProb-RitzForm}, \ref{IntfcProb-DN-Update}) have been well-studied at the continuous level \cite{yang1997parallel}, it is then sufficient to show that these iterative solutions could be accurately approximated using modern deep learning techniques.

\subsection{Error Estimates}

To realize the iterative schemes (\ref{IntfcProb-DN-DirichletSubProb-RitzForm} -- \ref{IntfcProb-DN-Update}) using neural networks \cite{goodfellow2016deep}, the essential boundary conditions are treated in a ``soft'' manner by modifying the energy functional with boundary penalty terms \cite{yu2018deep,karniadakis2021physics,sirignano2018dgm}, namely,
\begin{equation}
	\hat{u}_1^{[k]} = \operatorname*{arg\,min}_{\hat{u}_1 \in H^1(\Omega_1) } \frac12  b_1(\hat{u}_1, \hat{u}_1) - (f, \hat{u}_1)_1 + \frac{\beta_D}{2} \left( \lVert \hat{u}_1 \rVert_{L^2(\partial\Omega_1\cap\partial\Omega)}^2 + \lVert \hat{u}_1 - u_\Gamma^{[k]} \rVert_{L^2(\Gamma)}^2 \right)
	\label{IntfcProb-DN-DirichletSubProb-RitzForm-BndryPenalty}	
\end{equation}
for the Dirichlet subproblem \eqref{IntfcProb-DN-DirichletSubProb-RitzForm} at the $k$-th outer iteration, where $\beta_D>0$ is a user-defined penalty coefficient \cite{wang2021understanding}. On the other hand, by extending the local solution $u_2\in V_2$ of decomposed problem \eqref{IntfcProb-DN-NeumannSubProb-RitzForm} to its neighboring subdomain (not relabelled),
\begin{equation}
	R_1\gamma_0 u_2(x) = u_2(x) \in V_1 
	\label{Extension-ContinuousLevel}
\end{equation}
the modified loss functional associated with the Neumann subproblem \eqref{IntfcProb-DN-NeumannSubProb-RitzForm} gives
\begin{equation}
\begin{split}
	\hat{u}_2^{[k]}  = \operatorname*{arg\,min}_{\hat{u}_2\in H^1(\Omega)} & \frac12 b_2(\hat{u}_2, \hat{u}_2) - (f, \hat{u}_2)_2 + b_1(\hat{u}_1^{[k]}, \hat{u}_2) - (f, \hat{u}_2)_1 + \frac{\beta_N}{2} \lVert \hat{u}_2 \rVert_{L^2(\partial\Omega)}^2 \\
	& + (q,\hat{u}_2)_{L^2(\Gamma)} 
\end{split}	
	\label{IntfcProb-DN-NeumannSubProb-RitzForm-BndryPenalty}	
\end{equation}
where $\beta_N>0$ denotes another penalty coefficient. It is noteworthy that the minimizer of functional \eqref{IntfcProb-DN-NeumannSubProb-RitzForm-BndryPenalty} is now defined globally over the whole domain, which differs from the traditional mesh-based treatment \cite{toselli2004domain}. 

Before introducing the neural network parametrization of the unknown solutions in \eqref{IntfcProb-DN-DirichletSubProb-RitzForm-BndryPenalty} and \eqref{IntfcProb-DN-NeumannSubProb-RitzForm-BndryPenalty}	, the error estimations induced by the relaxation from exact boundary conditions to a penalization-based approach are established in what follows.

\begin{theorem}
	Let $u_1^{[k]}$ be the solution of problem \eqref{IntfcProb-DN-DirichletSubProb-RitzForm} and $\hat{u}_1^{[k]}$ the solution of problem \eqref{IntfcProb-DN-DirichletSubProb-RitzForm-BndryPenalty}, then there holds
\begin{equation}
		\lVert \hat{u}_1^{[k]} - u_1^{[k]} \rVert_{H^1(\Omega_1)} \leq C(\Omega_1,u_1^{[k]}) \frac{c_1}{\beta_D} \sqrt{\frac{\hat{c}_1}{\check{c}_1}}
\end{equation}
where $\hat{c}_1=\max\{c_1,1\}$, $\check{c}_1=\min\{c_1,1\}$, and $C(\Omega_1,u_1^{[k]})$ represents a generic constant that depends on the subdomain $\Omega_1$ and the solution $u_1^{[k]}$ of Dirichlet subproblem \eqref{IntfcProb-DN-DirichletSubProb-RitzForm}.
\label{IntfcProb-DN-DirichletSubProb-ErrAnlys}
\end{theorem}
\begin{proof}
\textit{Step 1)} We first denote by $L_1(\hat{u}_1)$ the loss function of problem \eqref{IntfcProb-DN-DirichletSubProb-RitzForm-BndryPenalty}, that is, a functional on $H^1(\Omega_1)$
\begin{equation}
	\mathcal{L}_1(\hat{u}_1) = \frac12  b_1(\hat{u}_1, \hat{u}_1) - (f, \hat{u}_1)_1 + \frac{\beta_D}{2} \left( \lVert \hat{u}_1 \rVert_{L^2(\partial\Omega_1\cap\partial\Omega)}^2 + \lVert \hat{u}_1 - u_\Gamma^{[k]} \rVert_{L^2(\Gamma)}^2 \right),
	\label{DirichletSubProb-RitzForm-BndryPenalty-Loss}
\end{equation}
and then derive the optimality conditions that are satisfied by the unique global minimizer of \eqref{DirichletSubProb-RitzForm-BndryPenalty-Loss}. To be precise, the function $\hat{u}_1\in H^1(\Omega_1)$ is decomposed as a sum of two local functions, \textit{i.e.}, $ \hat{u}_1 = \hat{u}_1^{[k]} + g$ with $\hat{u}_1^{[k]}\in H^1(\Omega_1)$ satisfying 
\begingroup
\renewcommand*{\arraystretch}{1.3}
\begin{equation}
\left\{
\begin{array}{cl}
\displaystyle - \nabla \cdot ( c_1 \nabla \hat{u}_1^{[k]} ) + \hat{u}_1^{[k]} = f\ \ & \text{in}\ \Omega_1,\\
\displaystyle \hat{u}_1^{[k]} + c_1\beta_D^{-1} \nabla \hat{u}_1^{[k]} \cdot \bm{n}_1 =  0\ \ & \text{on}\ \partial\Omega_1\cap\partial\Omega, \\
\displaystyle \hat{u}_1^{[k]} + c_1\beta_D^{-1} \nabla \hat{u}_1^{[k]} \cdot \bm{n}_1 = u_\Gamma^{[k]} \ \ & \text{on}\ \Gamma, \\
\end{array}\right.
\label{DirichletSubProb-BndryPenalty-StrongForm}
\end{equation}
\endgroup
in the sense of distributions. Then, by applying the Green's formula to \eqref{DirichletSubProb-BndryPenalty-StrongForm}, a direct calculation of \eqref{DirichletSubProb-RitzForm-BndryPenalty-Loss} implies that\footnote{More details can be found in the technical \textbf{Appendix A}.} for any $\hat{u}_1\in H^1(\Omega_1)$,
\begin{equation*}
	\mathcal{L}_1(\hat{u}_1) = \mathcal{L}_1(\hat{u}_1^{[k]}) + \int_{\Omega_1} \left( \frac{c_1}{2}|\nabla g|^2 + \frac12|g|^2 \right) dx + \frac{\beta_D}{2} \int_{\partial\Omega_1} |g|^2 ds \geq \mathcal{L}_1(\hat{u}_1^{[k]}).
\end{equation*}
Or, equivalently, the unique weak solution of problem \eqref{DirichletSubProb-BndryPenalty-StrongForm} is the global minimizer of functional \eqref{DirichletSubProb-RitzForm-BndryPenalty-Loss}. Notably, when comparing \eqref{DirichletSubProb-BndryPenalty-StrongForm} with the original Dirichlet subproblem \eqref{IntfcProb-DN-DirichletSubProb-RitzForm} (written in terms of differential operators), \textit{i.e.}, 
\begingroup
\renewcommand*{\arraystretch}{1.3}
\begin{equation}
\left\{
\begin{array}{cl}
\displaystyle - \nabla \cdot ( c_1 \nabla u_1^{[k]} ) + u_1^{[k]} = f\ \ & \text{in}\ \Omega_1,\\
\displaystyle u_1^{[k]} =  0\ \ & \text{on}\ \partial\Omega_1\cap\partial\Omega, \\
\displaystyle u_1^{[k]} = u_\Gamma^{[k]} \ \ & \text{on}\ \Gamma, \\
\end{array}\right.
\label{IntfcProb-DN-DirichletSubProb-StrongForm}
\end{equation}
\endgroup
the Dirichlet boundary condition is modified to be of a Robin type owing to the boundary penalty treatment in \eqref{IntfcProb-DN-DirichletSubProb-RitzForm-BndryPenalty}. 

\textit{Step 2)} Now we are ready to quantitatively estimate the error induced from the ``soft'' boundary enforcement, that is, the distance between the weak solutions of \eqref{IntfcProb-DN-DirichletSubProb-StrongForm} and \eqref{DirichletSubProb-BndryPenalty-StrongForm}. To deal with the inhomogeneous boundary conditions in \eqref{IntfcProb-DN-DirichletSubProb-StrongForm} and \eqref{DirichletSubProb-BndryPenalty-StrongForm}, let us write $u_1^{[k]} = w_1 + g_1$ with an extension $g_1\in V_1$ of $u_\Gamma^{[k]}$ into $\Omega_1$ \cite{gilbarg1977elliptic}, namely,
\begingroup
\renewcommand*{\arraystretch}{1.3}
\begin{equation}
\left\{
\begin{array}{cl}
\displaystyle - \nabla \cdot ( c_1 \nabla w_1 ) + w_1 = f &\!\!  \text{in}\ \Omega_1,\\
\displaystyle w_1 = 0 &\!\!  \text{on}\ \partial\Omega_1, \\
\end{array}\right.
\ 
\left\{
\begin{array}{cl}
\displaystyle - \nabla \cdot ( c_1 \nabla g_1 ) + g_1 = 0 &\!\!  \text{in}\ \Omega_1,\\
\displaystyle g_1 =  0 &\!\!  \text{on}\ \partial\Omega_1\cap\partial\Omega, \\
\displaystyle g_1 = u_\Gamma^{[k]} &\!\!  \text{on}\ \Gamma, \\
\end{array}\right.
\label{DirichletSubProb-Extension-StrongForm}
\end{equation}
\endgroup
and $\hat{u}_1^{[k]} = \hat{w}_1 + \hat{g}_1$ with another extension $\hat{g}_1\in H^1(\Omega_1)$ of $u_\Gamma^{[k]}$ into $\Omega_1$, that is,
\begingroup
\renewcommand*{\arraystretch}{1.3}
\begin{equation}
\left\{
\begin{array}{cl}
\displaystyle - \nabla \cdot ( c_1 \nabla \hat{w}_1 ) + \hat{w}_1 = f &\!\!  \text{in}\ \Omega_1,\\
\displaystyle \hat{w}_1 + c_1\beta_D^{-1} \nabla \hat{w}_1 \cdot \bm{n}_1 =  0 &\!\!  \text{on}\ \partial\Omega_1,\\
\end{array}\right.
\ 
\left\{
\begin{array}{cl}
\displaystyle - \nabla \cdot ( c_1 \nabla \hat{g}_1 ) + \hat{g}_1 = 0 &\!\!  \text{in}\ \Omega_1,\\
\displaystyle \hat{g}_1 + c_1\beta_D^{-1} \nabla \hat{g}_1 \cdot \bm{n}_1 =  0 &\!\! \text{on}\ \partial\Omega_1\cap\partial\Omega, \\
\displaystyle \hat{g}_1 + c_1\beta_D^{-1} \nabla \hat{g}_1 \cdot \bm{n}_1 = u_\Gamma^{[k]} &\!\!  \text{on}\ \Gamma, \\
\end{array}\right.
\label{DirichletSubProb-BndryPenalty-Extension-StrongForm}
\end{equation}
\endgroup
then it immediately follows from the triangle inequality that
\begin{equation*}
	\lVert \hat{u}_1^{[k]} - u_1^{[k]} \rVert_{H^1(\Omega_1)} = \lVert (\hat{w}_1 + \hat{g}_1) - (w_1+g_1) \rVert_{H^1(\Omega_1)} \leq \lVert \hat{w}_1 - w_1 \rVert_{H^1(\Omega_1)} +  \lVert \hat{g}_1 - g_1 \rVert_{H^1(\Omega_1)}.
\end{equation*}

\textit{Step 3)} Based on the variational form, the extension function $g_1$ in \eqref{DirichletSubProb-Extension-StrongForm} satisfies
\begin{equation*}
	b_1(g_1,\hat{g}_1) \!=\! \int_{\Omega_1}  ( c_1 \nabla g_1 \cdot \nabla \hat{g}_1 + g_1\hat{g}_1 ) \, dx \!=\! \int_{\partial\Omega_1} (c_1 \nabla g_1 \cdot \bm{n_1}) \hat{g}_1 \, ds \!=\! (c_1 \nabla g_1 \cdot \bm{n_1}, \hat{g}_1)_{L^2(\partial\Omega_1)} 
\end{equation*}
with $\hat{g}_1\in H^1(\Omega_1)$ being used as the test function, while the extension function $\hat{g}_1$ of \eqref{DirichletSubProb-BndryPenalty-Extension-StrongForm} is the minimizer of energy functional\footnote{Here and in what follows, we do not distinguish between the general function and the optimal solution for notational simplicity.}
\begingroup
\renewcommand*{\arraystretch}{2.}
\begin{equation*}
\begin{array}{cl}
\mathcal{F}_1 (\hat{g}_1)\!\!\!\!\! & = \displaystyle \frac12 b_1(\hat{g}_1,\hat{g}_1) + \frac{\beta_D}{2} \left( \lVert \hat{g}_1 \rVert^2_{L^2(\partial\Omega_1\cap\partial\Omega)} + ( \hat{g}_1-2u_\Gamma^{[k]}, \hat{g}_1 )_{L^2(\Gamma)} \right) - b_1(g_1,\hat{g}_1) \\
& \displaystyle \ \ \ +\ (c_1 \nabla g_1 \cdot \bm{n_1}, \hat{g}_1)_{L^2(\partial\Omega_1)} \\
& = \displaystyle \frac12 b_1(\hat{g}_1 - g_1,\hat{g}_1 - g_1) + \frac{\beta_D}{2} \Big( \lVert \hat{g}_1 + c_1\beta_D^{-1} \nabla g_1\cdot \bm{n}_1 \rVert^2_{L^2(\partial\Omega_1\cap\partial\Omega)} \\
& \displaystyle \ \ \ +\ \lVert \hat{g}_1 + c_1\beta_D^{-1} \nabla g_1 \cdot \bm{n}_1 - u_\Gamma^{[k]} \rVert^2_{L^2(\Gamma)} \Big) - \frac12 b_1( g_1, g_1 ) \\
& \ \ \ \displaystyle -\ \frac{\beta_D}{2} \left( \lVert c_1\beta_D^{-1} \nabla g_1 \cdot \bm{n}_1 \rVert^2_{L^2(\partial\Omega_1\cap\partial\Omega)} + \lVert c_1\beta_D^{-1}\nabla g_1 \cdot \bm{n}_1 - u_\Gamma^{[k]} \rVert^2_{L^2(\Gamma)} \right)
\end{array}
\end{equation*}
\endgroup
and from which we can conclude that the function $\hat{g}_1$ is also the minimizer of 
\begin{equation*}
\begin{split}
\mathcal{G}_1(\hat{g}_1) & =\displaystyle \frac12 b_1(\hat{g}_1 - g_1,\hat{g}_1 - g_1) + \frac{\beta_D}{2} \Big( \lVert \hat{g}_1 + c_1\beta_D^{-1} \nabla g_1\cdot \bm{n}_1 \rVert^2_{L^2(\partial\Omega_1\cap\partial\Omega)} \\
& \displaystyle \ \ \ + \lVert \hat{g}_1 + c_1\beta_D^{-1} \nabla g_1 \cdot \bm{n}_1 - u_\Gamma^{[k]} \rVert^2_{L^2(\Gamma)} \Big).
\end{split}
\end{equation*}
On the one hand, by defining $\check{c}_1 = \min\{c_1,1\}$, it is obvious that 
\begin{equation}
	\mathcal{G}_1(\hat{g}_1) \geq \int_{\Omega_1} \left( \frac{c_1}{2}|\nabla (\hat{g}_1 - g_1) |^2 + \frac12 | \hat{g}_1 - g_1 |^2 \right)dx \geq \frac{\check{c}_1}{2} \lVert \hat{g}_1 - g_1 \rVert_{H^1(\Omega_1)}^2
	\label{DirichletSubProb-temp1}
\end{equation}
On the other hand, due to the fact that $g_1\in H^2(\Omega_1)$ under mild assumptions \cite{evans2010partial}, we have by using the trace theorem \cite{lions2012non} that $(\nabla g_1 \cdot \bm{n}_1)|_{\partial \Omega_1}\in H^{\frac12}(\partial \Omega_1)$ and therefore there exists a function $\phi\in H^1(\Omega_1)$ such that $\phi |_{\partial \Omega_1} =  - (\nabla g_1 \cdot \bm{n}_1) |_{\partial \Omega_1}$. 

Then, by choosing $\bar{g} = c_1\beta_D^{-1}\phi + g_1$ and using the boundary conditions of $g_1$ \eqref{DirichletSubProb-Extension-StrongForm}, the optimality of $\hat{g}_1$ among all functions in $H^1(\Omega_1)$ implies that
\begin{equation*}
\begin{split}
	\mathcal{G}_1(\hat{g}_1) & \leq  \displaystyle \mathcal{G}_1( \bar{g} ) = c_1^2\beta_D^{-2} \int_{\Omega_1} b_1(\phi,\phi) + \frac{\beta_D}{2} \left( \lVert g_1 \rVert^2_{L^2(\partial\Omega_1\cap\partial\Omega)} + \lVert g_1 - u_\Gamma^{[k]} \rVert^2_{L^2(\Gamma)} \right) \\
	& \leq \displaystyle \frac{\hat{c}_1}{2} c_1^2\beta_D^{-2} \lVert \phi \rVert_{H^1(\Omega_1)}^2
\end{split}	
\end{equation*}
where $\hat{c}_1 = \max\{c_1,1\}$. As a direct result, we have by \eqref{DirichletSubProb-temp1} that
\begin{equation*}
	\lVert \hat{g}_1 - g_1 \rVert_{H^1(\Omega_1)} \leq \frac{c_1}{\beta_D}\sqrt{\frac{\hat{c}_1}{\check{c}_1}} \lVert \phi \rVert_{H^1(\Omega_1)}. 
\end{equation*}

\textit{Step 4)} It remains to prove that the solution $\hat{w}_1\in H^1(\Omega_1)$ of the Robin problem \eqref{DirichletSubProb-BndryPenalty-Extension-StrongForm} can converge to the solution $w_1\in H^1_0(\Omega_1)$ of the Dirichlet problem \eqref{DirichletSubProb-Extension-StrongForm} as $\beta_D\to\infty$ \cite{jiao2021error,duan2021analysis}. Similar as before, by employing $\hat{w}_1\in H^1(\Omega_1)$ as the test function in \eqref{DirichletSubProb-Extension-StrongForm} and resorting to variational form of \eqref{DirichletSubProb-BndryPenalty-Extension-StrongForm}, the function $\hat{w}_1\in H^1(\Omega_1)$ minimizes
\begingroup
\renewcommand*{\arraystretch}{2.}
\begin{equation*}
\begin{array}{cl}
\mathcal{I}_1 (\hat{w}_1)\!\!\!\!\! & = \displaystyle \frac12 b_1(\hat{w}_1,\hat{w}_1) - (f,\hat{w}_1)_1 + \frac{\beta_D}{2} \lVert \hat{w}_1 \rVert^2_{L^2(\partial\Omega_1)} - b_1(w_1,\hat{w}_1) + (f,\hat{w}_1) \\
& \displaystyle \ \ \ +\ (c_1 \nabla w_1 \cdot \bm{n_1}, \hat{w}_1)_{L^2(\partial\Omega_1)} \\
& = \displaystyle \frac12 b_1(\hat{w}_1 - w_1,\hat{w}_1 - w_1) + \frac{\beta_D}{2} \lVert \hat{w}_1 + c_1\beta_D^{-1} \nabla w_1\cdot \bm{n}_1 \rVert^2_{L^2(\partial\Omega_1)}  \\
& \displaystyle \ \ \ -\ \frac12 b_1( w_1, w_1 ) - \frac{\beta_D}{2} \lVert c_1\beta_D^{-1} \nabla w_1 \cdot \bm{n}_1 \rVert^2_{L^2(\partial\Omega_1)} 
\end{array}
\end{equation*}
\endgroup
and therefore is also the minimizer of the energy functional 
\begin{equation*}
	\mathcal{J}_1(\hat{w}_1) = \frac12 b_1(\hat{w}_1 - w_1,\hat{w}_1 - w_1) + \frac{\beta_D}{2} \lVert \hat{w}_1 + c_1\beta_D^{-1} \nabla w_1\cdot \bm{n}_1 \rVert^2_{L^2(\partial\Omega_1)}.
\end{equation*}

Note that $w_1\in H^2(\Omega_1)$ under mild assumptions \cite{evans2010partial}, the trace theorem implies that $(\nabla w_1 \cdot \bm{n}_1)|_{\partial \Omega_1}\in H^{\frac12}(\partial \Omega_1)$ and therefore there exists a function $\varphi\in H^1(\Omega_1)$ such that $\varphi |_{\partial \Omega_1} =  - (\nabla w_1 \cdot \bm{n}_1) |_{\partial \Omega_1}$ \cite{lions2012non}. As a consequence, by employing a particular function $\bar{w} = c_1\beta_D^{-1}\varphi + w_1\in H^1(\Omega_1)$ and using the boundary condition of $w_1$ in \eqref{DirichletSubProb-Extension-StrongForm}, 
\begin{equation*}
\begin{split}	
	\mathcal{J}_1(\hat{w}_1) & \displaystyle \leq \mathcal{J}_1( \bar{w} ) = c_1^2\beta_D^{-2} \int_{\Omega_1} \left( \frac{c_1}{2}|\nabla \varphi |^2 + \frac12 | \varphi |^2 \right)dx + \frac{\beta_D}{2} \lVert w_1 \rVert^2_{L^2(\partial\Omega_1)} \\
    & \displaystyle \leq \frac{\hat{c}_1}{2} c_1^2\beta_D^{-2} \lVert \varphi \rVert_{H^1(\Omega_1)}^2.
\end{split}		
\end{equation*}
On the other hand, it is obvious that
\begin{equation*}
\begin{split}	
	\mathcal{J}_1(\hat{w}_1) & \displaystyle \geq \frac12 b_1(\hat{w}_1 - w_1,\hat{w}_1 - w_1) = \int_{\Omega_1} \left( \frac{c_1}{2}|\nabla (\hat{w}_1 - w_1) |^2 + \frac12 | \hat{w}_1 - w_1 |^2 \right)dx \\
	& \displaystyle \geq \frac{\check{c}_1}{2} \lVert \hat{w}_1 - w_1 \rVert_{H^1(\Omega_1)}^2
\end{split}		
\end{equation*}
which leads to the error estimation that completes the proof
\begin{equation*}
	\lVert \hat{w}_1 - w_1 \rVert_{H^1(\Omega_1)} \leq \frac{c_1}{\beta_D} \sqrt{\frac{\hat{c}_1}{\check{c}_1}} \lVert \varphi \rVert_{H^1(\Omega_1)}.
\end{equation*}
\end{proof}

\begin{theorem}
	Assume that $\beta_D\to\infty$ in \eqref{IntfcProb-DN-DirichletSubProb-RitzForm-BndryPenalty} (or \textnormal{\textbf{Theorem \ref{IntfcProb-DN-DirichletSubProb-ErrAnlys}}}), let $u_2^{[k]}$ be the solution of problem \eqref{IntfcProb-DN-NeumannSubProb-RitzForm} and $\hat{u}_2^{[k]}$ the solution of problem \eqref{IntfcProb-DN-NeumannSubProb-RitzForm-BndryPenalty}, then there holds\footnote{The minimizer $\hat{u}_2^{[k]}\in H^1(\Omega_1)$ of \eqref{IntfcProb-DN-NeumannSubProb-RitzForm-BndryPenalty} is restricted on the subdomain $\Omega_2$ (denoted by $\hat{u}_2^{[k]} |_{\Omega_2}$).} 
\begin{equation}
	\lVert \hat{u}_2^{[k]} |_{\Omega_2} - u_2^{[k]} \rVert_{H^1(\Omega_2)} \leq C(\Omega_2,u_2^{[k]}) \frac{c_2}{\beta_N}\sqrt{\frac{\hat{c}_2}{\check{c}_2}}
\end{equation}
where $\hat{c}_2=\max\{c_2,1\}$, $\check{c}_2=\min\{c_2,1\}$, and $C(\Omega_2,u_2^{[k]})$ represents a generic constant that depends on the subdomain $\Omega_2$ and the solution $u_2^{[k]}$ of Neumann subproblem \eqref{IntfcProb-DN-NeumannSubProb-RitzForm}.
\label{IntfcProb-DN-NeumannSubProb-ErrAnlys}
\end{theorem}
\begin{proof}
\textit{Step 1)} We first denote by $L_2(\hat{u}_2)$ the loss function of problem \eqref{IntfcProb-DN-NeumannSubProb-RitzForm-BndryPenalty}, 
\begin{equation}
	\mathcal{L}_2(\hat{u}_2) = \frac12 b_2(\hat{u}_2, \hat{u}_2) - (f, \hat{u}_2)_2 + b_1(\hat{u}_1^{[k]}, \hat{u}_2) - (f, \hat{u}_2)_1 + (q,\hat{u}_2)_{L^2(\Gamma)} + \frac{\beta_N}{2} \lVert \hat{u}_2 \rVert_{L^2(\partial\Omega)}^2,
	\label{NeumannSubProb-RitzForm-BndryPenalty-Loss}
\end{equation}
\textit{i.e.}, a functional on $H^1(\Omega)$, and then derive the optimality conditions that are satisfied by the global minimizer of \eqref{NeumannSubProb-RitzForm-BndryPenalty-Loss}. It's of particular noteworthy that the function $\hat{u}_2\in H^1(\Omega)$ is defined over the entire domain, which greatly differs from the standard Neumann subproblem \eqref{IntfcProb-DN-NeumannSubProb-RitzForm} that only depends on the subdomain $\Omega_2$. 

As such, we decompose $\hat{u}_2\in H^1(\Omega)$ as a sum of two global functions, namely, $ \hat{u}_2 = \hat{u}_2^{[k]} + g$, where the restriction of $\hat{u}_2^{[k]}\in H^1(\Omega)$ on subdomain $\Omega_2$ (not relabelled) is required to satisfy the equations
\begingroup
\renewcommand*{\arraystretch}{1.3}
\begin{equation}
\left\{
\begin{array}{cl}
\displaystyle - \nabla \cdot ( c_2 \nabla \hat{u}_2^{[k]} ) + \hat{u}_2^{[k]} = f\ \ & \text{in}\ \Omega_2,\\
\displaystyle \hat{u}_2^{[k]} + c_2\beta_N^{-1} \nabla \hat{u}_2^{[k]} \cdot \bm{n}_2 =  0\ \ & \text{on}\ \partial\Omega_2\cap\partial\Omega, \\
\displaystyle c_2 \nabla \hat{u}_2^{[k]} \cdot \bm{n}_2 = - q - c_1 \nabla \hat{u}_1^{[k]} \cdot \bm{n}_1 \ \ & \text{on}\ \Gamma, \\
\end{array}\right.
\label{NeumannSubProb-BndryPenalty-StrongForm}
\end{equation}
\endgroup
in the sense of distributions. Furthermore, the extension of function $\hat{u}_2^{[k]} |_{\Omega_2}$ to the other subdomain $\Omega_1$ is required to be weakly differentiable and to satisfy the Robin boundary condition
\begin{equation}
	\hat{u}_2^{[k]} + c_1\beta_N^{-1} \nabla \hat{u}_1^{[k]} \cdot \bm{n}_1 =  0\ \ \ \text{on}\ \partial\Omega_1\cap\partial\Omega, 
	\label{NeumannSubProb-BndryPenalty-StrongForm-Extension}
\end{equation}
in the weak sense. Then, by applying the Green's formula to \eqref{DirichletSubProb-BndryPenalty-StrongForm}, \eqref{NeumannSubProb-BndryPenalty-StrongForm} and using the jump condition \eqref{IntfcProb-StrongForm}, it can be deduced directly from \eqref{NeumannSubProb-RitzForm-BndryPenalty-Loss} that\footnote{More details can be found in the technical \textbf{Appendix B}.} for any $\hat{u}_2\in H^1(\Omega)$,
\begin{equation*}
	\mathcal{L}_2(\hat{u}_2) = \mathcal{L}_2(\hat{u}_2^{[k]}) + \int_{\Omega_2} \left( \frac{c_2}{2} |\nabla g|^2 + \frac12 |g|^2 \right) dx + \frac{\beta_N}{2} \int_{\partial\Omega} |g|^2ds \geq \mathcal{L}_2(\hat{u}_2^{[k]}),
\end{equation*}
namely, the global minimizer of \eqref{NeumannSubProb-RitzForm-BndryPenalty-Loss} can be characterized by the function $\hat{u}_2^{[k]}\in H^1(\Omega)$ that satisfies \eqref{NeumannSubProb-BndryPenalty-StrongForm} and \eqref{NeumannSubProb-BndryPenalty-StrongForm-Extension}. It is noted that only the restricted solution $\hat{u}_2^{[k]} |_{\Omega_2}\in H^1(\Omega_2)$, or equivalently, the weak solution of subproblem \eqref{NeumannSubProb-BndryPenalty-StrongForm} is of interest for error estimation, which is still denoted by $\hat{u}_2^{[k]}$ for short in the remaining of this proof. It is also noteworthy that when compared to the original Neumann subproblem \eqref{IntfcProb-DN-NeumannSubProb-RitzForm} (written in terms of differential operators), that is,
\begingroup
\renewcommand*{\arraystretch}{1.3}
\begin{equation}
\left\{
\begin{array}{cl}
\displaystyle - \nabla \cdot ( c_2 \nabla u_2^{[k]} ) + u_2^{[k]} = f\ \ & \text{in}\ \Omega_2,\\
\displaystyle u_2^{[k]} =  0\ \ & \text{on}\ \partial\Omega_2\cap\partial\Omega, \\
\displaystyle c_2 \nabla u_2^{[k]} \cdot \bm{n}_2 = - q - c_1 \nabla u_1^{[k]} \cdot \bm{n}_1 \ \ & \text{on}\ \Gamma, \\
\end{array}\right.
\label{IntfcProb-DN-NeumannSubProb-StrongForm}
\end{equation}
\endgroup
a Robin boundary condition is imposed on $\partial\Omega_2\cap\partial\Omega$ instead of the Dirichlet type, while the interface jump condition can be maintained exactly as long as the penalty coefficient $\beta_D$ in \eqref{IntfcProb-DN-DirichletSubProb-RitzForm} goes to infinity (see Theorem \ref{IntfcProb-DN-DirichletSubProb-ErrAnlys}). 

\textit{Step 2)} To simplify the error analysis, it is assumed that the penalty coefficient $\beta_D\to\infty$ in \eqref{IntfcProb-DN-DirichletSubProb-RitzForm-BndryPenalty}, \textit{i.e.}, $\hat{u}_1^{[k]} = u_1^{[k]}$ in \eqref{NeumannSubProb-BndryPenalty-StrongForm} and \eqref{IntfcProb-DN-NeumannSubProb-StrongForm}. With the solution $\hat{u}_2^{[k]}\in H^1(\Omega_2)$ being used as the test function, integration by parts for \eqref{IntfcProb-DN-NeumannSubProb-StrongForm} implies that
\begin{equation*}
	b_2( u_2^{[k]}, \hat{u}_2^{[k]} ) + ( q + c_1\nabla u_1^{[k]} \cdot \bm{n}_1, \hat{u}_2^{[k]} )_{L^2(\Gamma)}  - ( c_2\nabla u_2^{[k]} \cdot \bm{n}_2, \hat{u}_2^{[k]} )_{L^2(\partial\Omega_2\cap\partial\Omega)} = (f, \hat{u}_2^{[k]})_2
\end{equation*}
which can be employed to reformulate the energy functional\footnote{We do not distinguish between the general function and the optimal solution for simplicity.} of subproblem \eqref{NeumannSubProb-BndryPenalty-StrongForm} as
\begingroup
\renewcommand*{\arraystretch}{2.}
\begin{equation*}
\begin{array}{cl}
\mathcal{F}_2 (\hat{u}_2^{[k]})\!\!\!\!\! & = \displaystyle \frac12 b_2(\hat{u}_2^{[k]},\hat{u}_2^{[k]}) - (f, \hat{u}_2^{[k]})_2 + \frac{\beta_N}{2} \lVert \hat{u}_2^{[k]} \rVert^2_{L^2(\partial\Omega_2\cap\partial\Omega)} \\
& \ \ \ +\ ( q + c_1 \nabla \hat{u}_1^{[k]} \cdot \bm{n}_1, \hat{u}_2^{[k]} )_{L^2(\Gamma)} \\
& = \displaystyle  \frac12 b_2(\hat{u}_2^{[k]},\hat{u}_2^{[k]}) - b_2( u_2^{[k]}, \hat{u}_2^{[k]}) + \frac{\beta_N}{2} \lVert \hat{u}_2^{[k]} \rVert^2_{L^2(\partial\Omega_1\cap\partial\Omega)} \\
& \ \ \ +\ ( c_2\nabla u_2^{[k]} \cdot \bm{n}_2, \hat{u}_2^{[k]} )_{L^2(\partial\Omega_2\cap\partial\Omega)} \\
& = \displaystyle \frac12 b_2( \hat{u}_2^{[k]} - u_2^{[k]}, \hat{u}_2^{[k]} - u_2^{[k]}) + \frac{\beta_N}{2} \lVert \hat{u}_2^{[k]} + c_2\beta_N^{-1} \nabla u_2^{[k]} \cdot \bm{n}_2 \rVert^2_{L^2(\partial\Omega_2\cap\partial\Omega)} \\
& \ \ \ \displaystyle - \frac12 b_2( u_2^{[k]}, u_2^{[k]} ) - \frac{\beta_N}{2} \lVert c_2\beta_N^{-1} \nabla u_2^{[k]} \cdot \bm{n}_2 \rVert^2_{L^2(\partial\Omega_2\cap\partial\Omega)}.
\end{array}
\end{equation*}
\endgroup
As a result, we can conclude that the weak solution $\hat{u}_2^{[k]}\in H^1(\Omega_2)$ of \eqref{NeumannSubProb-BndryPenalty-StrongForm} is also the minimizer of functional
\begin{equation*}
	\mathcal{G}_2(\hat{u}_2^{[k]}) = \frac12 b_2( \hat{u}_2^{[k]} - u_2^{[k]}, \hat{u}_2^{[k]} - u_2^{[k]}) + \frac{\beta_N}{2} \lVert \hat{u}_2^{[k]} + c_2\beta_N^{-1} \nabla u_2^{[k]} \cdot \bm{n}_2 \rVert^2_{L^2(\partial\Omega_2\cap\partial\Omega)}
\end{equation*}
Clearly, be defining $\check{c}_2 = \min\{c_2,1\}$, it is obvious that
\begin{equation}
\begin{split}
	\mathcal{G}_2(\hat{u}_2^{[k]}) & \displaystyle \geq \frac12 b_2( \hat{u}_2^{[k]} - u_2^{[k]}, \hat{u}_2^{[k]} - u_2^{[k]}) \\
	& \displaystyle = \int_{\Omega_2} \left( \frac{c_2}{2} |\nabla (\hat{u}_2^{[k]} - u_2^{[k]}) |^2 + \frac12 |\hat{u}_2^{[k]} - u_2^{[k]}|^2  \right)dx \\
	& \displaystyle \geq \frac{\check{c}_2}{2} \lVert \hat{u}_2^{[k]} - u_2^{[k]} \rVert_{H^1(\Omega_2)}^2.
\end{split}	
	\label{NeumannSubProb-temp1}
\end{equation}
On the other hand, note that $u_2^{[k]}\in H^2(\Omega_2)$ under mild assumptions \cite{evans2010partial,gilbarg1977elliptic}, we have by using the trace theorem \cite{lions2012non} that $(\nabla u_2^{[k]} \cdot \bm{n}_2)|_{\partial \Omega_2}\in H^{\frac12}(\partial \Omega_2)$ and therefore there exists a function $\zeta\in H^1(\Omega_2)$ such that $\zeta |_{\partial \Omega_2} =  - (\nabla u_2^{[k]} \cdot \bm{n}_2) |_{\partial \Omega_2}$. In particular, the trace operator has a continuous linear right inverse and therefore there holds
\begin{equation*}
	\lVert \zeta \rVert_{H^1(\Omega_2)} \leq C \lVert \nabla u_2^{[k]} \cdot \bm{n}_2 \rVert_{H^{\frac12}(\partial\Omega_2)} \leq C\lVert u_2^{[k]} \rVert_{H^2(\Omega_2)} 
\end{equation*}
where $C>0$ is a generic constant that only depends on the subdomain $\Omega_1$ \cite{bramble1995interpolation}. 

Then by setting $ \bar{u}_2 = c_2\beta_N^{-1}\zeta + u_2^{[k]} $ and using the homogeneous boundary condition of  $u_2^{[k]}$ in \eqref{IntfcProb-DN-NeumannSubProb-StrongForm}, the optimality of $\hat{u}_2^{[k]}$ among all functions in $H^1(\Omega_2)$ implies 
\begin{equation*}
\begin{split}
	\mathcal{G}_2(\hat{u}_2^{[k]}) & \displaystyle \leq \mathcal{G}_2( \bar{u}_2 ) = c_2^2\beta_N^{-2} \int_{\Omega_1} \left( \frac{c_1}{2}|\nabla \zeta |^2 + \frac12 | \zeta |^2 \right)dx + \frac{\beta_N}{2}  \lVert u_2^{[k]} \rVert^2_{L^2(\partial\Omega_1\cap\partial\Omega)}  \\
	& \displaystyle \leq \frac{\hat{c}_2}{2} c_2^2\beta_N^{-2} \lVert \zeta \rVert_{H^1(\Omega_1)}^2
\end{split}	
\end{equation*}
where $\hat{c}_2 = \max\{c_2,1\}$. Consequently, we have by \eqref{NeumannSubProb-temp1} that
\begin{equation*}
	\lVert \hat{u}_2^{[k]} - u_2^{[k]} \rVert_{H^1(\Omega_1)} \leq \frac{c_2}{\beta_N}\sqrt{\frac{\hat{c}_2}{\check{c}_2}} \lVert \zeta \rVert_{H^1(\Omega_1)}.
\end{equation*}
\end{proof}

To put it differently, by sending the penalty coefficients $\beta_D\to\infty$ in \eqref{IntfcProb-DN-DirichletSubProb-RitzForm-BndryPenalty} and $\beta_N\to\infty$ in \eqref{IntfcProb-DN-NeumannSubProb-RitzForm-BndryPenalty}, the minimizers of our relaxed optimization problems (\ref{IntfcProb-DN-DirichletSubProb-RitzForm-BndryPenalty}, \ref{IntfcProb-DN-NeumannSubProb-RitzForm-BndryPenalty}) could converge to that of the classical Dirichlet-Neumann algorithm (\ref{IntfcProb-DN-DirichletSubProb-RitzForm}, \ref{IntfcProb-DN-NeumannSubProb-RitzForm}), which paves the way for constructing the Dirichlet-Neumann learning algorithm in the practical scenario. Moreover, our analysis \eqref{NeumannSubProb-BndryPenalty-StrongForm} also sheds light on the tuning of penalty coefficients, that is, $\beta_N$ should be increased as $c_2$ increases, otherwise the trained network solution would fail to capture the Dirichlet boundary condition on $\partial\Omega_2\cap\partial\Omega$.


\subsection{Dirichlet-Neumann Learning Algorithm}

Next, the unknown solutions in \eqref{IntfcProb-DN-DirichletSubProb-RitzForm-BndryPenalty} and \eqref{IntfcProb-DN-NeumannSubProb-RitzForm-BndryPenalty} are parametrized using artificial neural networks \cite{goodfellow2016deep}, that is,
\begin{equation*}
	\hat{u}_1^{[k]}(x) = \hat{u}_1(x;\theta_1^{[k]}) \ \ \ \text{and}\ \ \ \hat{u}_2^{[k]}(x) = \hat{u}_2(x;\theta_2^{[k]})
\end{equation*}
where $\theta_i^{[k]}$ denotes the collection of trainable parameters at the $k$-th outer iteration for $i=1$, 2. More specifically, the fully connected neural network \cite{goodfellow2016deep} or other kinds of architectures can be deployed to construct the solution ansatz (more details about the model setup are included in \textbf{Appendix C}). Moreover, thanks to the mesh-free characteristic of artificial neural networks, the extension operator \eqref{Extension-ContinuousLevel} can be realized in a very straightforward way, that is,
\begin{equation}
	R_1\gamma_0 \hat{u}_2(x,\theta_2) = \hat{u}_2(x,\theta_2)
	\label{Extension-DiscreteLevel}
\end{equation}
which is required to be differentiable within $\Omega_1$ and to satisfy the zero boundary values on $\partial\Omega_1\cap\partial\Omega$ through an additional penalty term in \eqref{IntfcProb-DN-NeumannSubProb-RitzForm-BndryPenalty}. One can also employ a piecewise neural network \cite{he2022mesh} to realize the extension operation \eqref{Extension-DiscreteLevel}.

Accordingly, to discretize the energy functionals \eqref{IntfcProb-DN-DirichletSubProb-RitzForm-BndryPenalty} and \eqref{IntfcProb-DN-NeumannSubProb-RitzForm-BndryPenalty}, the routine way of generating training points inside each subdomain and at its boundary is to use the Monte Carlo method or its variants \cite{metropolis1949monte}, namely,
\begin{equation*}
	X_{\Omega_i} = \big\{ x_n^{\Omega_i} \big\}_{n=1}^{N_{\Omega_i}},\ \ \ X_{D_i} = \big\{ x_n^{D_i} \big\}_{n=1}^{N_{D_i}},\ \ \ \text{and}\ \ \ X_{\Gamma} = \big\{ x_n^{\Gamma} \big\}_{n=1}^{N_{\Gamma}},
\end{equation*}
where $D_i := \partial\Omega_i\cap\partial\Omega$, $N_{\Omega_i}$, $N_{D_i}$ and $N_\Gamma$ denote the sample size of training datasets $X_{\Omega_i}$, $X_{D_i}$ and $X_\Gamma$, respectively. 

As a result, by defining the following empirical loss functions for $1\leq i, j \leq 2$, 
\begingroup
\renewcommand*{\arraystretch}{2.5}
\begin{equation*}
\begin{array}{c}
\displaystyle L_{\Omega_i} ( \hat{u}_i ) = \frac{1}{N_{\Omega_i}} \sum_{n=1}^{N_{\Omega_i}} \left( \frac{c_i}{2} | \nabla \hat{u}_i(x_n^{\Omega_i};\theta_i) |^2 - f(x_n^{\Omega_i}) \hat{u}_i(x_n^{\Omega_i};\theta_i)\right), \\
\displaystyle L_{D_i} ( \hat{u}_j ) = \frac{1}{N_{D_i}} \sum_{n=1}^{N_{D_i}} | \hat{u}_j(x_n^{D_i};\theta_j) |^2,\ L_{\Gamma_N} \big( \hat{u}_2 \big) = \frac{1}{N_\Gamma} \sum_{n=1}^{N_\Gamma}  q( x_n^\Gamma )\hat{u}_2(x_n^\Gamma;\theta_2),\\
\displaystyle L_{\Gamma_D} ( \hat{u}_1, u_\Gamma^{[k]} ) = \frac{1}{N_\Gamma} \sum_{n=1}^{N_\Gamma} | \hat{u}_1(x_n^\Gamma;\theta_1) - u_\Gamma^{[k]}( x_n^\Gamma ) |^2,
\end{array}
\end{equation*}
\vspace{-0.8cm}
\begin{equation*}
\begin{array}{c}
\displaystyle L_N ( \hat{u}_2, \hat{u}_1^{[k]} ) = \frac{1}{N_{\Omega_1}} \sum_{n=1}^{N_{\Omega_1}} \left( c_1 \nabla \hat{u}_1(x_n^{\Omega_1};\theta_1^{[k]}) \cdot \nabla \hat{u}_2(x_n^{\Omega_1};\theta_2) - f(x_n^{\Omega_1}) \hat{u}_2( x_n^{\Omega_1}; \theta_2 ) \right),
\end{array}
\end{equation*}
\endgroup
the learning task associated with the Dirichlet subproblem \eqref{IntfcProb-DN-DirichletSubProb-RitzForm-BndryPenalty} is now given by
\begin{equation}
	\theta_1^{[k]} = \operatorname*{arg\,min}_{\theta_1} L_{\Omega_1} ( \hat{u}_1 ) + \frac{\beta_D}{2} \left( L_{D_1} ( \hat{u}_1 ) + L_{\Gamma_D} ( \hat{u}_1, u_\Gamma^{[k]} ) \right),
	\label{IntfcProb-DN-DirichletSubProb-DeepRitzForm}
\end{equation}
while that of the Neumann subproblem \eqref{IntfcProb-DN-NeumannSubProb-RitzForm-BndryPenalty} takes on the form (referred to as \textit{Compensated Deep Ritz Method})
\begin{equation}
	\theta_2^{[k]} = \operatorname*{arg\,min}_{\theta_2} L_{\Omega_2} ( \hat{u}_2 ) + L_N ( \hat{u}_2, \hat{u}_1^{[k]} ) + L_{\Gamma_N} ( \hat{u}_2 ) + \frac{\beta_N}{2} \left( L_{D_1} ( \hat{u}_2 ) + L_{D_2}( \hat{u}_2 ) \right).
	\label{IntfcProb-DN-NeumannSubProb-DeepRitzForm}
\end{equation}
It is noteworthy that although the loss value of \eqref{IntfcProb-DN-DirichletSubProb-DeepRitzForm} continues to decrease as the training proceeds, the trained model is often observed to possess highly fluctuating errors at and near the boundaries \cite{wang2021understanding,dockhorn2019discussion,bajaj2021robust}, which eventually leads to a failure to approximate the interface flux and therefore hampers the subsequent amendment \eqref{IntfcProb-DN-NeumannSubProb-StrongForm} \cite{sun2022domain}. Fortunately, by employing our algorithm \eqref{IntfcProb-DN-NeumannSubProb-DeepRitzForm}, the Neumann subproblem can now be solved without explicitly enforcing the flux jump condition \eqref{IntfcProb-StrongForm}, thereby ensuring the convergence of outer iteration in the presence of inaccurate flux prediction. 

\begin{algorithm}[t!]
\caption{Dirichlet-Neumann Learning Algorithm}
\begin{algorithmic}
\STATE{\% \textit{Initialization} }
\STATE{-- specify the network architecture $\hat{u}_i(x;\theta_i)$ $(i=1$, 2) for each subproblem;}
\STATE{-- generate the Monte Carlo sampling points $X_\Gamma$, $X_{\Omega_i}$, and $X_{D_i}$ for $i=1$, 2; }
\STATE{\% \textit{Outer Iteration Loop} }
\STATE{Start with the initial guess $u_\Gamma^{[0]}$ of unknown solution values at interface $\Gamma$;}
\FOR{$k \gets 0$ to $K$ (maximum number of outer iterations)}
\WHILE{stopping criteria are not satisfied}
\STATE{\% \textit{Dirichlet Subproblem-Solving via Deep Ritz Method} }
\STATE{
\vspace{-.62cm}
\begingroup
\renewcommand*{\arraystretch}{1.1}
\begin{equation*}
\theta_1^{[k]} = \operatorname*{arg\,min}_{\theta_1} L_{\Omega_1}( \hat{u}_1 ) + \frac{\beta_D}{2} \left( L_{D_1}( \hat{u}_1 ) + L_{\Gamma}( \hat{u}_1, u_\Gamma^{[k]} ) \right)
\end{equation*}
\endgroup
\vspace{-0.45cm}
}
\STATE{\% \textit{Neumann Subproblem-Solving via Compensated Deep Ritz Method} }
\STATE{
\vspace{-.62cm}
\begingroup
\renewcommand*{\arraystretch}{1.1}
\begin{equation*}
\theta_2^{[k]} = \operatorname*{arg\,min}_{\theta_2} L_{\Omega_2}( \hat{u}_2 ) + L_N( \hat{u}_2, \hat{u}_1^{[k]} ) + L_{\Gamma_N}( \hat{u}_2 ) + \frac{\beta_N}{2} \left( L_{D_1}( \hat{u}_2 ) + L_{D_2}( \hat{u}_2 ) \right)
\end{equation*}
\endgroup
\vspace{-0.45cm}
}
\STATE{\% \textit{Update of Unknown Solution Values at Interface} }
\STATE{
\vspace{-.62cm}
\begingroup
\renewcommand*{\arraystretch}{1.1}
\begin{equation*}
u_\Gamma^{[k+1]}(x_n^\Gamma) = \rho \hat{u}_2(x_n^\Gamma;\theta_2^{[k]}) + (1-\rho) u_\Gamma^{[k]}(x_n^\Gamma), \ \ i=1,\cdots,N_\Gamma,
\end{equation*}
\endgroup
\vspace{-0.6cm}
}
\ENDWHILE
\ENDFOR
\end{algorithmic}
\label{Algorithm-DNLM}
\end{algorithm}

More specifically, the detailed iterative scheme is summarized in Algorithm \ref{Algorithm-DNLM}, where the stopping criteria can be constructed by measuring the difference in solutions between two consecutive iterations \cite{li2020deep,he2022mesh}. We also note that as an alternative to the deep Ritz method \eqref{IntfcProb-DN-DirichletSubProb-DeepRitzForm} \cite{yu2018deep}, the Dirichlet subproblem \eqref{IntfcProb-DN-DirichletSubProb-RitzForm-BndryPenalty} can be solved using the physics-informed neural networks (abbreviated as PINNs in what follows) \cite{raissi2019physics}, which is known to empirically work better for problems with sufficient smooth solutions. To be precise, by incorporating the residual of the equation \eqref{IntfcProb-StrongForm} into the loss function 
\begin{equation*}
	L^{\textnormal{PINNs}}_{\Omega_1} ( \hat{u}_1 ) = \frac{1}{N_{\Omega_1}} \sum_{n=1}^{N_{\Omega_1}} \left| - \nabla \cdot \big( c_1 \nabla \hat{u}_1(x_n^{\Omega_1};\theta_1) \big) + \hat{u}_1(x_n^{\Omega_1};\theta_1) - f(x_n^{\Omega_1})\right|^2
\end{equation*}
the learning task of Dirichlet subproblem \eqref{IntfcProb-DN-DirichletSubProb-RitzForm-BndryPenalty} can alternatively be formulated as
\begin{equation*}
	\theta_1^{[k]} = \operatorname*{arg\,min}_{\theta_1} L^{\textnormal{PINNs}}_{\Omega_1} ( \hat{u}_1) + \frac{\beta_D}{2} \left( L_{D_1} ( \hat{u}_1 ) + L_{\Gamma_D} ( \hat{u}_1 ) \right).
\end{equation*}
Similar in spirit, domain decomposition leads to simpler functions to be learned on each subdomain, therefore the solution of Neumann subproblem \eqref{IntfcProb-DN-NeumannSubProb-RitzForm} on $\Omega_2$ can be assumed smooth enough. This allows us to incorporate an additional loss term
\begin{equation}
	L^{\textnormal{PINNs}}_{\Omega_2} ( \hat{u}_2 ) = \frac{1}{N_{\Omega_2}} \sum_{n=1}^{N_{\Omega_2}} \left| - \nabla \cdot \big( c_2 \nabla \hat{u}_2(x_n^{\Omega_2};\theta_2) \big) + \hat{u}_2(x_n^{\Omega_2};\theta_2) - f(x_n^{\Omega_2})\right|^2
	\label{IntfcProb-DN-NeumannSubProb-Modified-TrainLoss}
\end{equation}
into the learning task of our extended Neumann subproblem \eqref{IntfcProb-DN-NeumannSubProb-DeepRitzForm}. To put it differently, \eqref{IntfcProb-DN-NeumannSubProb-Modified-TrainLoss} can be regarded as a regularization term that helps improve the approximation accuracy and prevent the training process \eqref{IntfcProb-DN-NeumannSubProb-DeepRitzForm} from getting trapped in trivial solutions especially when $c_2\gg 1$.

Moreover, our Dirichlet-Neumann learning algorithm can be easily modified to adapt to a parallel computing environment, which is based on the theoretical work \cite{yang1997parallel} (see also \autoref{fig-computation-graph}) and is left for future investigation.

\begin{remark}\label{Remark-Err-Intfc}
Note that when solving the Neumann subproblem, the flux transmission condition in \eqref{IntfcProb-DN-NeumannSubProb-StrongForm}, \textit{i.e.},
\begin{equation*}
	c_2 \nabla u_2^{[k]} \cdot \bm{n}_2 = - q - c_1 \nabla u_1^{[k]} \cdot \bm{n}_1 \ \ \ \textnormal{on}\ \Gamma
\end{equation*}
can be enforced using either our proposed algorithm or the DeepDDM scheme \cite{li2020deep}, where both the neural networks are trained to approximately satisfy this jump condition and eventually ended up with 
\begin{equation*}
	c_2 \nabla \hat{u}_2^{[k]} \cdot \bm{n}_2 \approx - q - c_1 \nabla \hat{u}_1^{[k]} \cdot \bm{n}_1 \ \ \ \textnormal{on}\ \Gamma.
\end{equation*}
As such, we immediately arrive at the error estimation
\begin{equation}
	\nabla \hat{u}_2^{[k]} - \nabla u_2^{[k]} \approx \frac{c_1}{c_2} \big( \nabla \hat{u}_1^{[k]} - \nabla u_1^{[k]} \big) \ \ \ \textnormal{on}\ \Gamma
	\label{Err-Intfc-DNLM-DeepDDM}
\end{equation}
which shows that the error incurred by an inaccurate right-hand-side term in \eqref{Err-Intfc-DNLM-DeepDDM} may propagate to nearby interior points during training, and therefore eventually leads to failure models of the trained network solution \cite{daw2022rethinking}. 

Unfortunately, the DeepDDM method \cite{li2020deep} and other related work \cite{he2022mesh,jagtap2020extended,jagtap2020conservative} rely on on a direct computation of the Dirichlet subproblem \eqref{IntfcProb-DN-DirichletSubProb-StrongForm} using neural networks, followed by a straightforward evaluation of the trained network solutions at interface to execute the subsequent operations. However, the approximation errors for Dirichlet subproblem are typically found to concentrate and fluctuate at and near the boundary \cite{dockhorn2019discussion,bajaj2021robust,wang2021understanding,sun2022domain}, hence the representation of $\nabla u_1^{[k]}$ using $\nabla \hat{u}_1^{[k]}$ along the interface may be of low accuracy. Consequently, the iterative solutions would fail to converge to the exact solution once the error $\nabla \hat{u}_1^{[k]} - \nabla u_1^{[k]}$ could not be eliminated by the factor $c_1/c_2$, \textit{e.g.}, $c_1 \approx c_2$ or $c_1 = c_2$ (see \autoref{Section-Experiment} for more experimental evidences).

On the contrary, our proposed scheme \eqref{IntfcProb-DN-NeumannSubProb-DeepRitzForm} is based on the variational formulation of Dirichlet subproblem \eqref{IntfcProb-DN-DirichletSubProb-StrongForm}, which has a satisfactory accuracy in representing $\nabla u_1^{[k]} |_\Gamma$ in the weak sense and therefore substantially differs from other related methods. In other words, the right-hand-side term of \eqref{Err-Intfc-DNLM-DeepDDM} is small whenever the coefficients are set to be $c_1\ll c_2$ or $c_1 \approx c_2$ (see \autoref{Section-Experiment} for more details).
\end{remark}



\section{Numerical Experiments}\label{Section-Experiment}

In this section, numerical experiments on a series of elliptic interface problems \eqref{IntfcProb-StrongForm} are carried out to demonstrate the effectiveness of our Dirichlet-Neumann learning algorithm (abbreviated as DNLA in what follows, with text in bracket indicating the type of deep learning solver adopted for solving the Dirichlet subproblem \eqref{IntfcProb-DN-DirichletSubProb-StrongForm}). In all experiments, the network architecture deployed for local problem is a fully connected neural network, while the design of more sophisticated structures \cite{hu2022augmented,hu2022discontinuity} is left for future investigation. Moreover, due to the smoothness requirement of extension operator \eqref{Extension-DiscreteLevel} (see also \eqref{Extension-ContinuousLevel}), the hyperbolic tangent activation function is used rather than the ReLU activation function  \cite{goodfellow2016deep}. 

As a comparison to our method, we also conduct experiments using the most straightforward learning algorithm, that is, employing the PINNs for solving all the decomposed subproblems \eqref{IntfcProb-DN-StrongForm} as proposed in \cite{li2020deep}, which is referred to as DeepDDM. For a fair comparison, the same relaxation parameter, \textit{i.e.}, $\rho$ in \eqref{IntfcProb-DN-StrongForm} and \eqref{IntfcProb-DN-Update}, is set for both methods, and the mean value and standard deviation of the relative $L^2$ error evaluated on $100\times100$ uniform grid points (testing dataset)
\begingroup
\renewcommand*{\arraystretch}{1.3}
\begin{equation*}
\epsilon = \frac{ \lVert u - \hat{u}^{[k]} \rVert_{L^2(\Omega)} }{ \lVert u \rVert_{L^2(\Omega)} }
\ \ \ 
\text{where}
\ \ \ 
\hat{u}^{[k]}(x;\theta) = \left\{
\begin{array}{cl}
\hat{u}_1(x;\theta_1^{[k]})\ \ & \text{if}\ x\in\Omega_1,\\
\hat{u}_2(x;\theta_2^{[k]})\ \ & \text{if}\ x\in\Omega_2,
\end{array}\right.
\end{equation*}
\endgroup
are reported over 5 independent runs along the outer iterations, followed by a comprehensive study that validates our discussion in remark \ref{Remark-Err-Intfc}. When the maximum number of outer iteration is reached or the stopping criteria are met, the resulting network solution is denoted by $\hat{u}(x;\theta)$ instead of $\hat{u}^{[K]}(x;\theta)$ or $\hat{u}(x;\theta^{[K]})$ for simplicity. 

More detailed experimental setups, \textit{e.g.}, the depth and width of the network solution, the training and testing datasets, are briefly summarized in \textbf{Appendix C}. Unless otherwise specified, we shall use the following hyperparameter configurations. The AdamW optimizer \cite{kingma2014adam} is adopted with an initial learning rate set to be $0.1$, which is then multiplied by $0.1$ when the training process reaches $60\%$ and $80\%$. In each outer iteration, the Dirichlet or Neumann subproblem-solving, namely, \eqref{IntfcProb-DN-DirichletSubProb-DeepRitzForm} or \eqref{IntfcProb-DN-NeumannSubProb-DeepRitzForm} is terminated after $3k$ and $1k$ epochs, respectively, and the model solution with the minimum training loss is saved to execute the subsequent operations. All experiments are implemented using PyTorch \cite{paszke2017automatic} on Nvidia GeForce RTX 3090 graphic cards. It is of particular noteworthy that the penalty coefficient $\beta_N$ should be increased as the coefficient $c_2$ increases (see \textbf{Appendix C}), which can be easily concluded from our analysis \eqref{NeumannSubProb-BndryPenalty-StrongForm} but has not been addressed in the literature.

\subsection{Circle Interface in Two Dimension}

\begin{figure}[t!]
\centering
\includegraphics[width=0.212\textwidth]{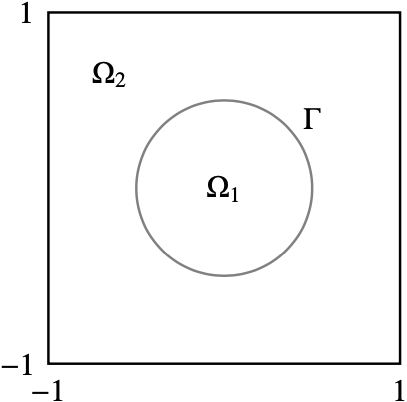}
\hspace{0.3cm}
\includegraphics[width=0.263\textwidth]{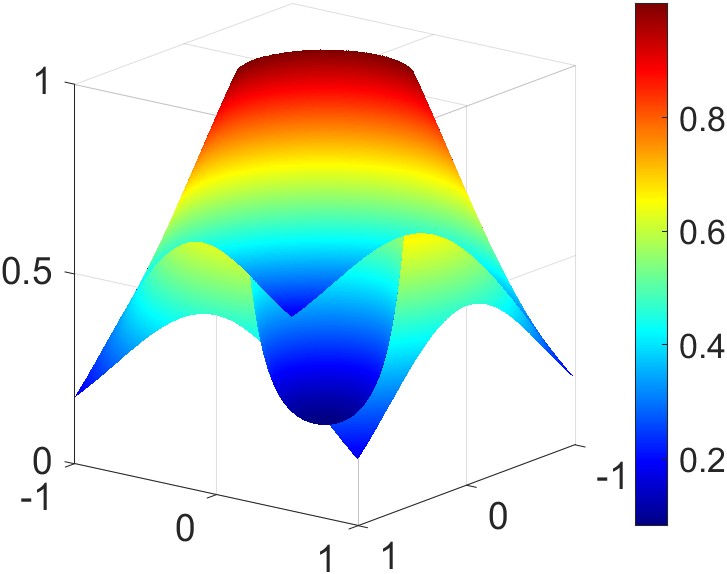}
\hspace{0.3cm}
\includegraphics[width=0.263\textwidth]{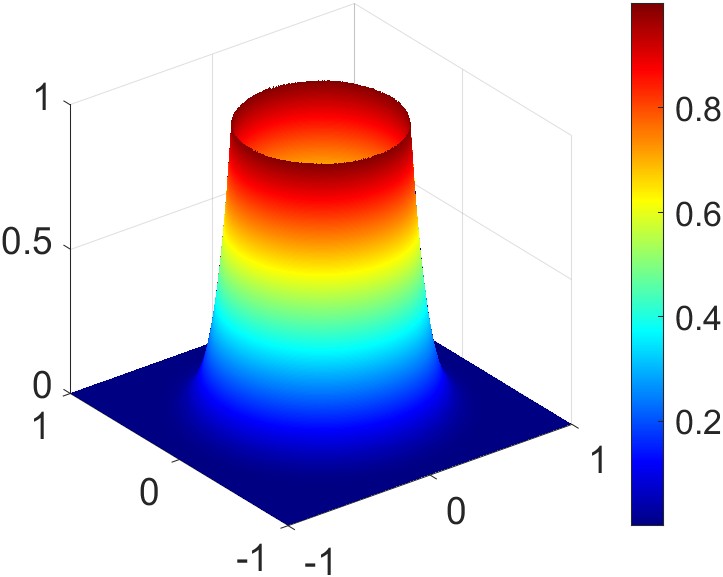}
\caption{Geometry of interface, exact solution with $(c_1,c_2)=(1,10^3)$ and $(1,1)$ for \eqref{NumXmp1-Circle}.}
\label{fig-Xmp1-exact-solution}
\vspace{-0.3cm}
\end{figure}

\begin{table}[t!]
\caption{Error profiles $ \hat{u}(x,y;\theta) - u(x,y) $ of different methods for example \eqref{NumXmp1-Circle} (see also \textbf{Appendix C}).}
\centering
\begin{tabular}{ c | c | c | c }
\toprule
Coefficients & DeepDDM & DNLA (PINNs) & DNLA (deep Ritz) \\ 
\midrule
\makecell{$(1, 10^3)$ \\ with \\ $\rho = 1$}  &
\begin{minipage}{.23\textwidth}
\centering
\includegraphics[width=0.9\linewidth]{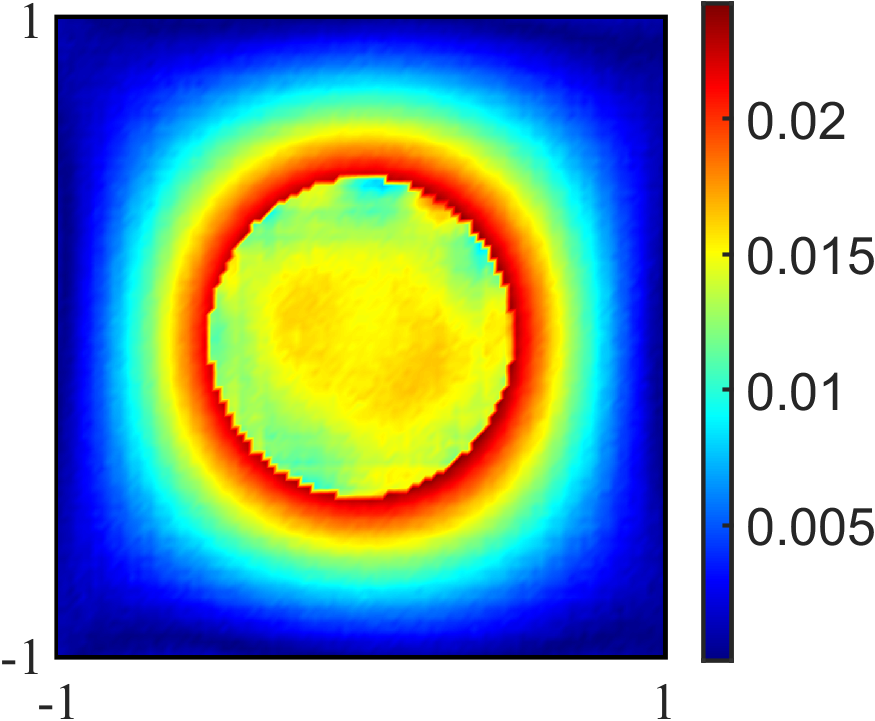}\vspace*{-0.1cm}
\centerline{\small{$K=2$}}
\end{minipage}
& 
\begin{minipage}{.23\textwidth}
\centering
\includegraphics[width=0.9\linewidth]{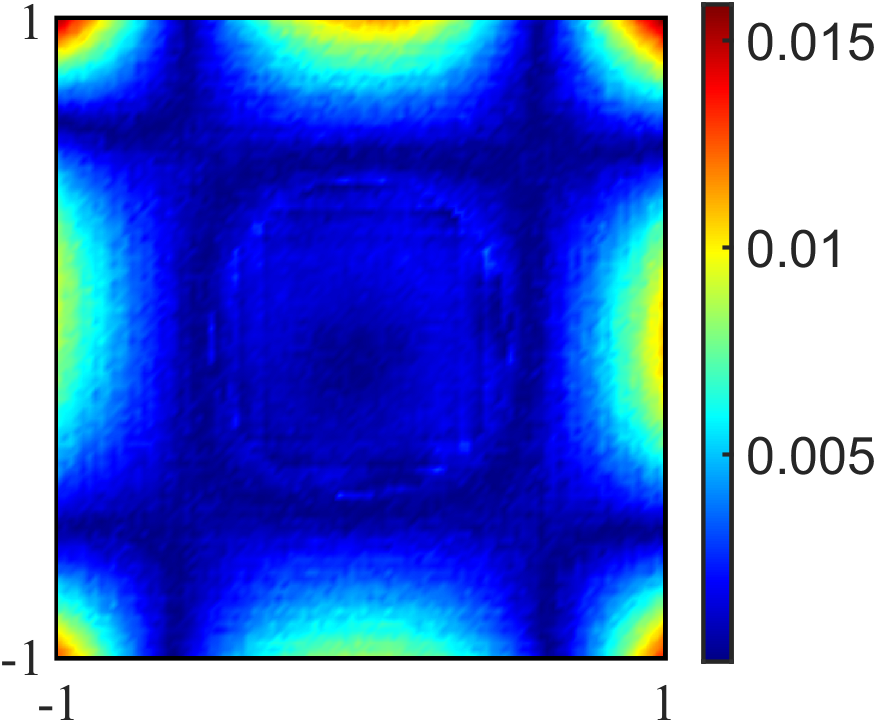}\vspace*{-0.1cm}
\centerline{\small{$K=2$}}
\end{minipage}
& 
\begin{minipage}{.23\textwidth}
\centering
\includegraphics[width=0.9\linewidth]{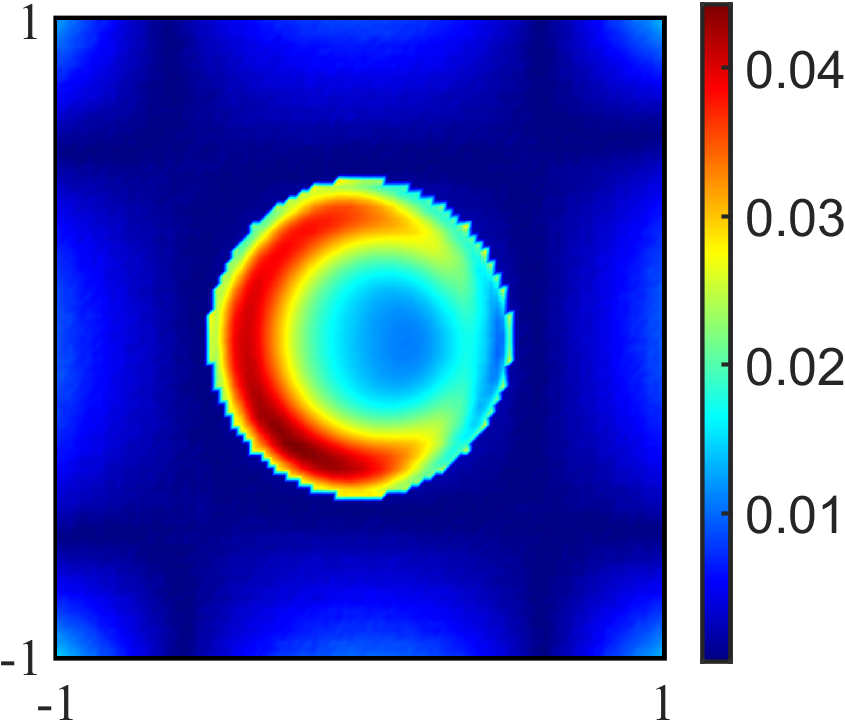}\vspace*{-0.1cm}
\centerline{\small{$K=2$}}
\end{minipage}
\\ 
\midrule
\makecell{$(1, 1)$ \\ with \\ $\rho = 0.5$} &
\begin{minipage}{.23\textwidth}
\centering
\includegraphics[width=0.88\linewidth]{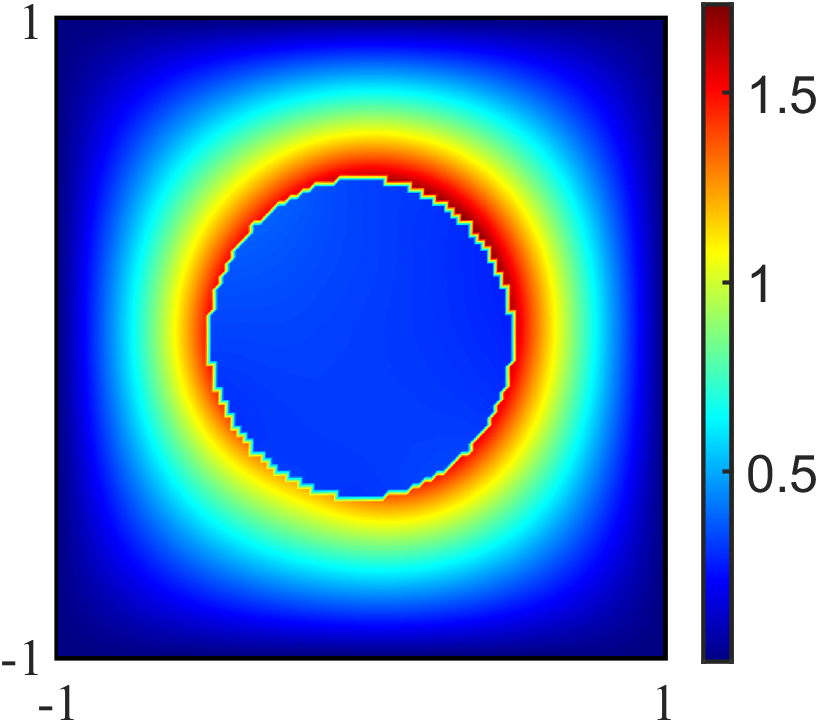}\vspace*{-0.1cm}
\centerline{\small{$K=15$}}
\end{minipage}
& 
\begin{minipage}{.23\textwidth}
\centering
\includegraphics[width=0.9\linewidth]{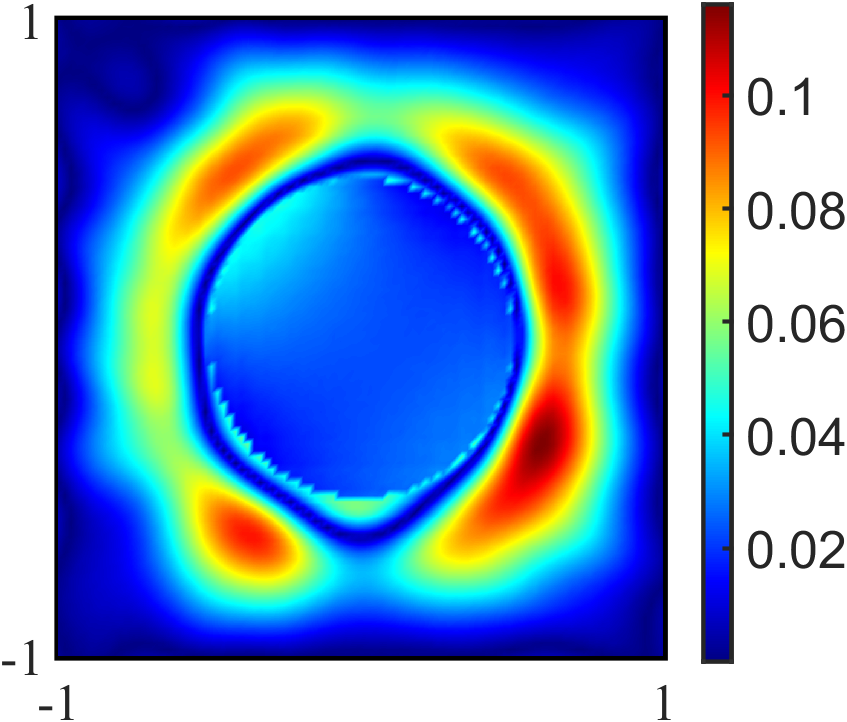}\vspace*{-0.1cm}
\centerline{\small{$K=13$}}
\end{minipage}
& 
\begin{minipage}{.23\textwidth}
\centering
\includegraphics[width=0.9\linewidth]{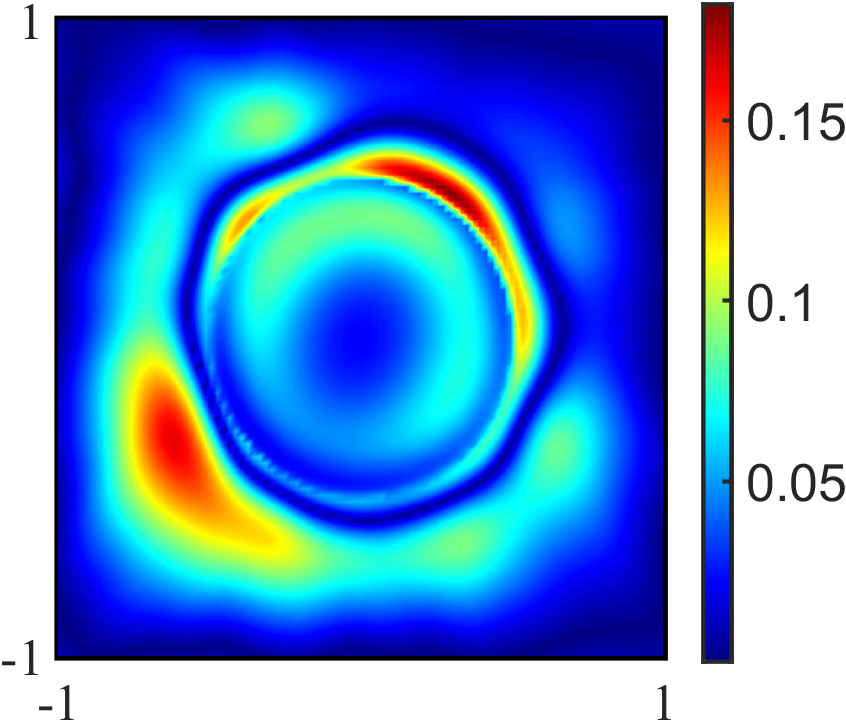}\vspace*{-0.1cm}
\centerline{\small{$K=13$}}
\end{minipage}
\\
\bottomrule
\end{tabular}
\label{table-Xmp1-error-profiles}

\vspace{0.35cm}

\caption{Comparison of different methods in terms of relative $L^2$ errors for example \eqref{NumXmp1-Circle}.}
\centering
\renewcommand{\arraystretch}{1.1}
\begin{tabular}{ | p{1cm} || c | c | c | c |   }
\hline
\multicolumn{2}{|c|}{ \diagbox[width=13em]{Coefficients}{Outer Iteration} } & 2 & 4 & 10   \\
\hline
\hline
\multirow{3}{*}{$\!(1,10^3)\!$} & DeepDDM & 0.013 $\pm$ 0.000 & 0.014 $\pm$ 0.001 & -  \\ 
\cline{2-5}
& DNLA (PINNs) & 0.004 $\pm$ 0.001 & 0.004 $\pm$ 0.000 & -  \\ 
\cline{2-5}
& \!\!\!\!\! DNLA (deep Ritz)\!\!\! & 0.014 $\pm$ 0.001 & 0.014 $\pm$ 0.001 & -  \\ 
\hline		
\hline
\multirow{3}{*}{\ $(1,1)$} & DeepDDM & 11.126 $\pm$ 0.096 & 2.997 $\pm$ 0.177 & 0.910 $\pm$ 0.356 \\
\cline{2-5}
& DNLA (PINNs) & 11.110 $\pm$ 0.245 & 2.720 $\pm$ 0.067 & 0.095 $\pm$ 0.011  \\ 
\cline{2-5}
& \!\!\!\!\! DNLA (deep Ritz)\!\!\! & 9.456 $\pm$ 4.506 & 2.322 $\pm$ 1.092 & 0.122 $\pm$ 0.023 \\ 
\hline		                                                     
\end{tabular}
\label{table-Xmp1-relative-errors}

\vspace{0.35cm}

\caption{Error profiles in terms of the pointwise $\ell^2$-norm of $ \nabla \hat{u}(x,y;\theta) - \nabla u(x,y) $ for example \eqref{NumXmp1-Circle}.}
\centering
\begin{tabular}{ c | c | c | c }
\toprule
Coefficients & DeepDDM  & DNLM (PINNs)  & DNLM (deep Ritz)  \\ 
\midrule
\makecell{$(1, 1)$ \\ with \\ $\rho = 0.5$} &
\begin{minipage}{.23\textwidth}
\centering
\includegraphics[width=0.9\linewidth]{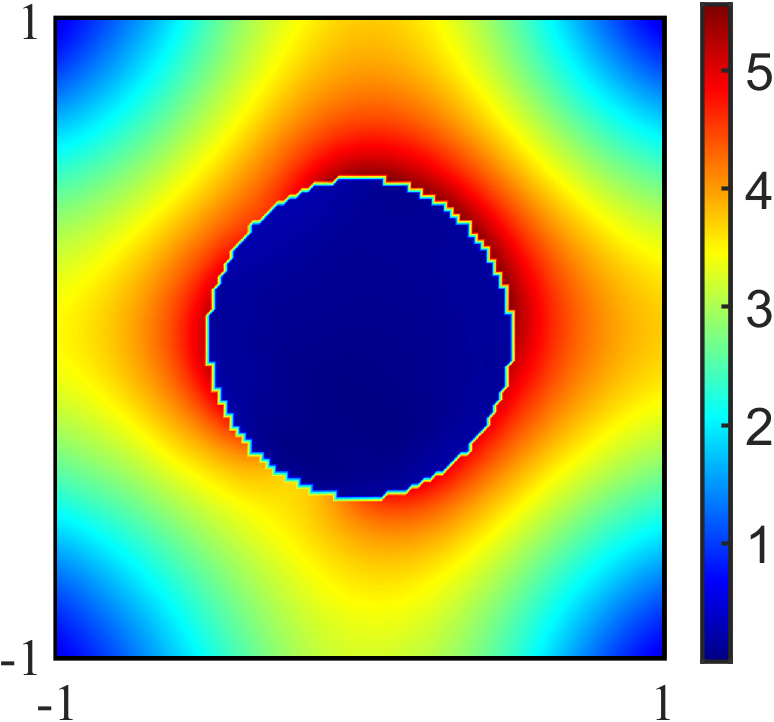}
\centerline{\small{$K=15$}}
\end{minipage}
& 
\begin{minipage}{.23\textwidth}
\centering
\includegraphics[width=0.9\linewidth]{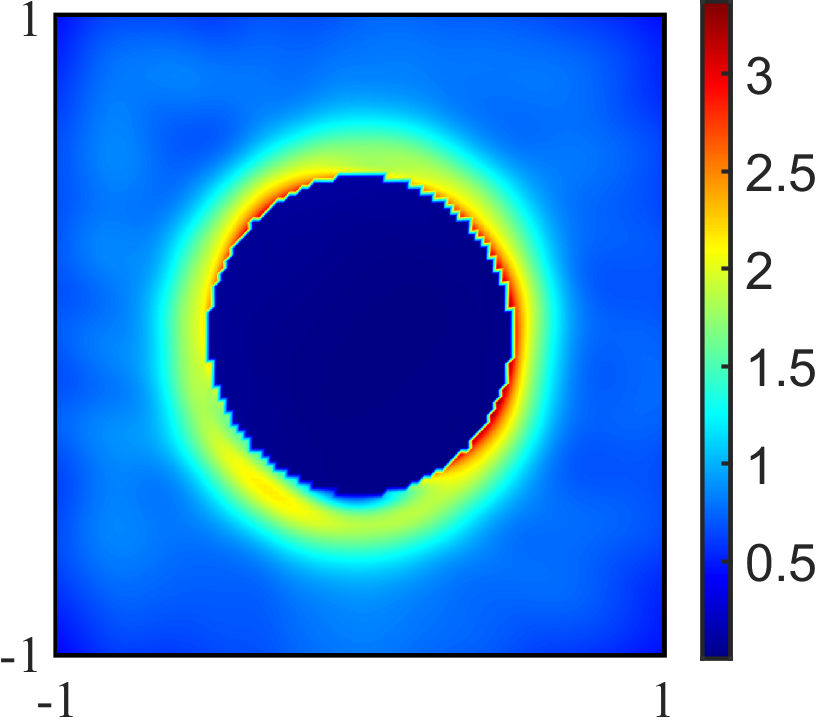}
\centerline{\small{$K=13$}}
\end{minipage}
& 
\begin{minipage}{.23\textwidth}
\centering
\includegraphics[width=0.9\linewidth]{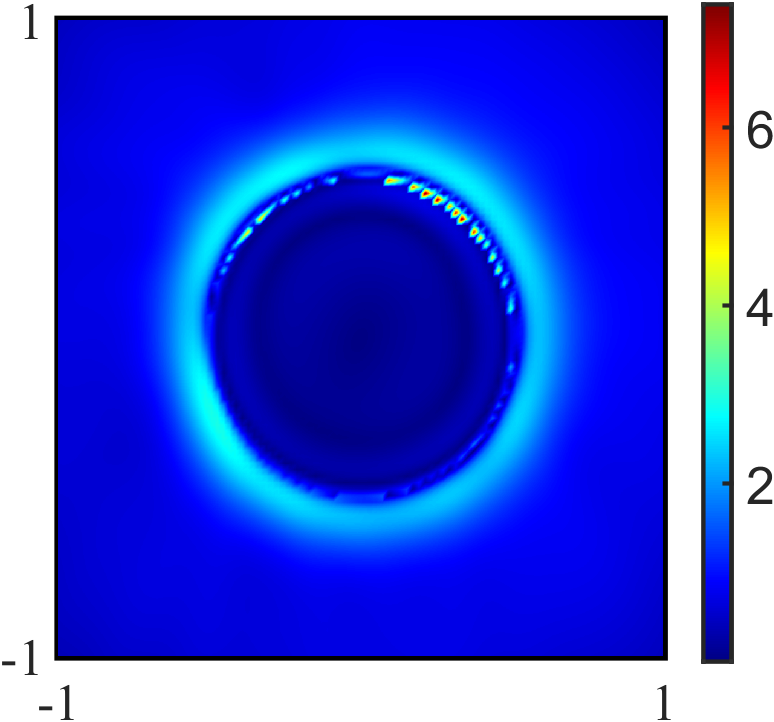}
\centerline{\small{$K=13$}}
\end{minipage}
\\ 
\bottomrule
\end{tabular}
\label{table-Xmp1-graderror-profiles}
\vspace{-0.2cm}
\end{table}

To begin with, we consider an elliptic interface problem \eqref{IntfcProb-StrongForm} whose solution is continuous \cite{li2006immersed}, namely,
\begingroup
\renewcommand*{\arraystretch}{1.1}
\begin{equation}
\begin{array}{cl}
-\nabla \cdot \left( c_i \nabla u_i(x,y)  \right) = f(x,y)\ \ & \text{in}\ \Omega_i,\\
u_i(x,y) = g(x,y)\ \ & \text{on}\ \partial \Omega_i\cap\partial \Omega, \\
u_1(x,y) = u_2(x,y)\ \text{and}\ \llbracket c(x) \nabla u(x)\cdot\bm{n} \rrbracket = q(x,y) \ \ & \text{on}\ \Gamma,
\end{array}
\label{NumXmp1-Circle}
\end{equation}
\endgroup
where the entire computation domain is $\Omega=(-1,1)\times(-1,1)$ and the interface $\Gamma = \big\{ (x,y) \,|\, x^2+y^2=0.25 \big\}$ is a circle centered at point $(0,0)$ with radius $R=0.5$. The source term $f(x,y)$ and the boundary data $g(x,y)$ on each subdomain, \textit{i.e.}, $\Omega_1 = \big\{ (x,y) \,|\, x^2+y^2<0.25 \big\}$ and $\Omega_2=\Omega\setminus\Omega_1$, are calculated from the exact solution 
\begingroup
\renewcommand*{\arraystretch}{1.1}
\begin{equation*}
u(x,y) = \left\{
\begin{array}{cl}
c_1^{-1} e^{ 10 (r^2-R^2)}\ \ & \text{if}\ r<R \ \ \ \ (\text{or in}\ \Omega_1),\\
c_2^{-1}e^{ 10 (R^2-r^2)} +\big( c_1^{-1} - c_2^{-1}\big) e^{r^2-R^2} \ \ & \text{otherwise} \ (\text{or in}\ \Omega_2),
\end{array}\right.
\end{equation*}
\endgroup
where $r=\sqrt{x^2+y^2}$. More specifically, \autoref{fig-Xmp1-exact-solution} depicts the exact solution $u(x,y)$ of \eqref{NumXmp1-Circle} when the coefficients are given by $(c_1,c_2) = (1,10^3)$ and $(1,1)$ respectively, both of which should be solved reasonably well so as to meet the robustness requirement with respect to the varying coefficients. The initial guess of the unknown solution's value at interface is set to be $u_\Gamma^{[0]} = -1000x(x-1)y(y-1)+1$.

\autoref{table-Xmp1-error-profiles} shows the error profiles $| \hat{u}(x,y;\theta) - u(x,y) |$ of different methods in a typical simulation, while the relative $L^2$ errors (mean $\pm$ standard deviation over 5 independent runs) are displayed in \autoref{table-Xmp1-relative-errors}. Clearly, the elliptic interface problem \eqref{NumXmp1-Circle} with high-contrast coefficient $(c_1,c_2)=(1,10^3)$ can be solved with desired accuracy through the use of the DeepDDM scheme \cite{li2020deep}, however, it fails to be effective when the coefficient degenerates to the case of $(c_1,c_2)=$ $(1,1)$. On the other hand, our proposed methods can achieve promising performance no matter $c_1\ll c_2$ or $c_1=c_2$. 

More specifically, when solving the Dirichlet subproblem using neural networks, the trained model is often found to satisfy the underlying equations but exhibit erroneous Neumann traces (see \autoref{table-Xmp1-graderror-profiles}). As such, a direct computation of $\nabla \hat{u}(x,y;\theta)$ at interface would lead to an erroneous right-hand-side term in \eqref{Err-Intfc-DNLM-DeepDDM}, which explains the failure of DeepDDM when $c_1 = c_2$. On the contrary, our methods is based on the variational formulation of \eqref{DirichletSubProb-BndryPenalty-StrongForm} and \eqref{IntfcProb-DN-DirichletSubProb-StrongForm} with vanishing terms $\hat{u}_1^{[k]}$ and $u_1^{[k]}$, \textit{i.e.},
\begin{equation}
	\int_{\Omega_1} \big( \nabla \hat{u}_1^{[k]} - \nabla u_1^{[k]} \big) \cdot \nabla v_1 \, dx = \int_{\Gamma} \big( \nabla \hat{u}_1^{[k]} - \nabla u_1^{[k]} \big) \cdot \bm{n}_1 v_1 \,ds\ \ \ \textnormal{for any}\ v_1\in V_1,
	\label{Err-Intfc-DNLM-DeepDDM-Xmp1}
\end{equation}
which allows the flux transmission without evaluating $\nabla \hat{u}^{[k]}_1$ at interface, and therefore substantially differs from other deep learning-based methods (see section \ref{Section-RelatedWork}). The experimental results in \autoref{table-Xmp1-graderror-profiles} further demonstrate that our proposed algorithms are still effective in the presence of erroneous flux data. Furthermore, it can be inferred from \eqref{Err-Intfc-DNLM-DeepDDM-Xmp1} that the Neumann subproblem solver could benefit from a good approximation of $\nabla u_1^{[k]}$ inside the subdomain $\Omega_1$, and hence DNLA (PINNs) empirically performs better than DNLA (deep Ritz) (see \autoref{table-Xmp1-error-profiles}).


\subsection{Zigzag Interface in Two Dimension}

\begin{figure}[t!]
\centering
\includegraphics[width=0.212\textwidth]{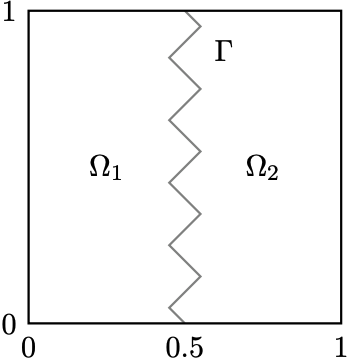}
\hspace{0.3cm}
\includegraphics[width=0.272\textwidth]{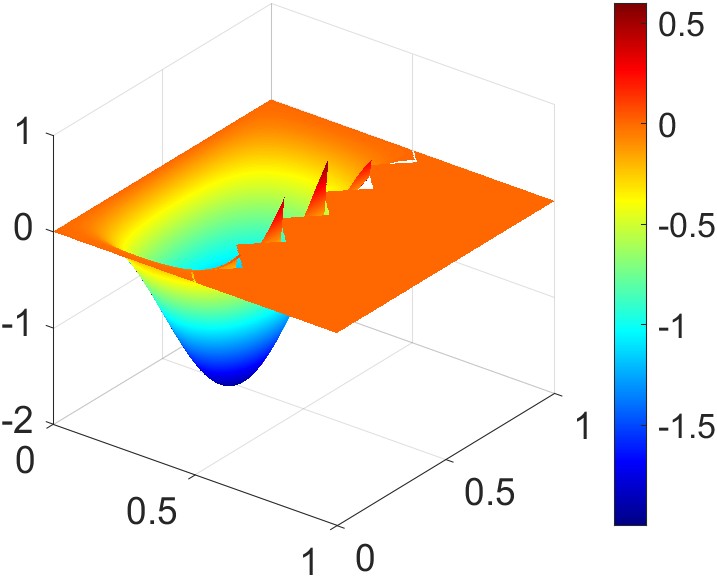}
\hspace{0.3cm}
\includegraphics[width=0.272\textwidth]{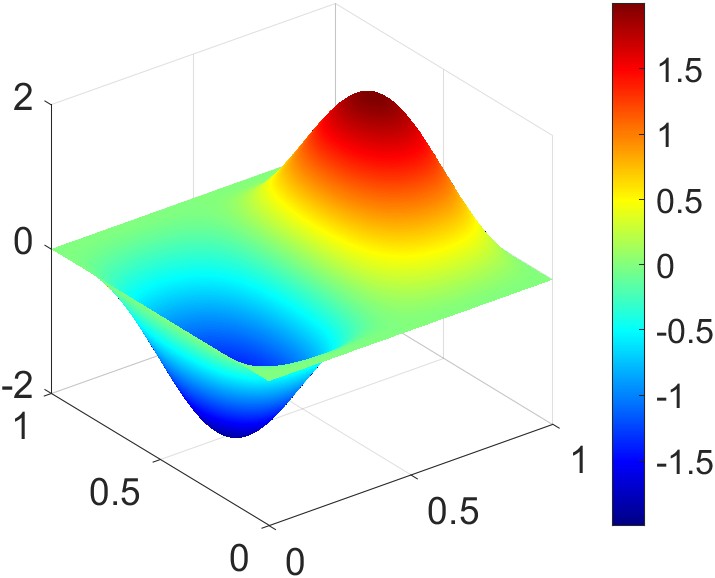}
\caption{Geometry of interface, exact solution with $(c_1,c_2)=(1,10^3)$ and $(1,1)$ for \eqref{NumXmp2-Zigzag}.}
\label{Xmp2-exact-solution}
\vspace{-0.3cm}
\end{figure}

\begin{table}[t!]
\caption{Error profiles $ \vert \hat{u}(x,y;\theta) - u(x,y)\vert $ of different methods for example \eqref{NumXmp2-Zigzag} (see also \textbf{Appendix C}).}
\centering
\begin{tabular}{ c | c | c | c }
\toprule
Coefficients & DeepDDM & DNLA (PINNs) & DNLA (deep Ritz) \\ 
\midrule
\makecell{$(1, 10^3)$ \\ with \\ $\rho = 1.0$}  &
\begin{minipage}{.23\textwidth}
\centering
\includegraphics[width=.9\linewidth]{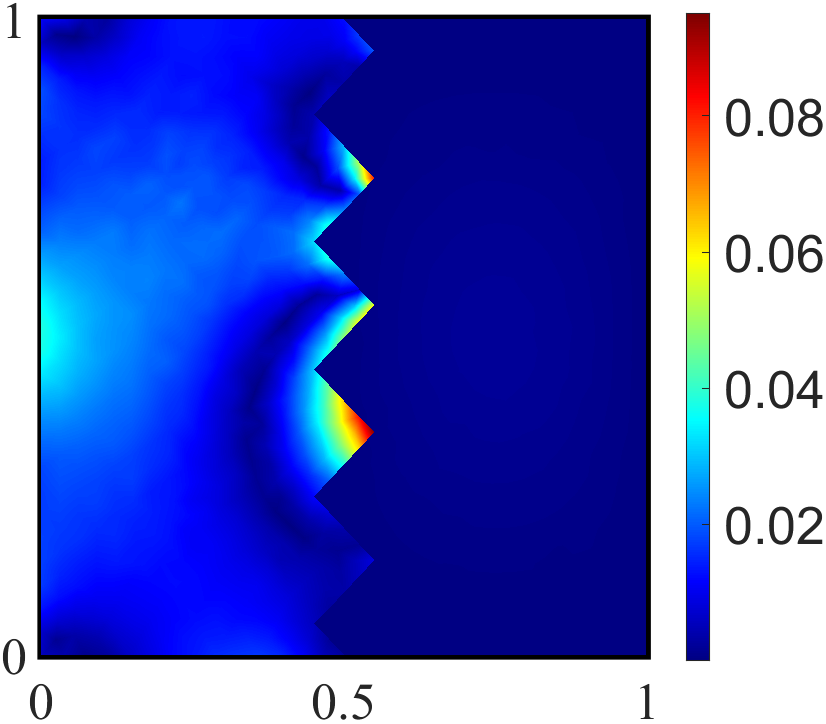}\vspace*{-0.1cm}
\centerline{\small{$(K=2)$}}
\end{minipage}
& 
\begin{minipage}{.23\textwidth}
\centering
\includegraphics[width=.9\linewidth]{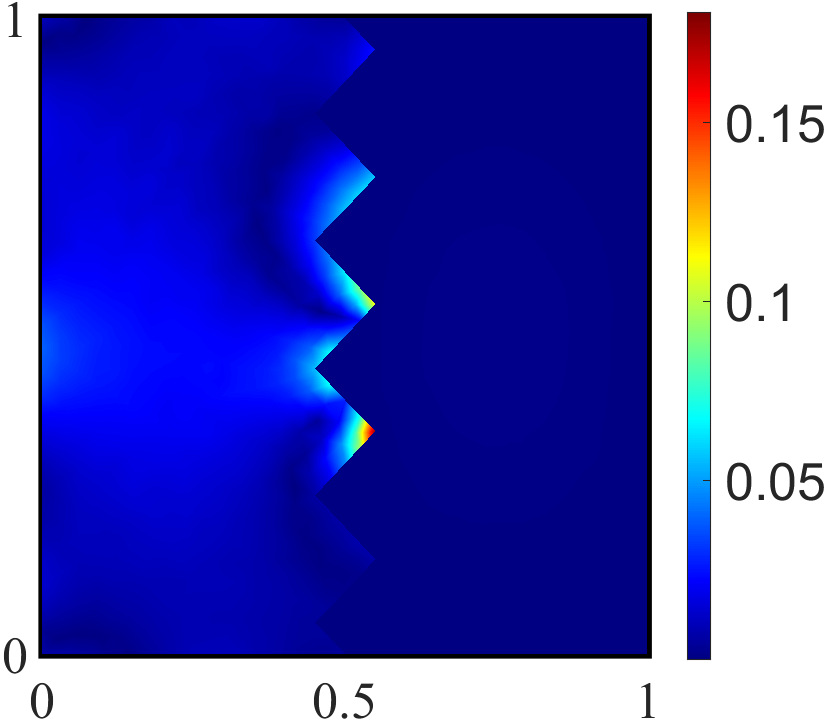}\vspace*{-0.1cm}
\centerline{\small{$(K=2)$}}
\end{minipage}
& 
\begin{minipage}{.23\textwidth}
\centering
\includegraphics[width=.9\linewidth]{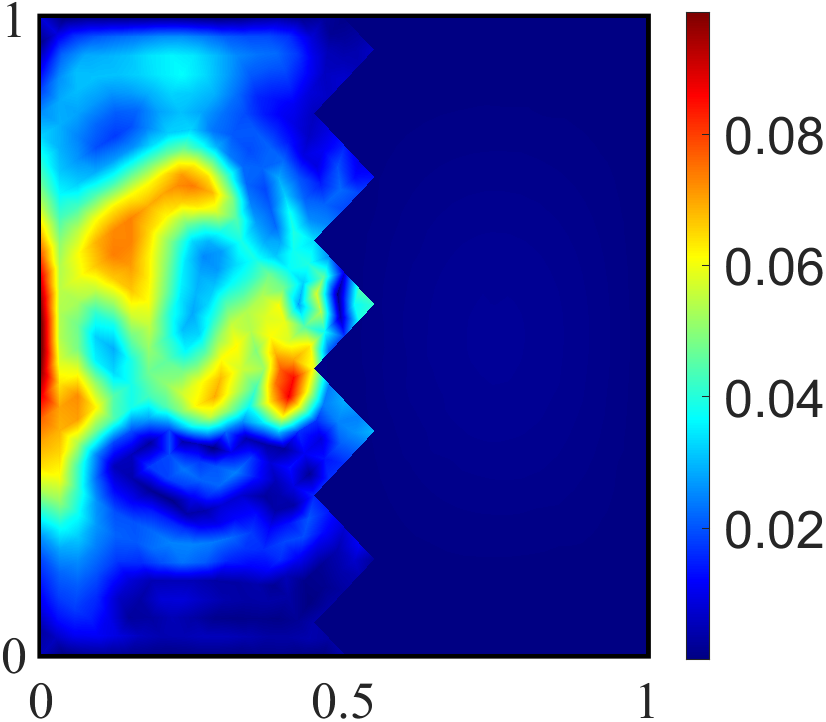}\vspace*{-0.1cm}
\centerline{\small{$(K=2)$}}
\end{minipage}
\\ 
\midrule
\makecell{$(c_1,c_2)=(1, 1)$ \\ $\rho = 0.5$} &
\begin{minipage}{.23\textwidth}
\centering
\includegraphics[width=0.87\linewidth]{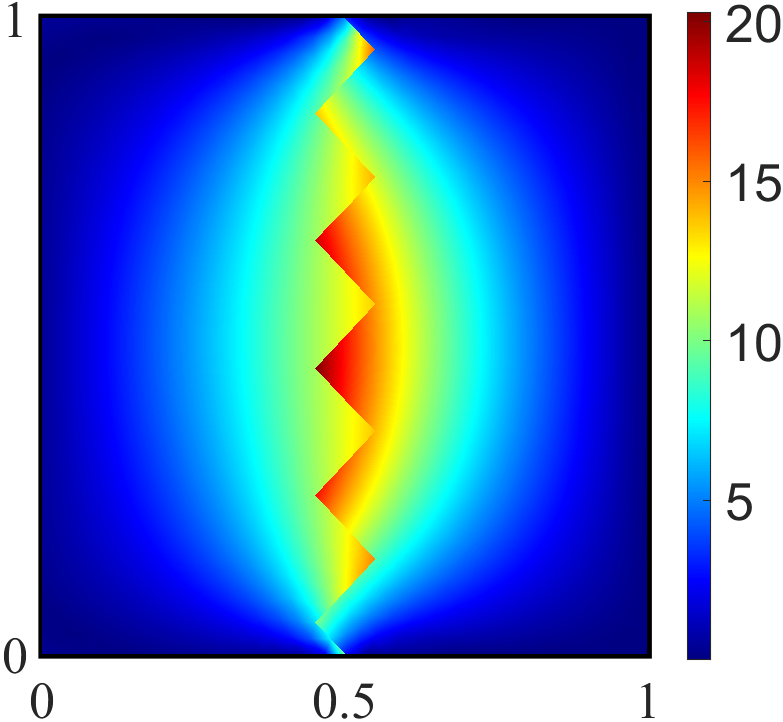}\vspace*{-0.1cm}
\centerline{\small{$(K=15)$}}
\end{minipage}
& 
\begin{minipage}{.23\textwidth}
\centering
\includegraphics[width=.9\linewidth]{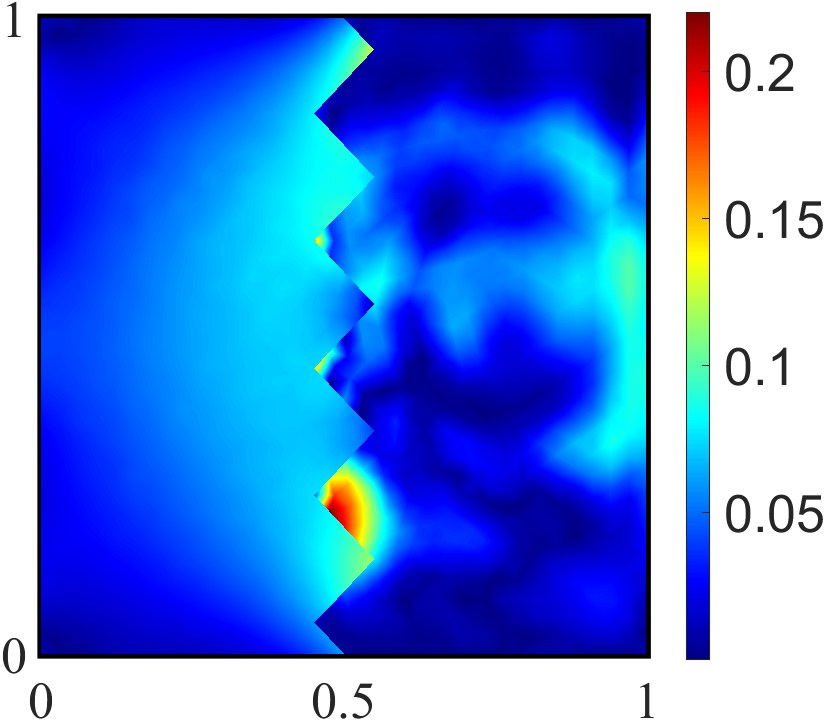}\vspace*{-0.1cm}
\centerline{\small{$(K=8)$}}
\end{minipage}
& 
\begin{minipage}{.23\textwidth}
\centering
\includegraphics[width=.92\linewidth]{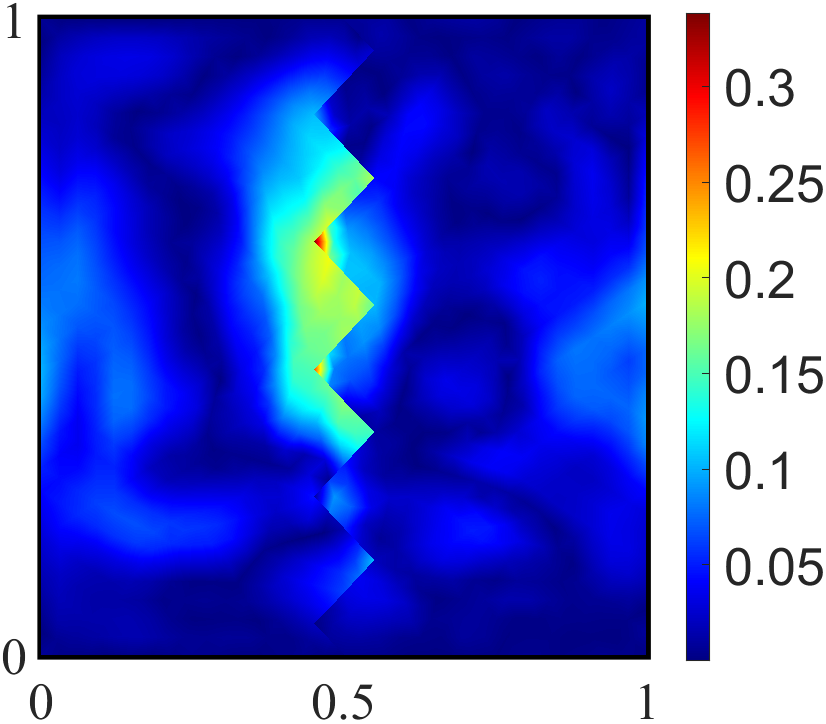}\vspace*{-0.1cm}
\centerline{\small{$(K=12)$}}
\end{minipage}
\\
\bottomrule
\end{tabular}
\label{table-Xmp2-error-profiles}

\vspace{0.35cm}

\caption{Comparison of different methods in terms of relative $L^2$ errors for example \eqref{NumXmp2-Zigzag}.}
\centering
\renewcommand{\arraystretch}{1.1}
\begin{tabular}{ | p{1cm} || c | c | c | c | c |  }
\hline
\multicolumn{2}{|c|}{ \diagbox[width=12.8em]{Coefficients}{Outer Iteration} }  & 2 & 4 & 8 \\
\hline
\hline
\multirow{3}{*}{\centerline{$\!\!\!(1,10^3)\!\!\!$}} & DeepDDM & \!\!\! 0.054 \!$\pm$\! 0.023 \!\!\! & \!\!\! 0.041 \!$\pm$\! 0.025 \!\!\! & - \\ 
\cline{2-5}
& DNLA (PINNs) & \!\!\! 0.059 \!$\pm$\! 0.026 \!\!\! & \!\!\! 0.061 \!$\pm$\! 0.019 \!\!\! & - \\ 
\cline{2-5}
& \!\!\!\! DNLA (deep Ritz)\!\!\! & \!\!\! 0.041 \!$\pm$\! 0.006 \!\!\! & \!\!\! 0.040 \!$\pm$\! 0.009 \!\!\! & - \\ 
\hline	
\hline
\multirow{3}{*}{\centerline{$\!\!\!(1,1)\!\!\!$}} & DeepDDM & \!\!\! 8.453 \!$\pm$\! 0.372 \!\!\! & \!\!\! 9.103 \!$\pm$\! 0.715 \!\!\! & \!\!\! 10.569 \!$\pm$\! 0.682 \!\!\! \\ 
\cline{2-5}
& DNLA (PINNs) & \!\!\! 2.132 \!$\pm$\! 0.772 \!\!\! & \!\!\! 0.157 \!$\pm$\! 0.036 \!\!\! & \!\!\! 0.132 \!$\pm$\! 0.033 \!\!\! \\ 
\cline{2-5}
& \!\!\!\! DNLA (deep Ritz)\!\!\! & \!\!\! 21.022 \!$\pm$\! 4.628 \!\!\! & \!\!\! 5.183 \!$\pm$\! 1.469 \!\!\! & \!\!\! 0.814 \!$\pm$\! 0.030 \!\!\! \\ 
\hline		                                                     
\end{tabular}
\label{table-Xmp2-relative-errors}

\vspace{0.35cm}

\caption{Error profiles in terms of the pointwise $\ell^2$-norm of $\nabla \hat{u}(x,y;\theta) - \nabla u(x,y) $ for example \eqref{NumXmp2-Zigzag}.}
\centering
\begin{tabular}{ c | c | c | c }
\toprule
Coefficients & DeepDDM & DNLA (PINNs) & DNLA (deep Ritz) \\ 
\midrule
\makecell{$(1, 1)$ \\ with \\ $\rho = 0.5$} &
\begin{minipage}{.23\textwidth}
\centering
\includegraphics[width=.9\linewidth]{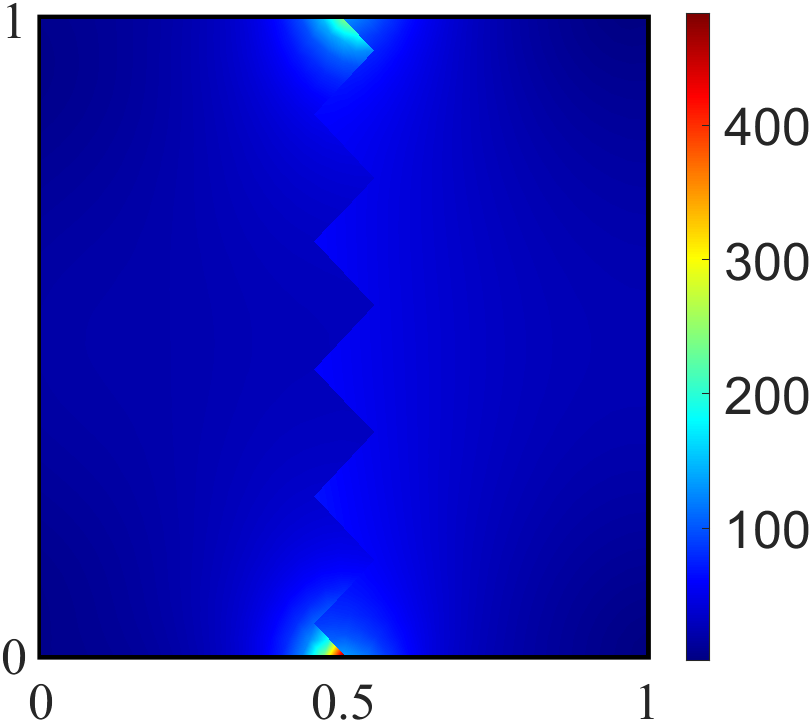}\vspace*{-0.1cm}
\centerline{\small{$(K=15)$}}
\end{minipage}
& 
\begin{minipage}{.23\textwidth}
\centering
\includegraphics[width=.86\linewidth]{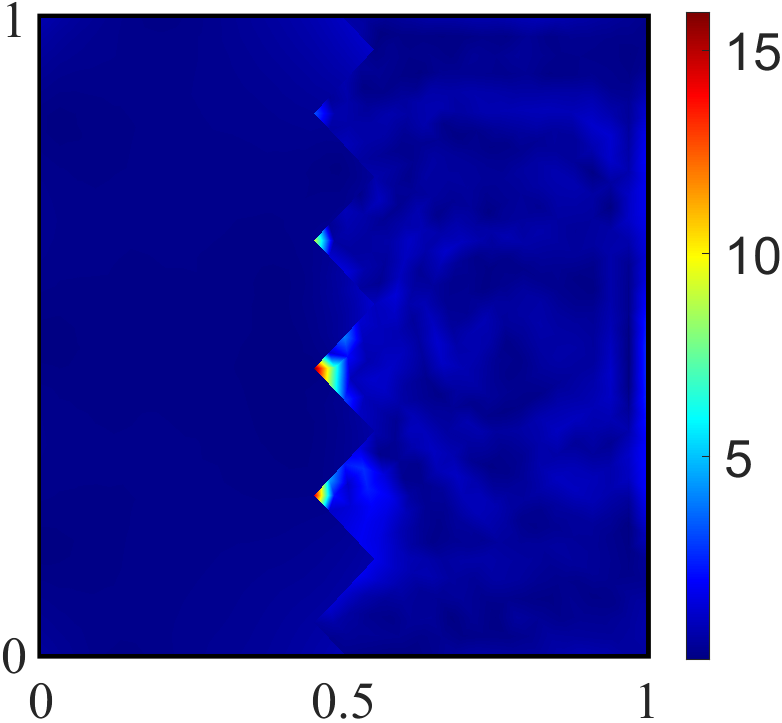}\vspace*{-0.1cm}
\centerline{\small{$(K=8)$}}
\end{minipage}
& 
\begin{minipage}{.23\textwidth}
\centering
\includegraphics[width=.83\linewidth]{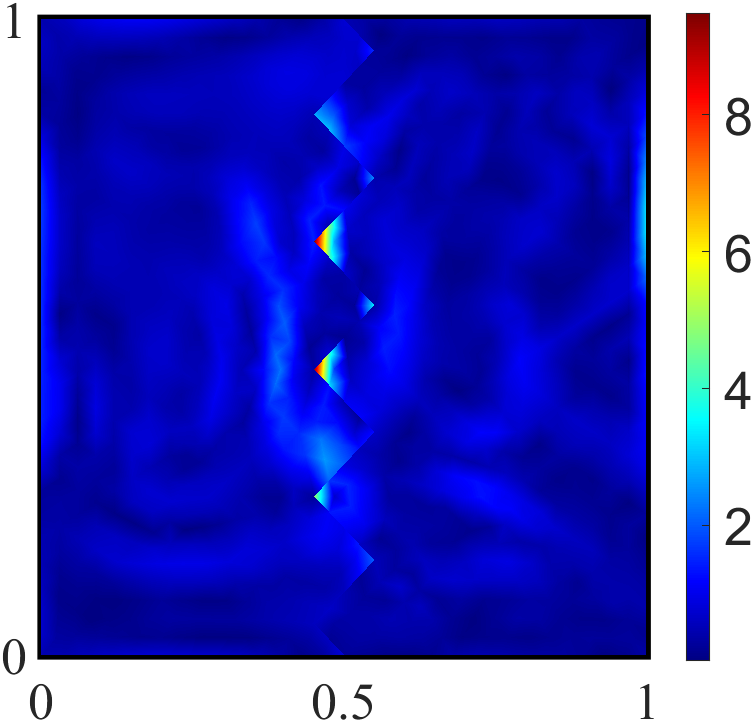}\vspace*{-0.1cm}
\centerline{\small{$(K=12)$}}
\end{minipage}
\\
\bottomrule
\end{tabular}
\label{table-Xmp2-graderror-profiles}
\vspace{-0.2cm}
\end{table}

We then consider a more complicated interface (\textit{i.e.}, the zigzag curve shown in \autoref{Xmp2-exact-solution}) with non-homogeneous jump conditions to demonstrate the meshless advantage of our proposed algorithm over the traditional mesh-based numerical methods. More specifically, the boundary value problem \eqref{IntfcProb-StrongForm} now takes on the form, for $i=1$, 2,
\begingroup
\renewcommand*{\arraystretch}{1.1}
\begin{equation}
\begin{array}{cl}
-\nabla \cdot \left( c_i \nabla u_i(x,y)  \right) + u_i(x,y) = f(x,y)\ \ & \text{in}\ \Omega_i,\\
u_i(x,y) = g(x,y)\ \ & \text{on}\ \partial \Omega_i\cap\partial \Omega, \\
u_1(x,y) - u_2(x,y) = p(x,y)\ \text{and}\ \llbracket c(x) \nabla u(x)\cdot\bm{n} \rrbracket = q(x,y) \ \ & \text{on}\ \Gamma,
\end{array}
\label{NumXmp2-Zigzag}
\end{equation}
\endgroup
where $\Omega=(0,1)\times(0,1)$ and the interface is characterized by a zigzag function
\begin{equation*}
	x = Z_3 ( Z_1(20y-\text{floor}(20y)) + Z_2 ) + 0.5
\end{equation*}
with $ Z_1 = 0.05 (-1 + 2 \text{mod}(\text{floor}(20y), 2))$, $Z_2=-0.05 \text{mod}(\text{floor}(20x),2)$ and, $Z_3=-2\text{mod}(\text{floor}(10x),2)+1$. Here, the source term $f(x,y)$, boundary value $g(x,y)$, jump functions $p(x,y)$ and $q(x,y)$ are derived from the true solution 
\begingroup
\renewcommand*{\arraystretch}{1.1}
\begin{equation*}
u(x,y) = \left\{
\begin{array}{cl}
c_1^{-1} \sin(2\pi x) ( \cos(2\pi y) - 1)\ \ &  \text{in}\ \Omega_1,\\
c_2^{-1} \sin(2\pi x) ( \cos(2\pi y) - 1)\ \ & \text{in}\ \Omega_2.
\end{array}\right.
\end{equation*}
\endgroup
Similar as before, the elliptic coefficient is chosen as $(c_1,c_2) = (1,10^3)$ and $(1,1)$ to evaluate the robustness against varying coefficients, where the exact solutions are depicted in \autoref{Xmp2-exact-solution}. Moreover, the initial guess of the unknown solution's value at interface is set to be $u_\Gamma^{[0]}(x,y)=\sin(2\pi x)(\cos(2\pi y)-1) - 1000 x (x-1) y (y-1)$. 

In a typical simulation, the error profiles and relative $L^2$ errors for different methods are reported in \autoref{table-Xmp2-error-profiles}, \autoref{table-Xmp2-graderror-profiles}, and \autoref{table-Xmp2-relative-errors}. It can be concluded that the DeepDDM approach performs well in the case of high-contrast coefficients, but fails to converge when the coefficients degenerate to $c_1=c_2=1$, which reveals the fact that DeepDDM breaks down rapidly in face of inaccurate flux predictions.

On the contrary, the iterative solutions using our algorithms can match well with the exact solution whenever the coefficient satisfies $c_1\ll c_2$ or $c_1 = c_2$ (see \autoref{table-Xmp2-error-profiles}, \autoref{table-Xmp2-relative-errors}). Note that the relaxation parameter is set to be $\rho = 1$ for the case of high-contrast coefficients $(c_1,c_2)=(1,10^3)$ and therefore the optimal convergence \cite{na2022domain,gander2015optimized} is nearly obtained in two iterations. Accordingly, we then choose $\rho=0.5$ for the case of $(c_1,c_2)=(1,1)$, where additional iterations are required due to the numerical error caused by the deep learning solvers. Notably, when solving the Dirichlet subproblem through PINNs, the second-order derivatives of $\hat{u}_1(x;\theta_1)$ are explicitly incorporated into the training loss function, leading to a satisfactory approximation of $\nabla \hat{u}_1(x;\theta_1)$ inside the subdomain $\Omega_1$. As such, DNLA (PINNs) shows better empirical performance than DNLA (deep Ritz) as reported in \autoref{table-Xmp2-relative-errors}. 

It is also noteworthy that even though the outer iteration converges by using our methods, the erroneous gradient approximation along the subdomain interfaces remains inevitable as depicted in \autoref{table-Xmp2-graderror-profiles}. Fortunately, our flux transmission does not rely on the direct computation of $\nabla \hat{u}_1(x;\theta_1)$ at interface, which plays a key role.


\subsection{Checkerboard Lattice in Two Dimension}

\begin{figure}[t!]
\centering
\includegraphics[width=0.218\textwidth]{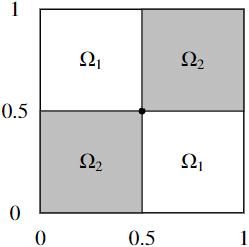}
\hspace{0.3cm}
\includegraphics[width=0.272\textwidth]{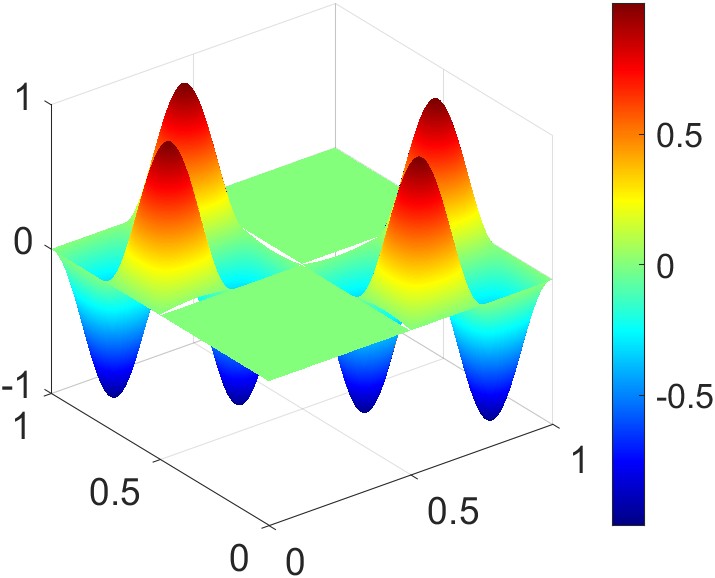}
\hspace{0.3cm}
\includegraphics[width=0.265\textwidth]{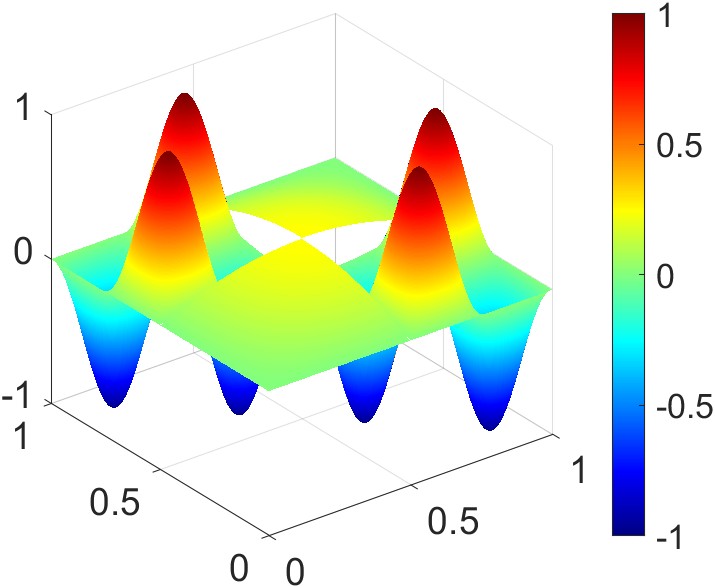}
\caption{Geometry of interface, exact solution with $(c_1,c_2)=(1,10^3)$ and $(1,1)$ for \eqref{NumXmp3-Checkerboard}.}
\label{Xmp3-exact-solution}
\end{figure}

\begin{table}[t!]
\caption{Error profiles $ \vert \hat{u}(x,y;\theta) - u(x,y)\vert $ of different methods for example \eqref{NumXmp3-Checkerboard} (see also \textbf{Appendix C}).}
\centering
\begin{tabular}{ c | c | c | c }
\toprule
Coefficients & DeepDDM & DNLA (PINNs) & DNLA (deep Ritz) \\ 
\midrule
\makecell{$(1, 10^3)$ \\ with \\ $\rho = 1.0$}  &
\begin{minipage}{.23\textwidth}
\centering
\includegraphics[width=.9\linewidth]{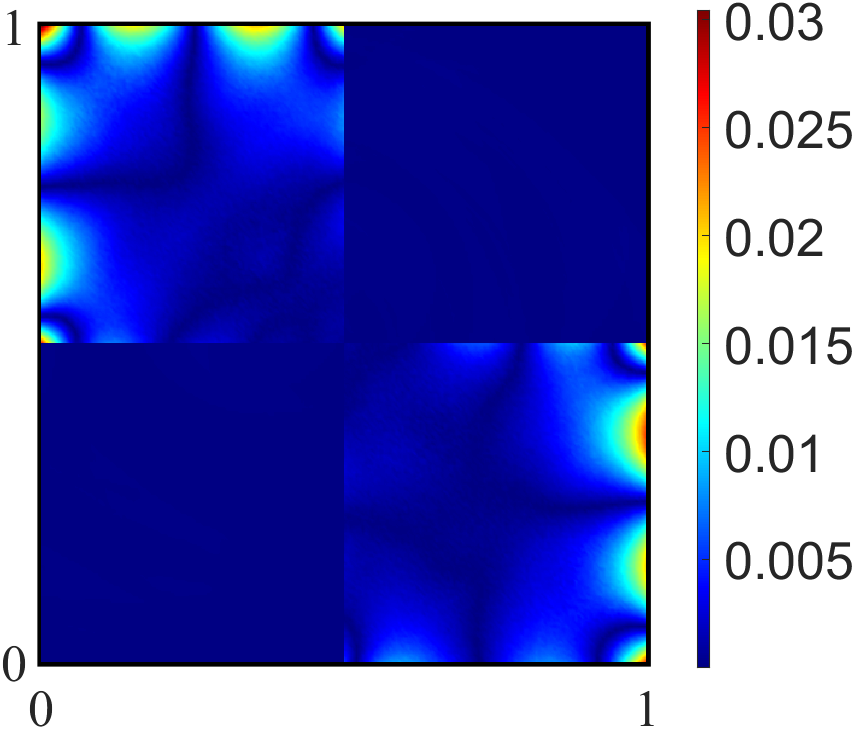}\vspace*{-0.1cm}
\centerline{\small{$(K=2)$}}
\end{minipage}
& 
\begin{minipage}{.23\textwidth}
\centering
\includegraphics[width=.9\linewidth]{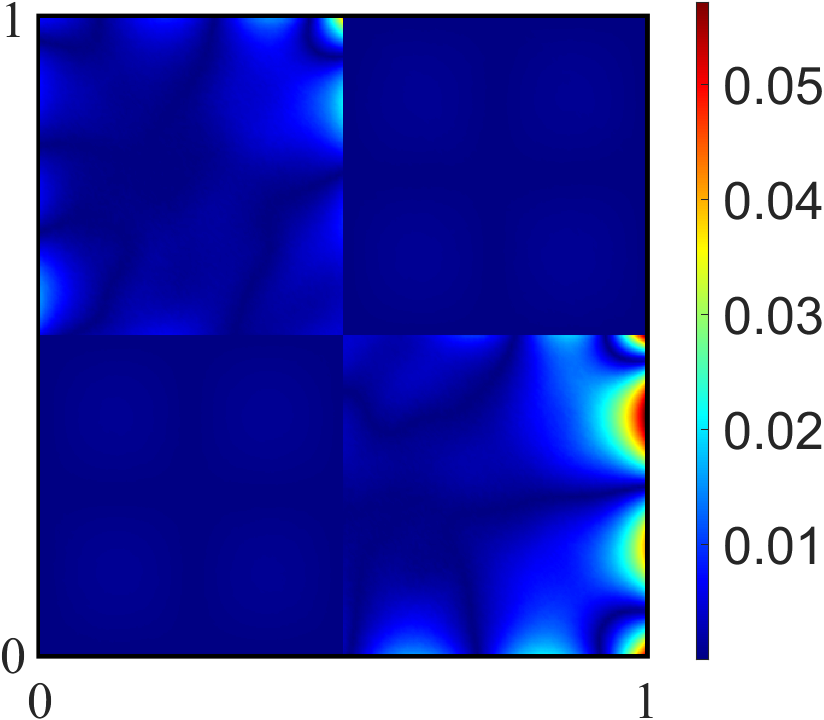}\vspace*{-0.08cm}
\centerline{\small{$(K=2)$}}
\end{minipage}
& 
\begin{minipage}{.23\textwidth}
\centering
\includegraphics[width=.9\linewidth]{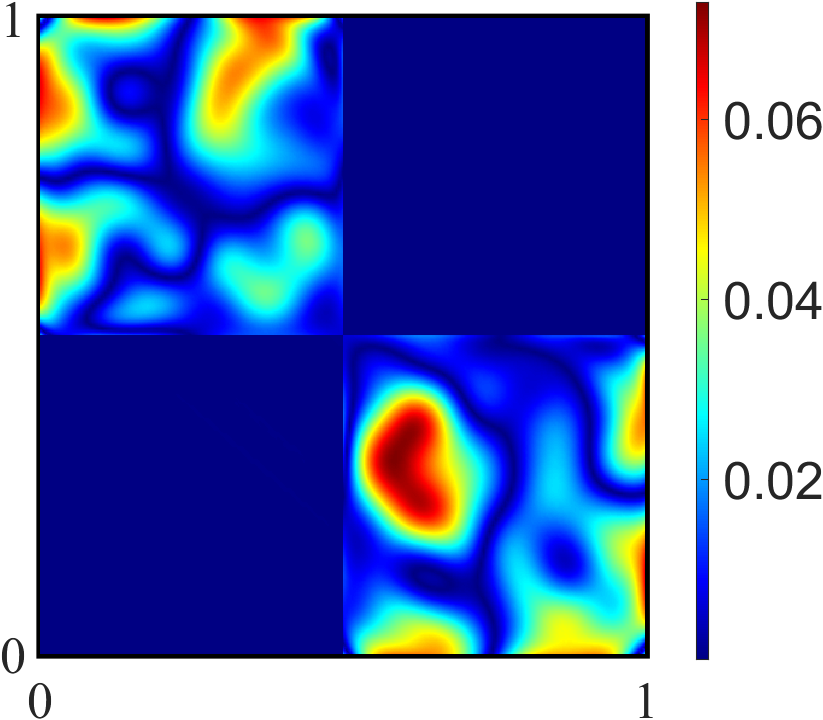}\vspace*{-0.1cm}
\centerline{\small{$(K=2)$}}
\end{minipage}
\\ 
\midrule
\makecell{$(1, 1)$ \\ with \\ $\rho = 0.5$} &
\begin{minipage}{.23\textwidth}
\centering
\includegraphics[width=0.9\linewidth]{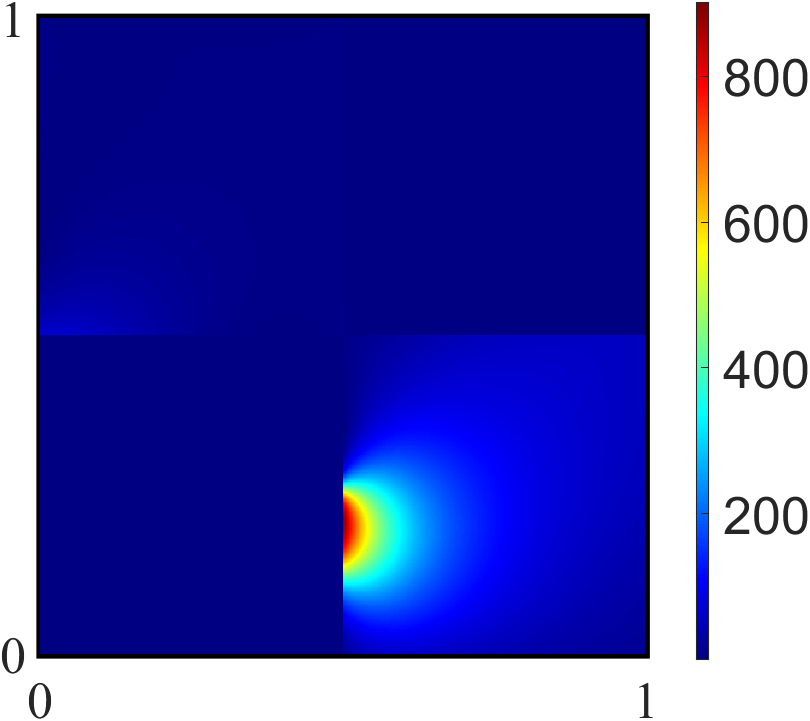}\vspace*{-0.1cm}
\centerline{\small{$(K=15)$}}
\end{minipage}
& 
\begin{minipage}{.23\textwidth}
\centering
\includegraphics[width=.9\linewidth]{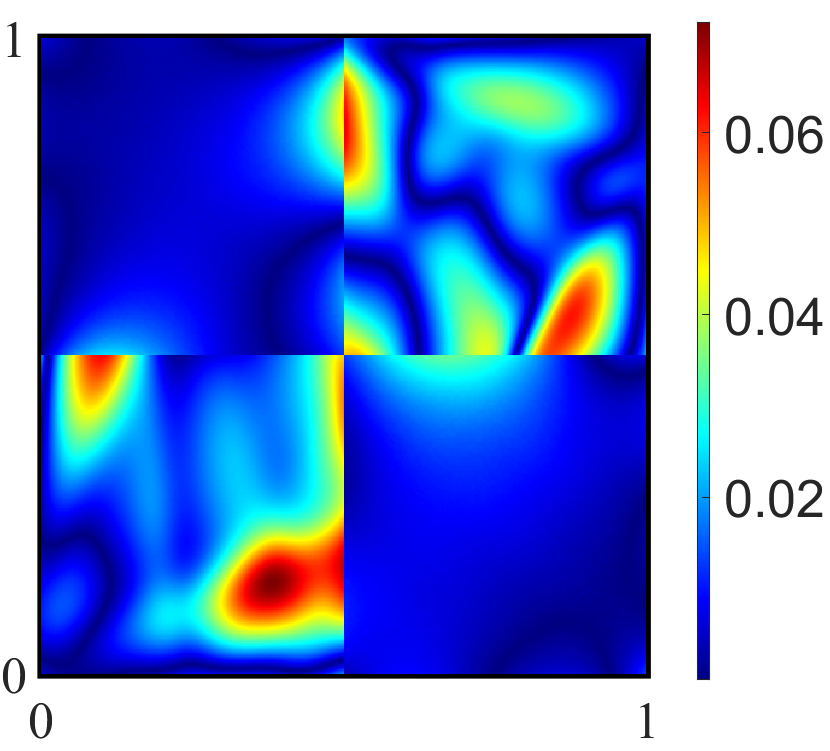}\vspace*{-0.1cm}
\centerline{\small{$(K=9)$}}
\end{minipage}
& 
\begin{minipage}{.23\textwidth}
\centering
\includegraphics[width=0.9\linewidth]{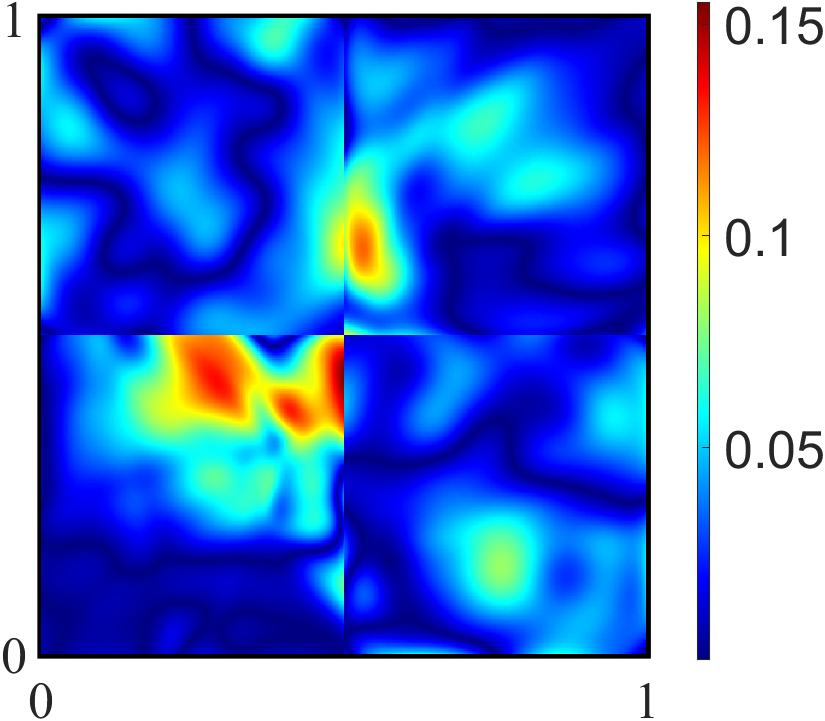}\vspace*{-0.1cm}
\centerline{\small{$(K=10)$}}
\end{minipage}
\\
\bottomrule
\end{tabular}
\label{table-Xmp3-error-profiles}

\vspace{0.35cm}

\caption{Comparison of different methods in terms of relative $L^2$ errors for example \eqref{NumXmp3-Checkerboard}.}
\centering
\renewcommand{\arraystretch}{1.1}
\begin{tabular}{ | p{1cm} || c | c | c | c |  }
\hline
\multicolumn{2}{|c|}{ \diagbox[width=12.8em]{Coefficients}{Outer Iterations} } & 1 & 2 & 10  \\
\hline
\hline
\multirow{3}{*}{\centerline{$\!\!\!(1,10^3)\!\!\!$}} & DeepDDM & 0.448 $\pm$ 0.002 & 0.015 $\pm$ 0.006   & - \\ 
\cline{2-5}
& DNLA (PINNs) & 0.447 $\pm$ 0.002 & 0.020 $\pm$ 0.009 & -   \\ 
\cline{2-5}
& \!\!\!\! DNLA (deep Ritz)\!\!\! & 0.457 $\pm$ 0.009 & 0.053 $\pm$ 0.001  & - \\ 
\hline	
\hline
\multirow{3}{*}{\centerline{$\!\!\!(1,1)\!\!\!$}} & DeepDDM & 1.820 $\pm$ 0.125 & 8.981 $\pm$ 0.486 & 459.855 $\pm$ 352.586 \\ 
\cline{2-5}
& DNLA (PINNs) & 1.052 $\pm$ 0.244 & 0.644 $\pm$ 0.069 & 0.134 $\pm$ 0.040 \\ 
\cline{2-5}
& \!\!\!\! DNLA (deep Ritz)\!\!\! & 0.944 $\pm$ 0.302 & 0.720 $\pm$ 0.428 & 0.365 $\pm$ 0.166 \\ 
\hline		                                                     
\end{tabular}
\label{table-Xmp3-relative-errors}
\vspace{-0.2cm}
\end{table}

Next, we consider the interface problem with domain partitioned in a checkerboard fashion (see \autoref{Xmp3-exact-solution}),
\begingroup
\renewcommand*{\arraystretch}{1.1}
\begin{equation}
\begin{array}{cl}
-\nabla \cdot \left( c_i \nabla u_i(x,y)  \right) + u_i(x,y) = f(x,y)\ \ & \text{in}\ \Omega_i,\\
u_i(x,y) = g(x,y)\ \ & \text{on}\ \partial \Omega_i\cap\partial \Omega, \\
u_1(x,y) - u_2(x,y) = p(x,y)\ \text{and}\ \llbracket c(x) \nabla u(x)\cdot\bm{n} \rrbracket = q(x,y) \ \ & \text{on}\ \Gamma,
\end{array}
\label{NumXmp3-Checkerboard}
\end{equation}
\endgroup
for $i=1$, 2, where the interior cross-point is well-known to require extra treatments when using mesh-based methods \cite{toselli2004domain,quarteroni1999domain}. Similar as before, the source term, boundary and interface conditions of our underlying problem are derived from the exact solution with a large jump across interface (see \autoref{Xmp3-exact-solution})
\begingroup
\renewcommand*{\arraystretch}{1.1}
\begin{equation*}
u(x,y) = \left\{
\begin{array}{cl}
c_1^{-1} \sin(4\pi y)\sin(4\pi x) \ \ &  \text{in}\ \Omega_1,\\
c_2^{-1} 4x(x-1)y(y-1) \ \ & \text{in}\ \Omega_2.
\end{array}\right.
\end{equation*}
\endgroup

As can be seen from \autoref{table-Xmp3-error-profiles} and \autoref{table-Xmp3-relative-errors}, the DeepDDM, DNLA (PINNs), and DNLA (deep Ritz) methods are able to provide satisfactory performance for interface problem \eqref{NumXmp3-Checkerboard} with high-contrast coefficients $(c_1,c_2)=(1,10^3)$. However, the DeepDDM scheme diverges as the ratio $c_1/c_2$ fails to eliminate the gradient approximation error \eqref{Err-Intfc-DNLM-DeepDDM}, while our methods still work reasonably well. Here, the initial guess takes on the form $u_\Gamma^{[0]}(x,y) = \sin(4\pi y)\sin(4\pi x) + 100x(x-1)^3y(y-1)^3$.

\begin{table}[t!]
\caption{Error profiles in terms of the pointwise $\ell^2$-norm of $\nabla \hat{u}(x,y;\theta) - \nabla u(x,y) $ for example \eqref{NumXmp3-Checkerboard}.}
\centering
\begin{tabular}{ c | c | c | c }
\toprule
Coefficients & DeepDDM & DNLA (PINNs) & DNLA (Deep Ritz) \\ 
\midrule
\makecell{$(1, 1)$ \\ with \\ $\rho = 0.5$} &
\begin{minipage}{.23\textwidth}
\centering
\includegraphics[width=.9\linewidth]{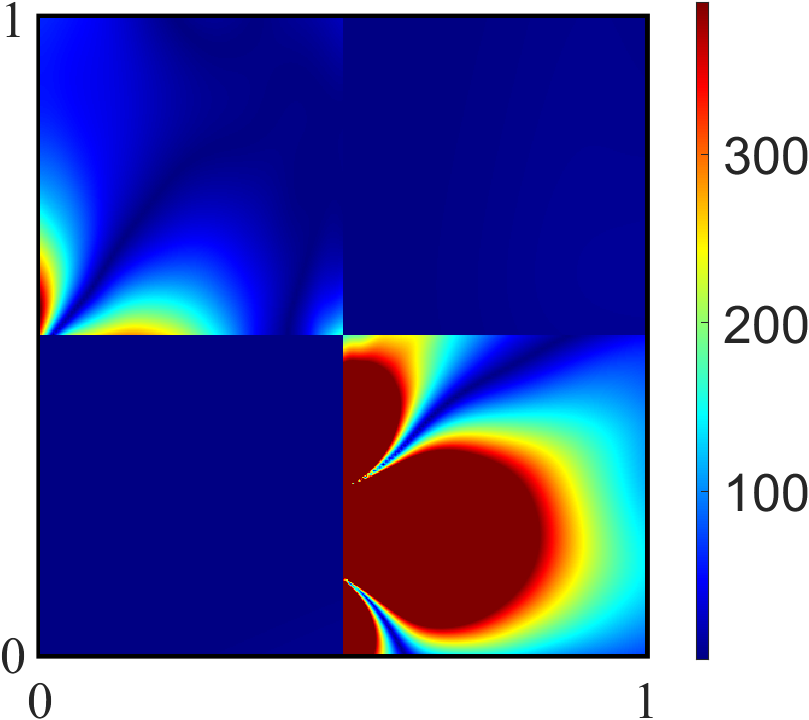}\vspace*{-0.1cm}
\centerline{\small{$(K=15)$}}
\end{minipage}
& 
\begin{minipage}{.23\textwidth}
\centering
\includegraphics[width=0.9\linewidth]{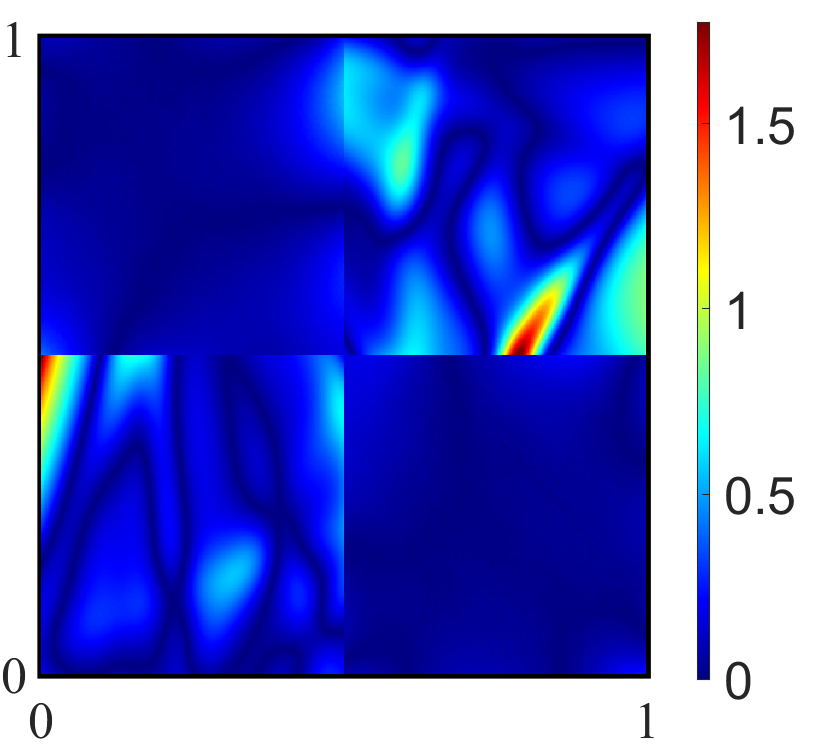}\vspace*{-0.1cm}
\centerline{\small{$(K=6)$}}
\end{minipage}
& 
\begin{minipage}{.23\textwidth}
\centering
\includegraphics[width=0.9\linewidth]{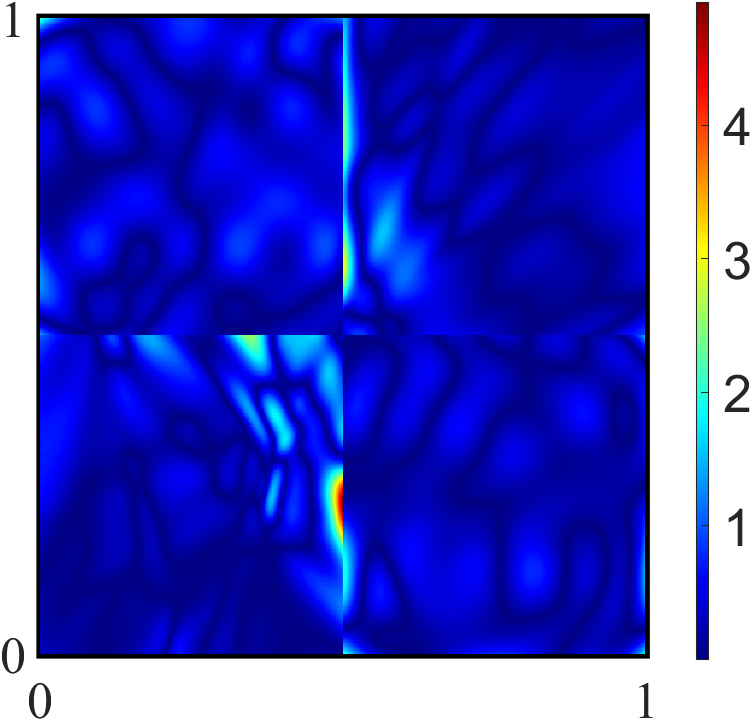}\vspace*{-0.1cm}
\centerline{\small{$(K=10)$}}
\end{minipage}
\\
\bottomrule
\end{tabular}
\label{table-Xmp3-graderror-profiles}
\vspace{-0.2cm}
\end{table}

Thanks to the variational treatment of the flux jump condition \eqref{IntfcProb-DN-NeumannSubProb-DeepRitzForm}, the erroneous flux prediction in Dirichlet subproblem (see \autoref{table-Xmp3-graderror-profiles}) is a key threat to the DeepDDM scheme \cite{li2020deep} but not our proposed algorithms, which demonstrates that our methods are robust with respect to the varying coefficients. Furthermore, DNLA (PINNs) achieves a smaller approximation error $\nabla \hat{u}_1^{[k]} - \nabla u_1^{[k]}$ inside the subdomain $\Omega_1$ (see \autoref{table-Xmp3-graderror-profiles}), and therefore performs better than that of DNLA (deep Ritz) (see \autoref{table-Xmp3-relative-errors}).



\section{Conclusion} \label{Section-Conlcusion}

Motivated by the observation that domain decomposition methods are essentially continuous schemes, a novel mesh-free method is developed in this work to realize the classical Dirichlet-Neumann algorithm using neural networks, which paves the way for effectively solving the elliptic interface problem with high-contrast coefficients and irregular interfaces. Instead of a direct substitution of the local problem solver with modern deep learning tools, a rigorous error analysis is first established to obtain the error bound of boundary penalty treatment for both the Dirichlet and Neumann subproblems, and therefore sheds light on the design of training loss functions that can alleviate the burden of erroneous flux prediction along subdomain interfaces. In other words, the accuracy of flux transmission (or derivatives of the network solution at interfaces plays a key role in combining non-overlapping domain decomposition methods with techniques from the deep learning community, which has not yet been well addressed by the existing deep learning-based algorithms and is the focus of this work. Moreover, the tedious tuning of penalty coefficients can also get insights from our theoretical study. A wide variety of numerical examples are carried out to validate the effectiveness and robustness of our proposed methods, achieving promising performance especially in the presence of erroneous flux data during outer iterations. 

We believe that our theoretical and experimental studies can also be generalized to the Robin-Robin algorithm \cite{quarteroni1999domain,na2022domain}, while substantial improvements can be made by employing coarse grid correction \cite{mercier2021coarse}, adaptive sampling strategy \cite{he2022mesh}, special network architectures \cite{li2022deep}, and more comprehensive error analysis \cite{jiao2021error,duan2021analysis}.


\newpage

\appendix

\section{Detailed Proof of Step 1) in \textnormal{\textbf{Theorem \ref{IntfcProb-DN-DirichletSubProb-ErrAnlys}}}} 

Recall that the function $\hat{u}_1\in H^1(\Omega_1)$ is decomposed as a sum of two local functions, \textit{i.e.}, $\hat{u}_1 = \hat{u}_1^{[k]} + g$, then 
\begingroup
\renewcommand*{\arraystretch}{2.5}
\begin{equation*}
\begin{array}{cl}
\mathcal{L}_2(\hat{u}_2) \!\!\!\!\! & = \displaystyle \frac12 b_1(\hat{u}_1^{[k]} + g, \hat{u}_1^{[k]} + g) - (f, \hat{u}_1^{[k]} + g)_1 + \frac{\beta_D}{2} \Big( \lVert \hat{u}_1^{[k]} + g \rVert_{L^2(\partial\Omega_1\cap\partial\Omega)}^2 \\
& \displaystyle \ \ \ +\ \lVert \hat{u}_1^{[k]} + g - u_\Gamma^{[k]} \rVert_{L^2(\Gamma)}^2 \Big) \\
& \displaystyle = \mathcal{L}_1(\hat{u}_1^{[k]}) + b_1( \hat{u}_1^{[k]}, g ) - (f,g)_1 + \frac12 b_1(g, g) + \frac{\beta_D}{2} \lVert g \rVert_{L^2(\partial\Omega_1)}^2 \\
& \displaystyle \ \ \ + \ \beta_D \left( (\hat{u}_1^{[k]}, g)_{L^2(\partial\Omega_1\cap\partial\Omega)} + (\hat{u}_1^{[k]} - u_\Gamma^{[k]}, g)_{L^2(\Gamma)} \right)
\end{array}
\end{equation*}
\endgroup
Note that the function $\hat{u}_1^{[k]}\in H^1(\Omega_1)$ is required to satisfy the equation \eqref{DirichletSubProb-BndryPenalty-StrongForm} in the sense of distributions, that is, for any $g\in H^2(\Omega_1)$,
\begin{equation*}
	b_1(\hat{u}_1^{[k]}, g) + \beta_D \left( (\hat{u}_1^{[k]}, g)_{L^2(\partial\Omega_1\cap\partial\Omega)} + (\hat{u}_1^{[k]} - u_\Gamma^{[k]}, g)_{L^2(\Gamma)} \right)  = (f, g)_1
\end{equation*}
and therefore we arrive at for any $\hat{u}_1\in H^1(\Omega_1)$
\begin{equation*}
\mathcal{L}_1(\hat{u}_1) = \mathcal{L}_1(\hat{u}_1^{[k]}) + \int_{\Omega_1} \left( \frac{c_1}{2}|\nabla g|^2 + \frac12 |g|^2 \right) dx + \frac{\beta_D}{2} \int_{\partial\Omega_1} |g|^2 ds \geq \mathcal{L}_1(\hat{u}_1^{[k]}).
\end{equation*}

\section{Detailed Proof of Step 1) in \textnormal{\textbf{Theorem \ref{IntfcProb-DN-NeumannSubProb-ErrAnlys}}}}

Recall that the function $\hat{u}_2\in H^1(\Omega)$ is decomposed as a sum of two global functions, \textit{i.e.}, $\hat{u}_2 = \hat{u}_2^{[k]} + g$, then 
\begingroup
\renewcommand*{\arraystretch}{2.5}
\begin{equation*}
\begin{array}{cl}
\mathcal{L}_2(\hat{u}_2) \!\!\!\!\! & = \displaystyle \frac12 b_2( \hat{u}_2^{[k]} + g, \hat{u}_2^{[k]} + g)  - (f, \hat{u}_2^{[k]} + g)_2 + b_1( \hat{u}_1^{[k]}, \hat{u}_2^{[k]} + g)  - (f, \hat{u}_2^{[k]} + g)_1 \\
&\displaystyle \ \ \ +\ (q, \hat{u}_2^{[k]} + g)_{L^2(\Gamma)} + \frac{\beta_N}{2} \lVert \hat{u}_2^{[k]} + g \rVert_{L^2(\partial\Omega)}^2 \\
& \displaystyle = \mathcal{L}_2(\hat{u}_2^{[k]}) + b_2( \hat{u}_2^{[k]}, g)  - (f, g)_2 + b_1( \hat{u}_1^{[k]}, g)  - (f, g)_1 + (q, g)_{L^2(\Gamma)} \\
& \displaystyle \ \ \ +\ \beta_N ( \hat{u}_2^{[k]}, g )_{L^2(\partial\Omega)} + \frac12 b_2(g, g) + \frac{\beta_N}{2} \lVert g \rVert_{L^2(\partial\Omega)}^2
\end{array}
\end{equation*}
\endgroup
where $g\in H^1(\Omega)$ is arbitrary and is defined over the entire domain. Note that $\hat{u}_1^{[k]}\in H^1(\Omega_1)$ and $\hat{u}_2^{[k]} |_{\Omega_2} \in H^1(\Omega_2)$ are required to satisfy the equations \eqref{DirichletSubProb-BndryPenalty-StrongForm} and \eqref{NeumannSubProb-BndryPenalty-StrongForm} in the sense of distributions, namely,
\begin{equation*}
	b_1( \hat{u}_1^{[k]}, g_1 ) - ( c_1 \nabla \hat{u}_1^{[k]} \cdot \bm{n}_1, g_1 )_{L^2(\partial\Omega_1)} = (f, g_1)_1\ \ \ \textnormal{for any}\ g_1\in H^1(\Omega_1)
\end{equation*}
and 
\begin{equation*}
	b_2( \hat{u}_2^{[k]}, g_2 ) - ( c_2 \nabla \hat{u}_2^{[k]} \cdot \bm{n}_2, g_2 )_{L^2(\partial\Omega_2)} = (f, g_2)_2\ \ \ \textnormal{for any}\ g_2\in H^1(\Omega_2)
\end{equation*}
respectively, we then have by the boundary conditions imposed in \eqref{NeumannSubProb-BndryPenalty-StrongForm-Extension} and \eqref{NeumannSubProb-BndryPenalty-StrongForm} that for any $\hat{u}_2\in H^1(\Omega)$
\begingroup
\renewcommand*{\arraystretch}{2.5}
\begin{equation*}
\begin{array}{cl}
\mathcal{L}_2(\hat{u}_2) \!\!\!\!\! & = \displaystyle \mathcal{L}_2(\hat{u}_2^{[k]}) + ( c_1 \nabla \hat{u}_1^{[k]} \cdot \bm{n}_1+ \beta_N\hat{u}_2^{[k]}, g)_{L^2(\partial\Omega_1\cap\partial\Omega)} \\
& \displaystyle\ \ \  +\ ( c_2 \nabla \hat{u}_2^{[k]} \cdot \bm{n}_2 + \beta_N\hat{u}_2^{[k]}, g )_{L^2(\partial\Omega_2\cap\partial\Omega)} \\
& \displaystyle \ \ \ +\ ( c_1 \nabla \hat{u}_1^{[k]} \cdot \bm{n}_1 + c_2 \nabla \hat{u}_2^{[k]} \cdot \bm{n}_2 + q, g)_{L^2(\Gamma)} + \frac12 b_2(g, g) + \frac{\beta_N}{2} \lVert g \rVert_{L^2(\partial\Omega)}^2 \\
& \displaystyle = \mathcal{L}_2(\hat{u}_2^{[k]}) + \int_{\Omega_2} \left( \frac{c_2}{2}|\nabla g|^2 + \frac12 |g|^2 \right) dx + \frac{\beta_N}{2} \int_{\partial\Omega} |g|^2 ds \geq \mathcal{L}_2(\hat{u}_2^{[k]}).
\end{array}
\end{equation*}
\endgroup


\section{Detailed Experimental Results}

The hyperparameter configuration used for our numerical experiments is briefly summarized in \autoref{hyperparamter-configuration}, together with the iterative network solutions using different methods for each example.

\begin{table}[htp]
\caption{Hyperparameter configuration used for our experiments.}
\centering
\renewcommand{\arraystretch}{1.1}
\begin{tabular}{ | c || c | c | c | c | }
\hline
\multicolumn{2}{|c|}{ } & \!\! \makecell{Train Datasets \\ ($N_{\Omega_i}$, $N_{D_i}$, $N_\Gamma$)} \!\! & \!\! \makecell{Penalty Coeffs \\ ($\beta_D$, $\beta_N$)} \!\! & \!\! \makecell{Network \\ (depth, width)} \!\! \\
\hline	
\hline
\multirow{3}{*}{\!\!$c_2=10^3$\!\!} & DeepDDM & $(20k,5k,5k)$ & 400 & $(6,50)$ \\ 
\cline{2-5}
& DNLA (PINNs)  & $(20k,5k,5k)$ & $(800,800k)$ & $(6,50)$  \\ 
\cline{2-5}
& \!\!\! DNLA (deep Ritz) \!\!\! & $(20k,5k,5k)$ & $(800,800k)$ & $(6,50)$  \\ 
\hline         
\hline
\multirow{3}{*}{\!\!$c_2=1$\!\!} & DeepDDM & $(20k,5k,5k)$ & 400 & $(6,50)$ \\ 
\cline{2-5}
& DNLA (PINNs) & $(20k,5k,5k)$ & $(800,800)$ & $(6,50)$  \\ 
\cline{2-5}
& \!\!\! DNLA (deep Ritz) \!\!\! & $(20k,5k,5k)$ & $(800,800)$ & $(6,50)$   \\ 
\hline                                             
\end{tabular}
\label{hyperparamter-configuration}
\end{table}

\begin{table}[!htb]
\caption{The iterative solutions $\hat{u}^{[k]}(x,y;\theta)$ for numerical example \eqref{NumXmp1-Circle} with $(c_1,c_2)=(1,10^3)$.}
\centering
\adjustbox{max width=0.85\textwidth}{
\centering
\begin{tabular}{ c  c  c  c }
\makecell{DeepDDM} &
\begin{minipage}{.23\textwidth}
\centering
\includegraphics[width=1\linewidth]{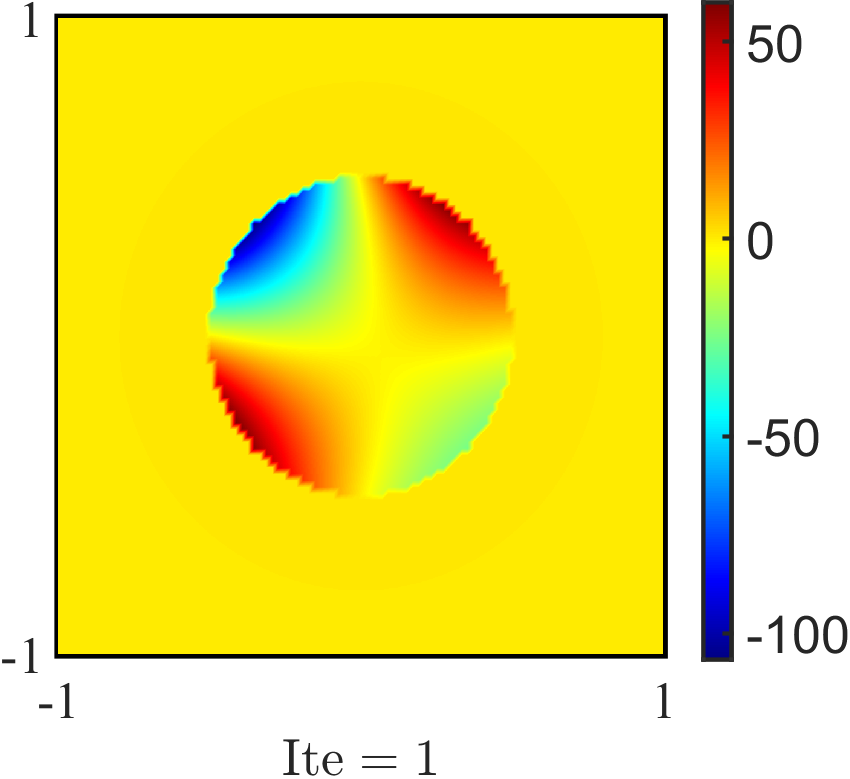}\vspace*{-0.08cm}
\end{minipage}
&
\begin{minipage}{.23\textwidth}
\centering
\includegraphics[width=1\linewidth]{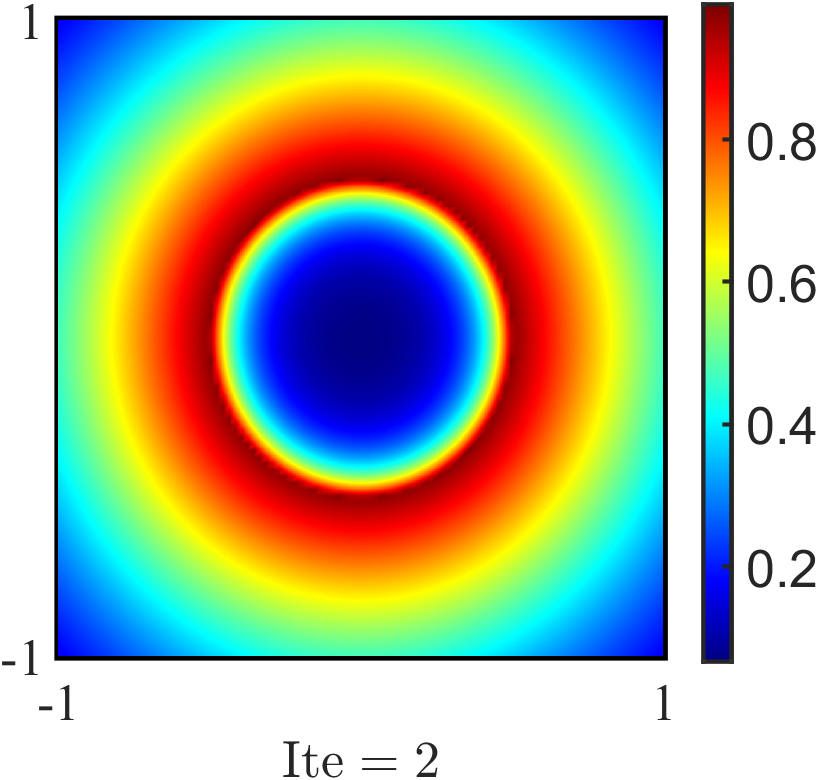}\vspace*{-0.08cm}
\end{minipage}
& 
\begin{minipage}{.23\textwidth}
\centering
\includegraphics[width=1\linewidth]{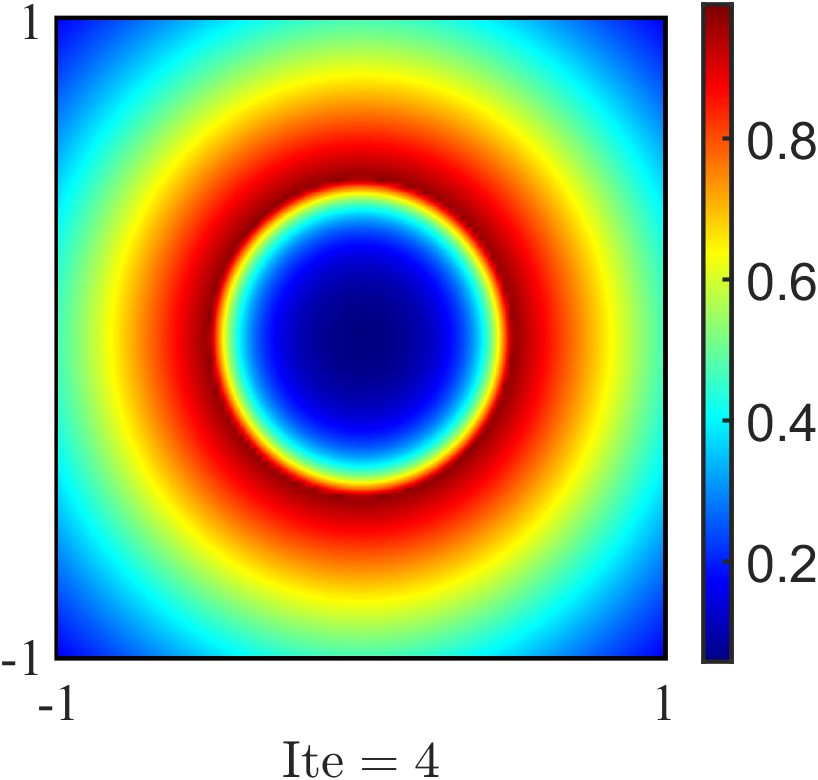}\vspace*{-0.08cm}
\end{minipage}
\\ 
\\
\makecell{DNLA \\ (PINNs)} &
\begin{minipage}{.23\textwidth}
\centering
\includegraphics[width=1\linewidth]{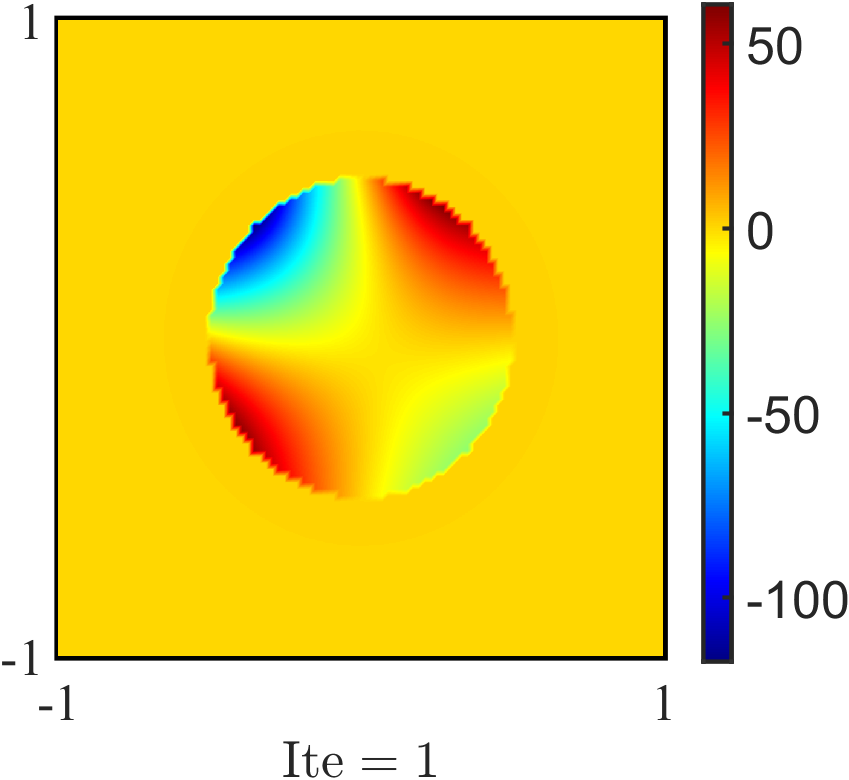}\vspace*{-0.08cm}
\end{minipage}
&
\begin{minipage}{.23\textwidth}
\centering
\includegraphics[width=1\linewidth]{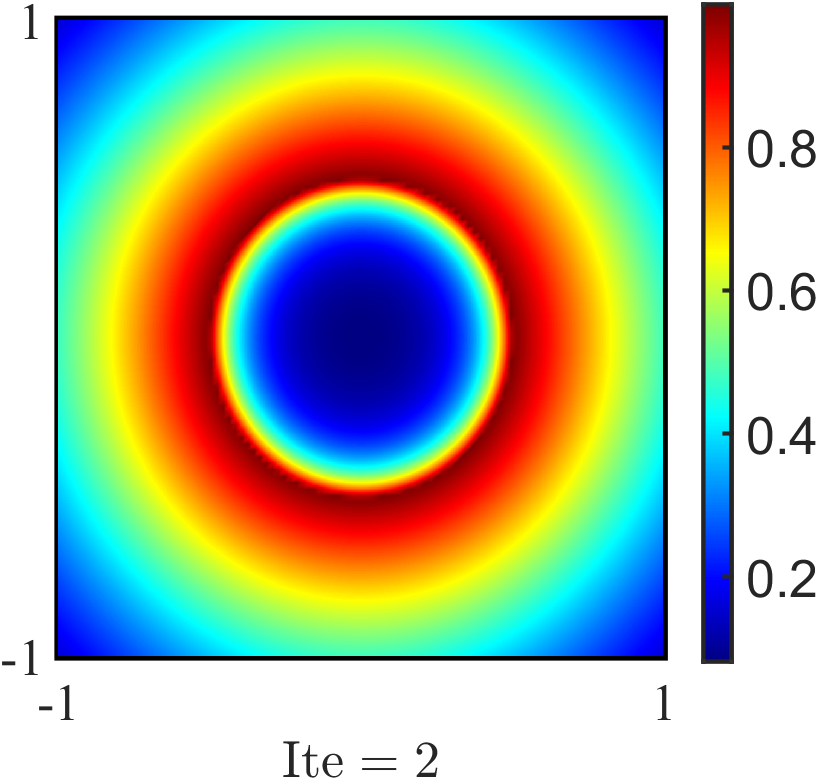}\vspace*{-0.08cm}
\end{minipage}
& 
\begin{minipage}{.23\textwidth}
\centering
\includegraphics[width=1\linewidth]{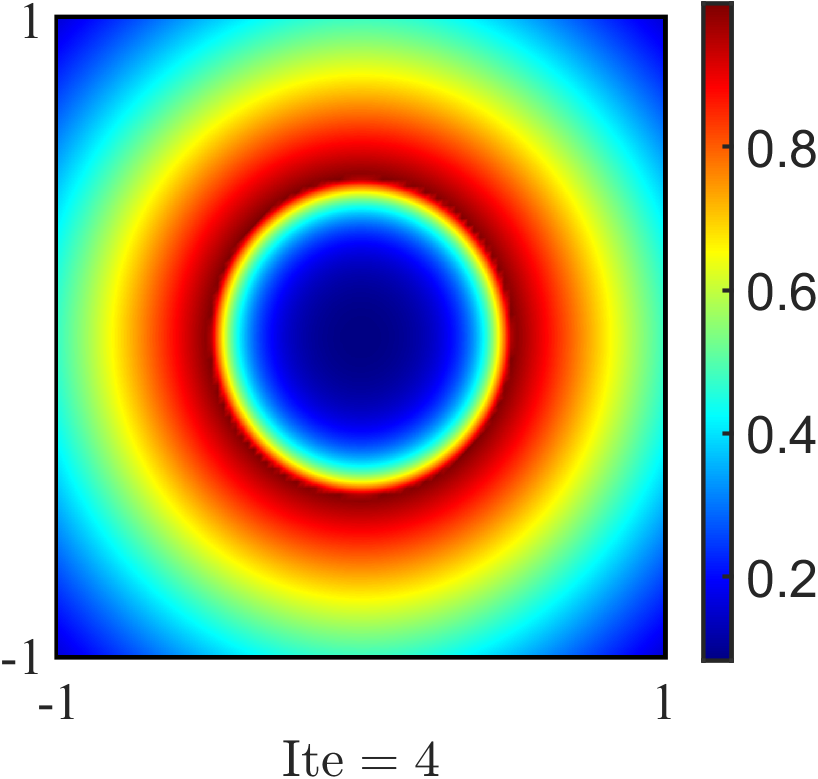}\vspace*{-0.08cm}
\end{minipage}
\\
\\
\makecell{DNLA \\ (deep Ritz)} &
\begin{minipage}{.23\textwidth}
\centering
\includegraphics[width=1.\linewidth]{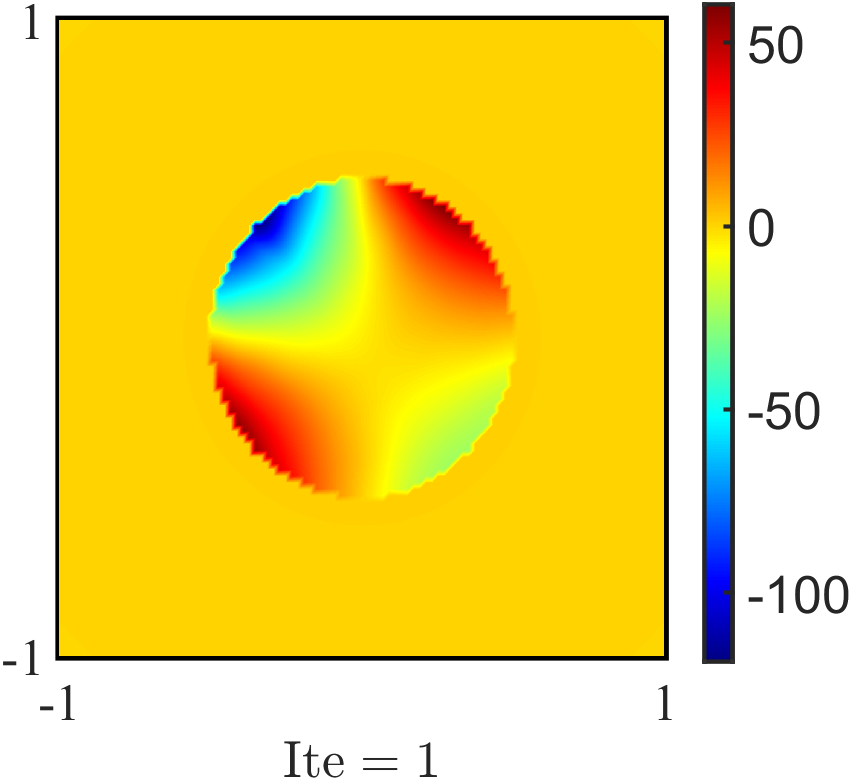}\vspace*{-0.08cm}
\end{minipage}
&
\begin{minipage}{.23\textwidth}
\centering
\includegraphics[width=1\linewidth]{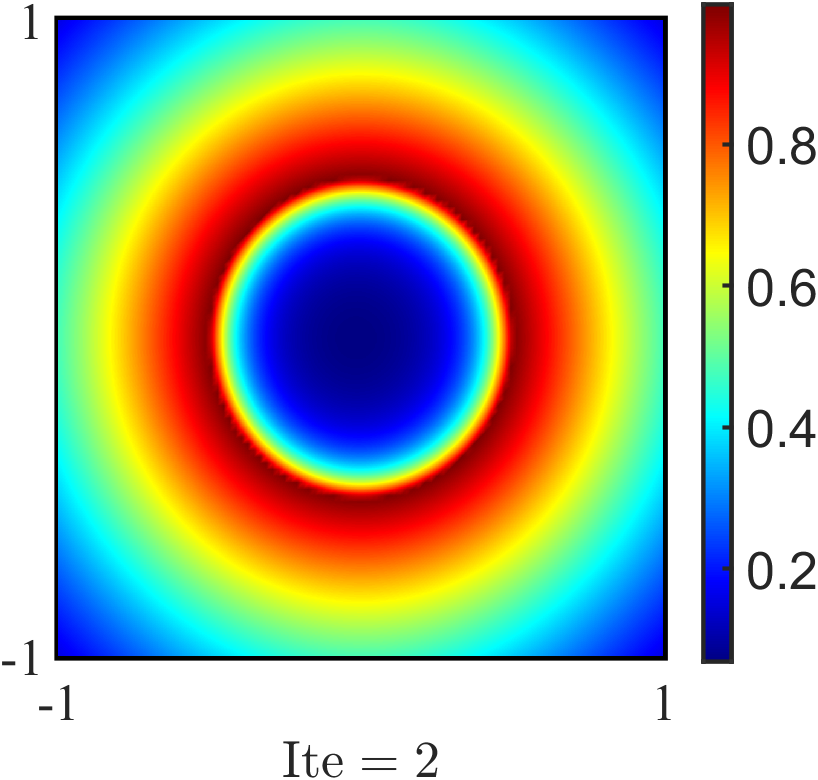}\vspace*{-0.08cm}
\end{minipage}
& 
\begin{minipage}{.23\textwidth}
\centering
\includegraphics[width=1\linewidth]{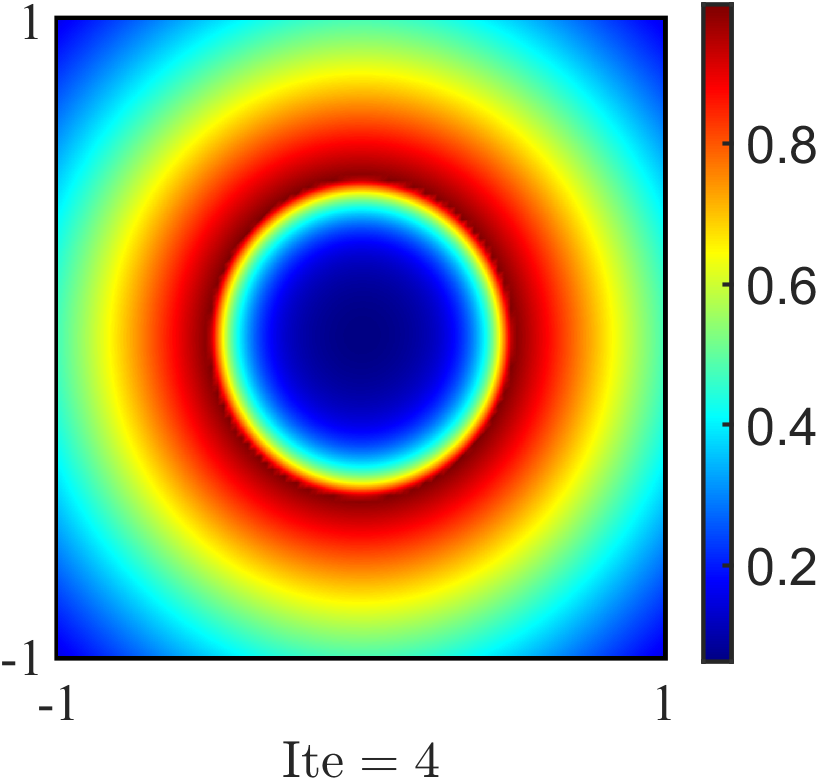}\vspace*{-0.08cm}
\end{minipage}
\end{tabular}}
\end{table}

\hfill

\begin{table}[!htb]
\caption{The iterative solutions $\hat{u}^{[k]}(x,y;\theta)$ for numerical example \eqref{NumXmp1-Circle} with $(c_1,c_2)=(1,1)$.}
\centering
\adjustbox{max width=\textwidth}{
\centering
\begin{tabular}{ c c  c  c  c }
\makecell{DeepDDM} &
\begin{minipage}{.23\textwidth}
\centering
\includegraphics[width=1\linewidth]{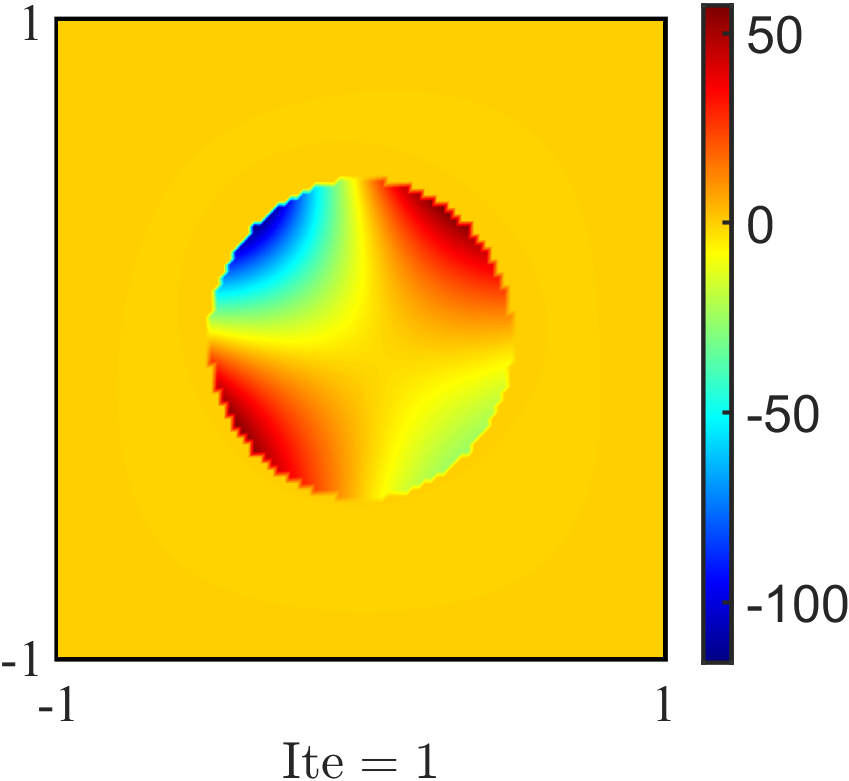}\vspace*{-0.08cm}
\end{minipage}
&
\begin{minipage}{.23\textwidth}
\centering
\includegraphics[width=1\linewidth]{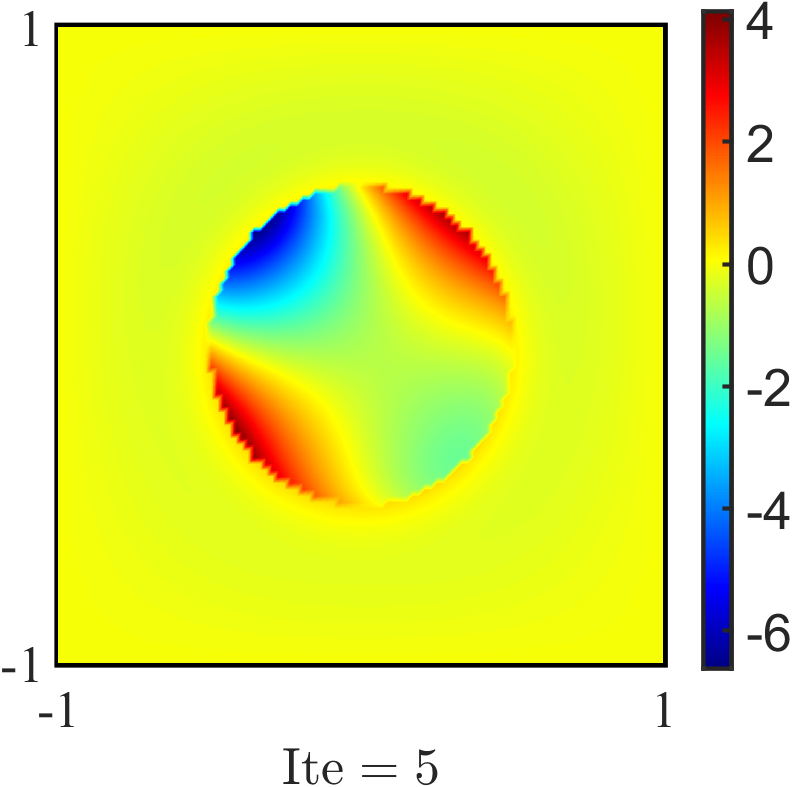}\vspace*{-0.08cm}
\end{minipage}
& 
\begin{minipage}{.23\textwidth}
\centering
\includegraphics[width=1\linewidth]{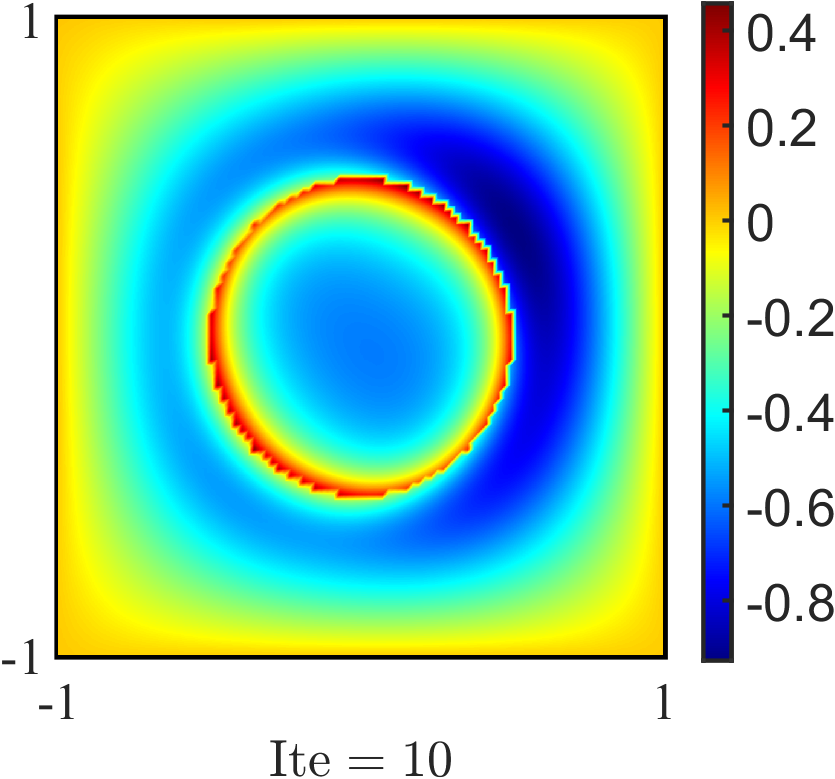}\vspace*{-0.08cm}
\end{minipage}
& 
\begin{minipage}{.23\textwidth}
\centering
\includegraphics[width=1\linewidth]{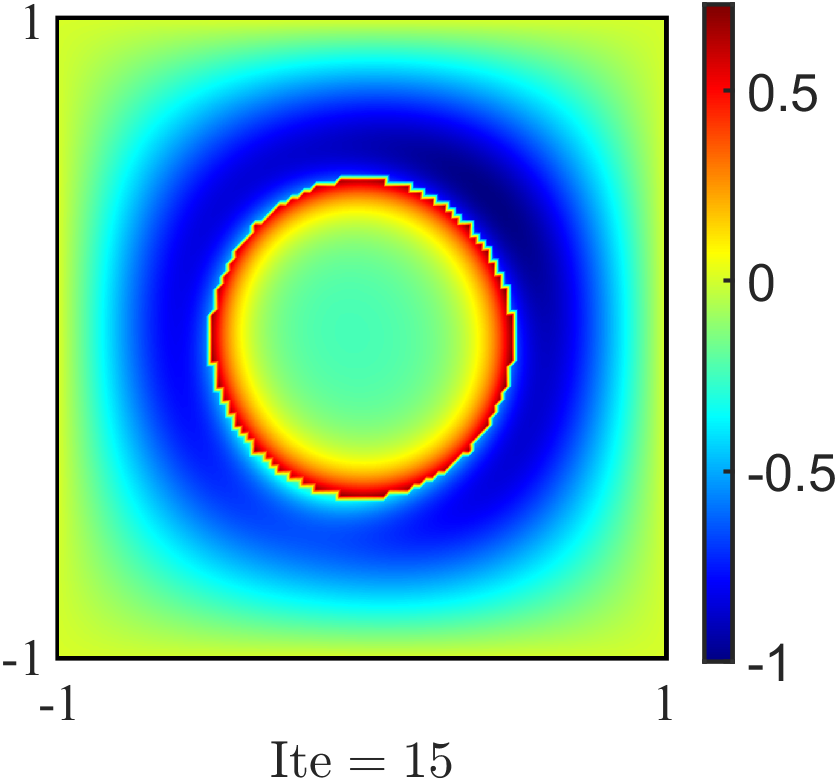}\vspace*{-0.08cm}
\end{minipage}
\\ 
\\
\makecell{DNLA \\ (PINNs)} &
\begin{minipage}{.23\textwidth}
\centering
\includegraphics[width=1\linewidth]{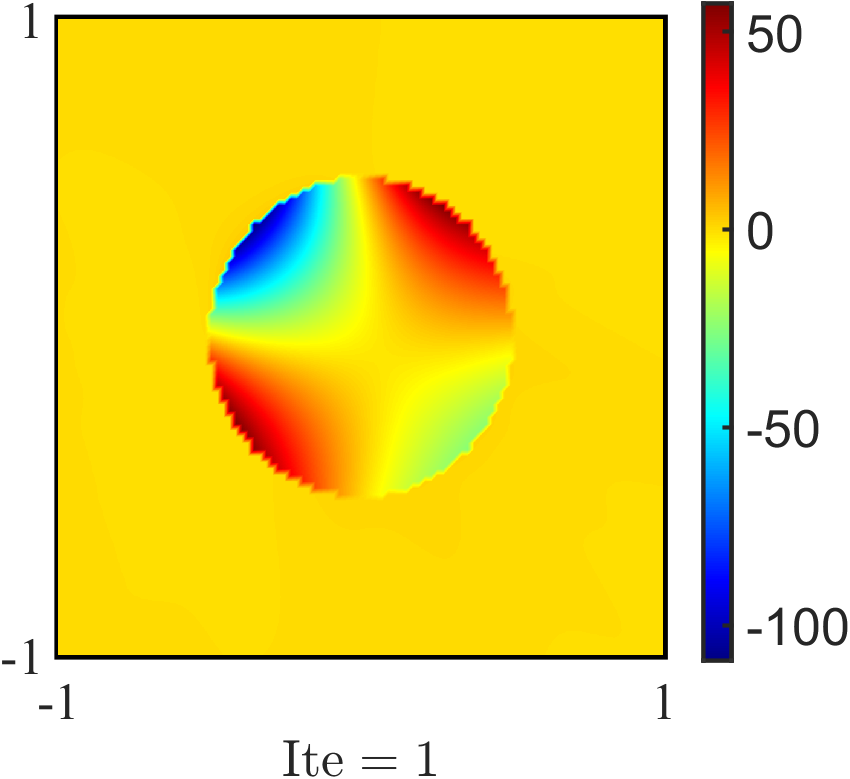}\vspace*{-0.08cm}
\end{minipage}
&
\begin{minipage}{.23\textwidth}
\centering
\includegraphics[width=1\linewidth]{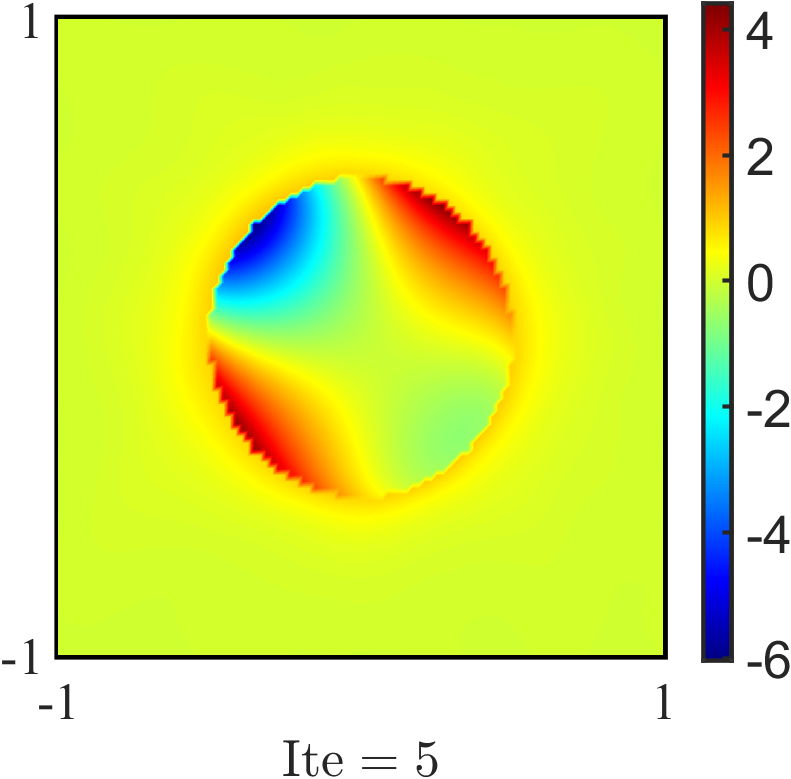}\vspace*{-0.08cm}
\end{minipage}
& 
\begin{minipage}{.23\textwidth}
\centering
\includegraphics[width=1\linewidth]{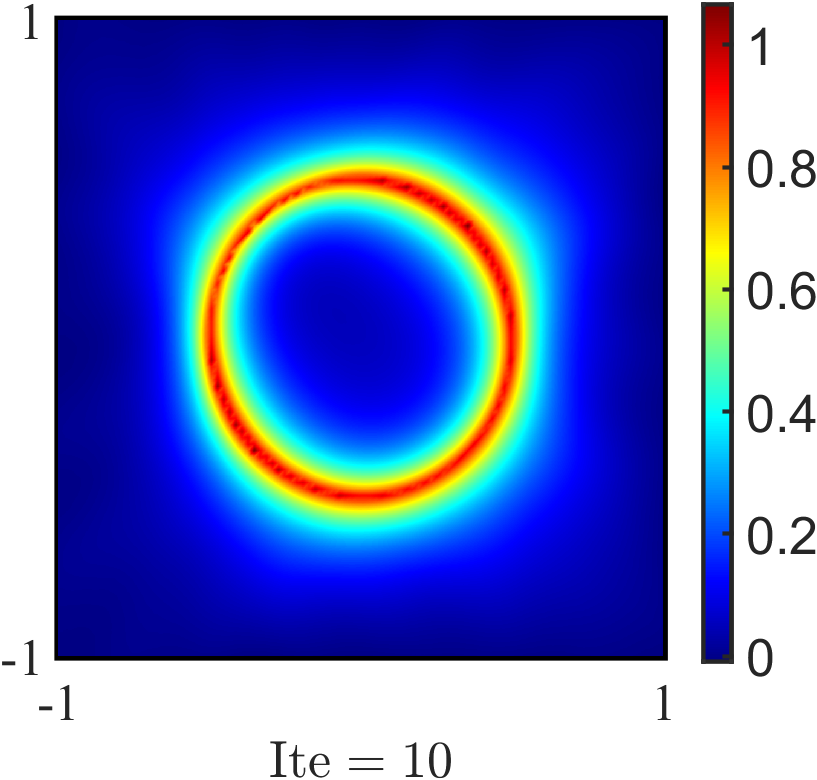}\vspace*{-0.08cm}
\end{minipage}
& 
\begin{minipage}{.23\textwidth}
\centering
\includegraphics[width=1\linewidth]{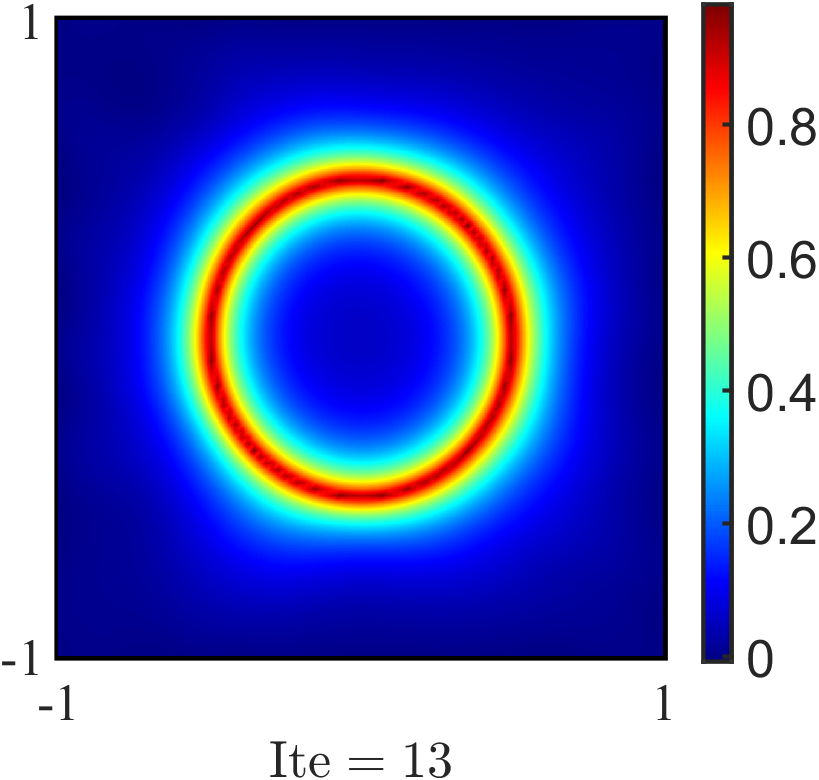}\vspace*{-0.08cm}
\end{minipage}
\\
\\
\makecell{DNLA \\ (deep Ritz)} &
\begin{minipage}{.23\textwidth}
\centering
\includegraphics[width=1.\linewidth]{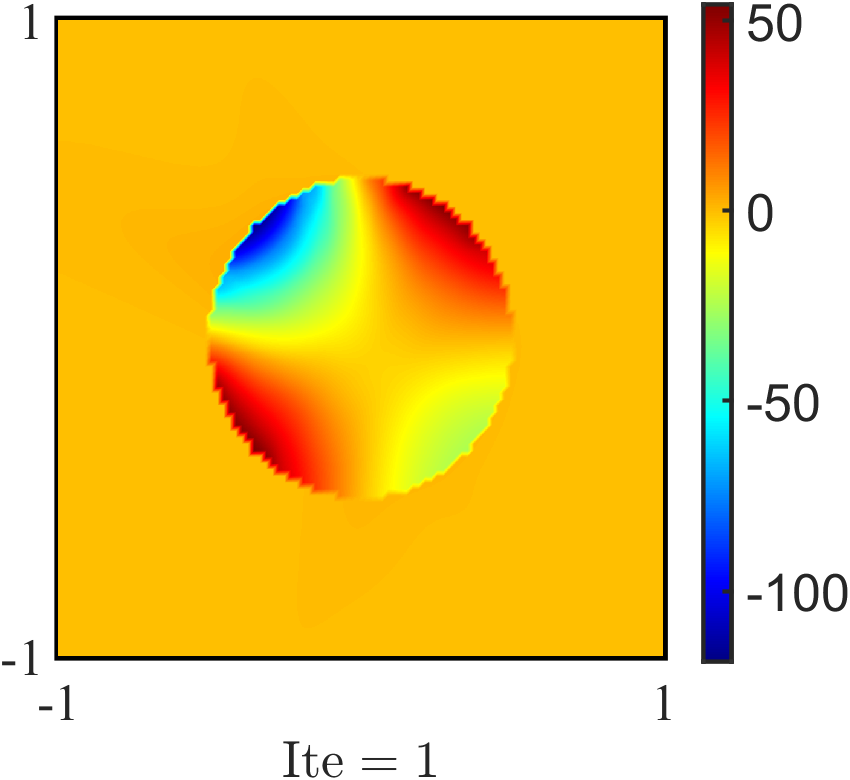}\vspace*{-0.08cm}
\end{minipage}
&
\begin{minipage}{.23\textwidth}
\centering
\includegraphics[width=1\linewidth]{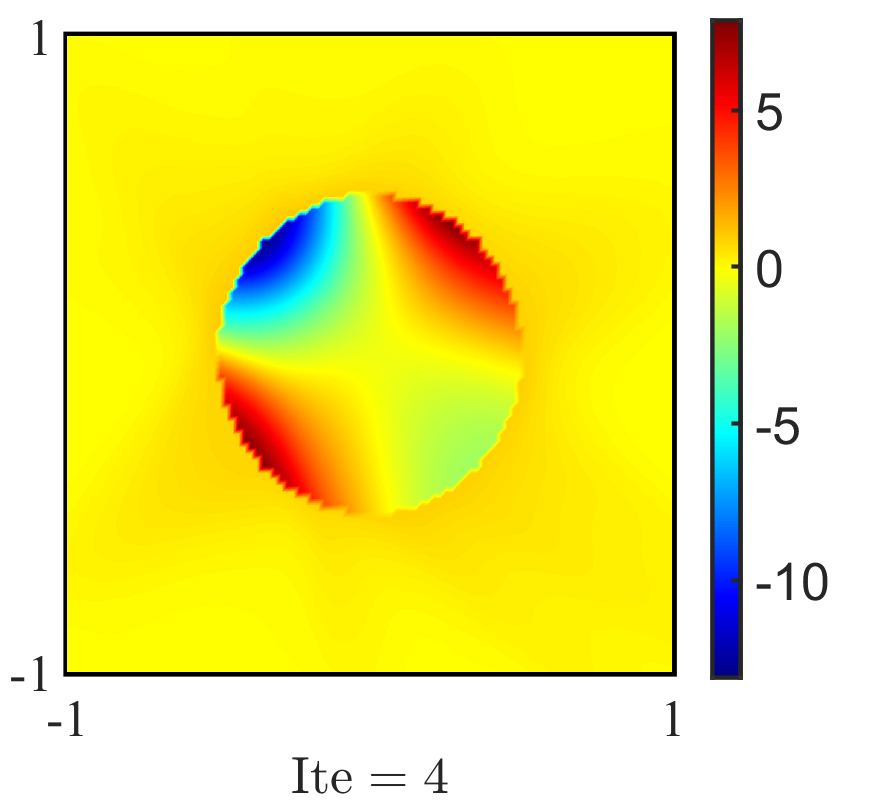}\vspace*{-0.08cm}
\end{minipage}
& 
\begin{minipage}{.23\textwidth}
\centering
\includegraphics[width=1\linewidth]{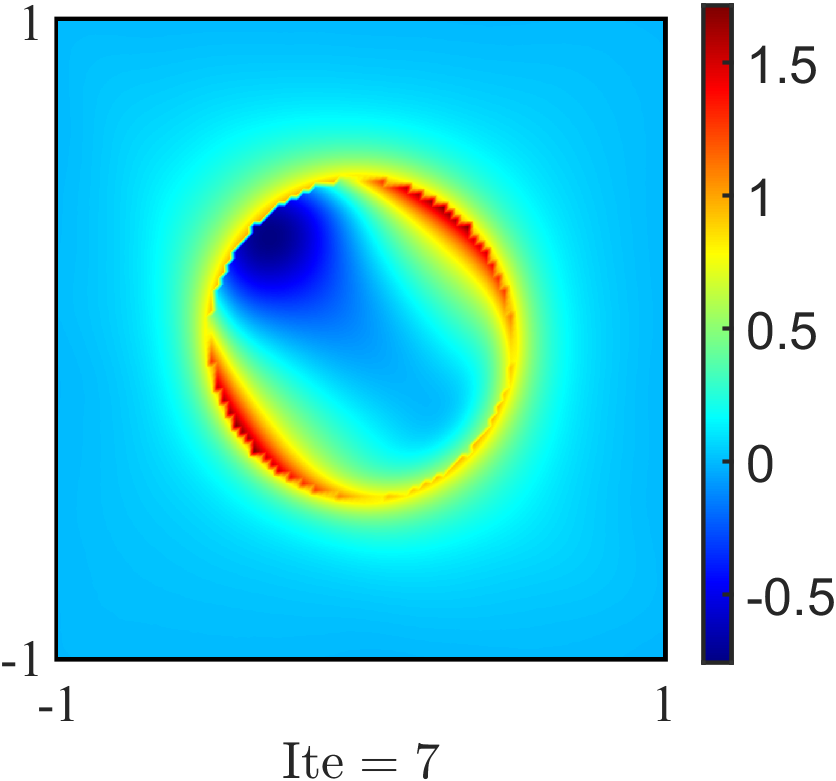}\vspace*{-0.08cm}
\end{minipage}
& 
\begin{minipage}{.23\textwidth}
\centering
\includegraphics[width=1\linewidth]{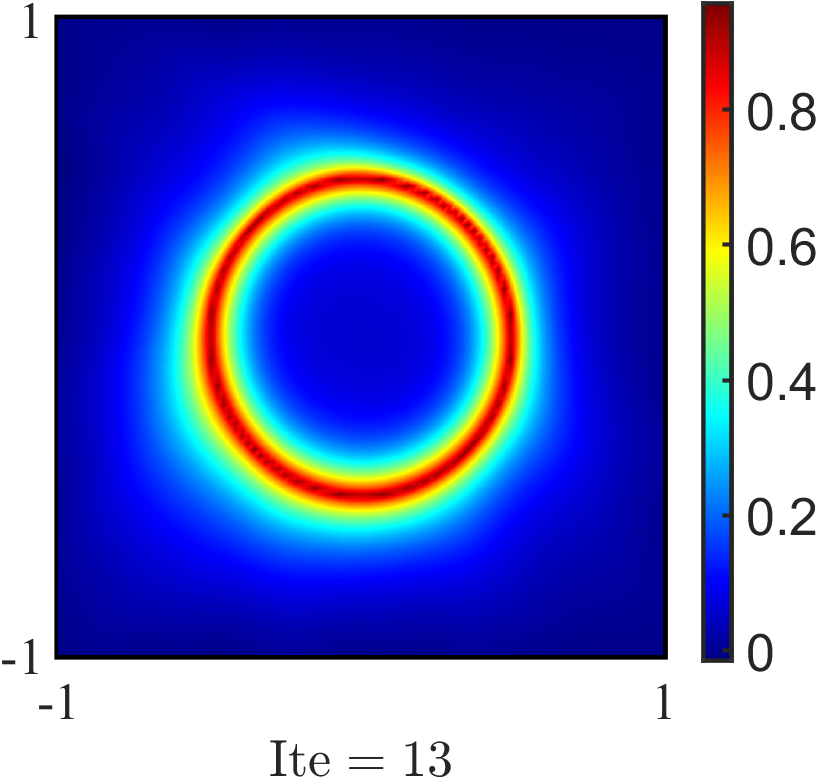}\vspace*{-0.08cm}
\end{minipage}
\end{tabular}}
\end{table}

\vspace{0.5cm}

\begin{table}[!htb]
\caption{The iterative solutions $\hat{u}^{[k]}(x,y;\theta)$ for numerical example \eqref{NumXmp2-Zigzag} with $(c_1,c_2)=(1,10^3)$.}
\centering
\adjustbox{max width=0.85\textwidth}{
\centering
\begin{tabular}{ c c  c  c }
\makecell{DeepDDM} &
\begin{minipage}{.23\textwidth}
\centering
\includegraphics[width=1\linewidth]{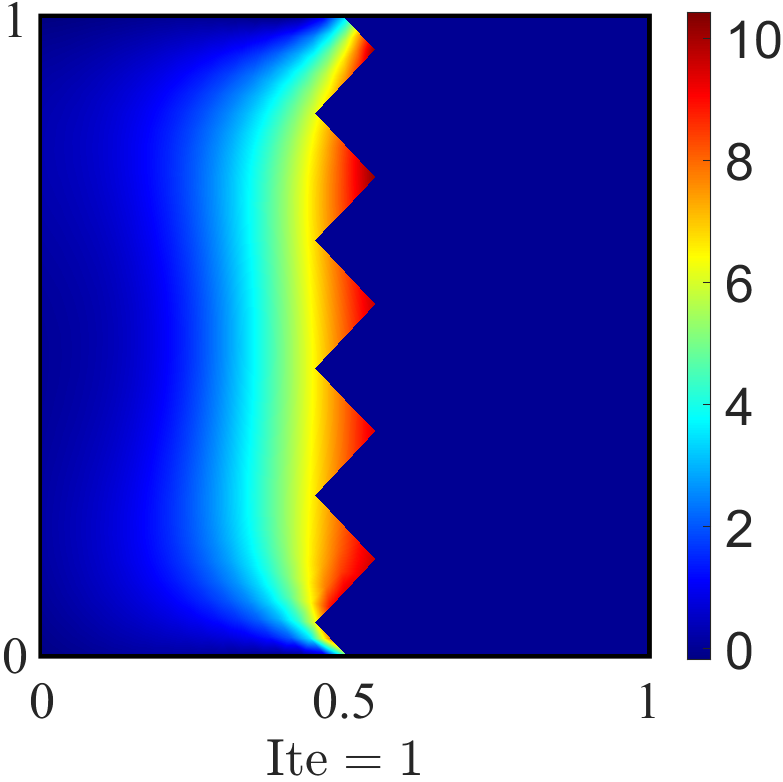}\vspace*{-0.08cm}
\end{minipage}
&
\begin{minipage}{.23\textwidth}
\centering
\includegraphics[width=1\linewidth]{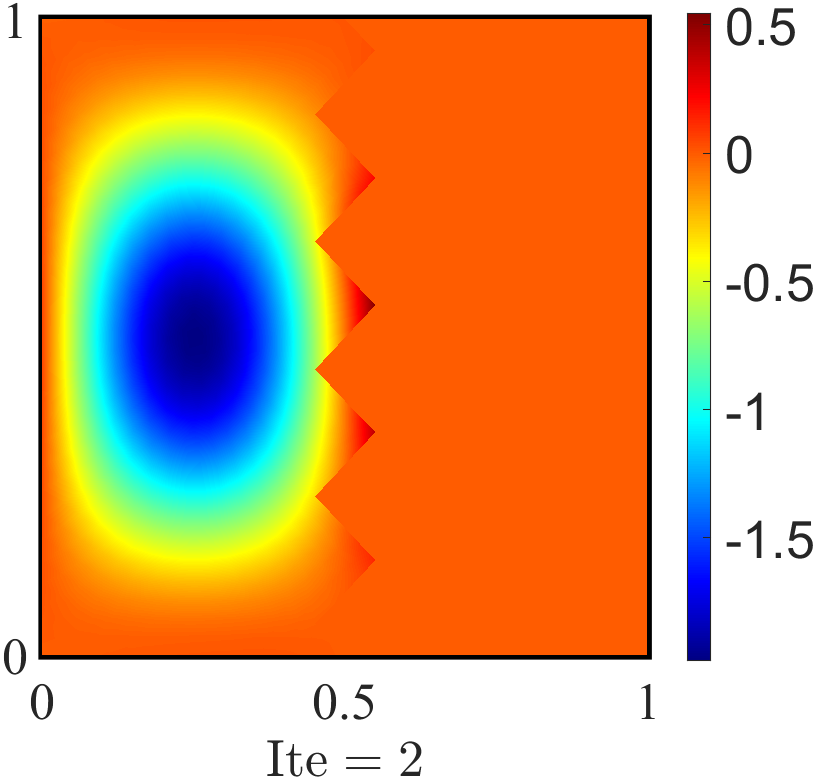}\vspace*{-0.08cm}
\end{minipage}
& 
\begin{minipage}{.23\textwidth}
\centering
\includegraphics[width=1\linewidth]{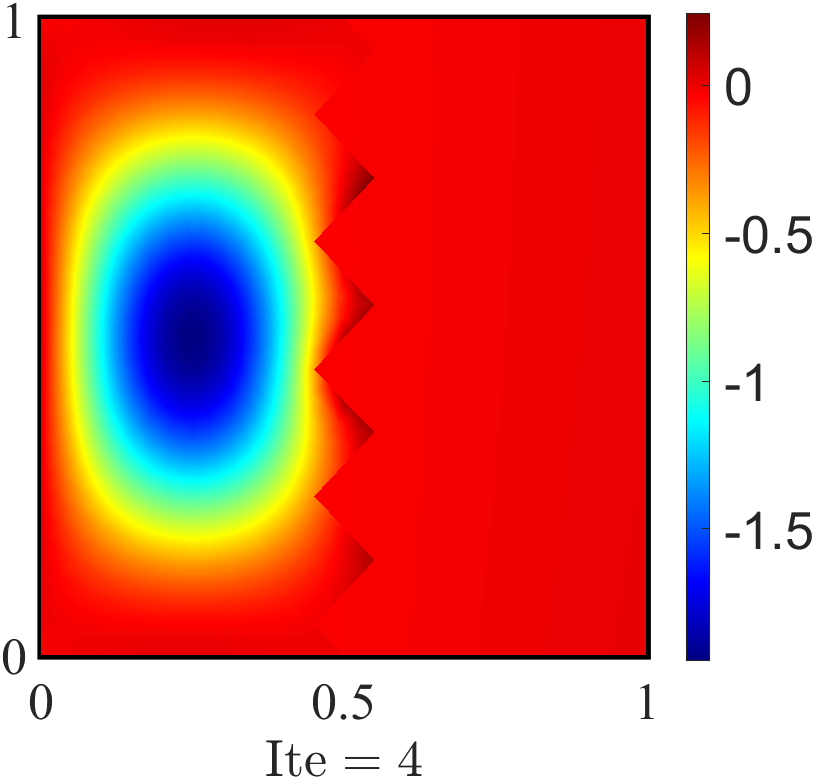}\vspace*{-0.08cm}
\end{minipage}
\\ 
\\
\makecell{DNLA \\ (PINNs)} &
\begin{minipage}{.23\textwidth}
\centering
\includegraphics[width=1\linewidth]{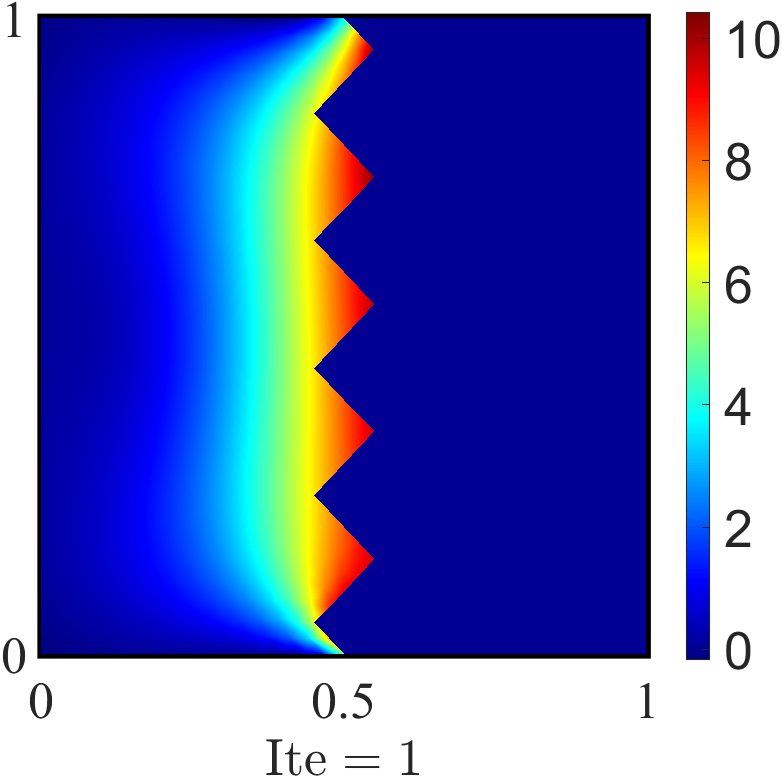}\vspace*{-0.08cm}
\end{minipage}
&
\begin{minipage}{.23\textwidth}
\centering
\includegraphics[width=1\linewidth]{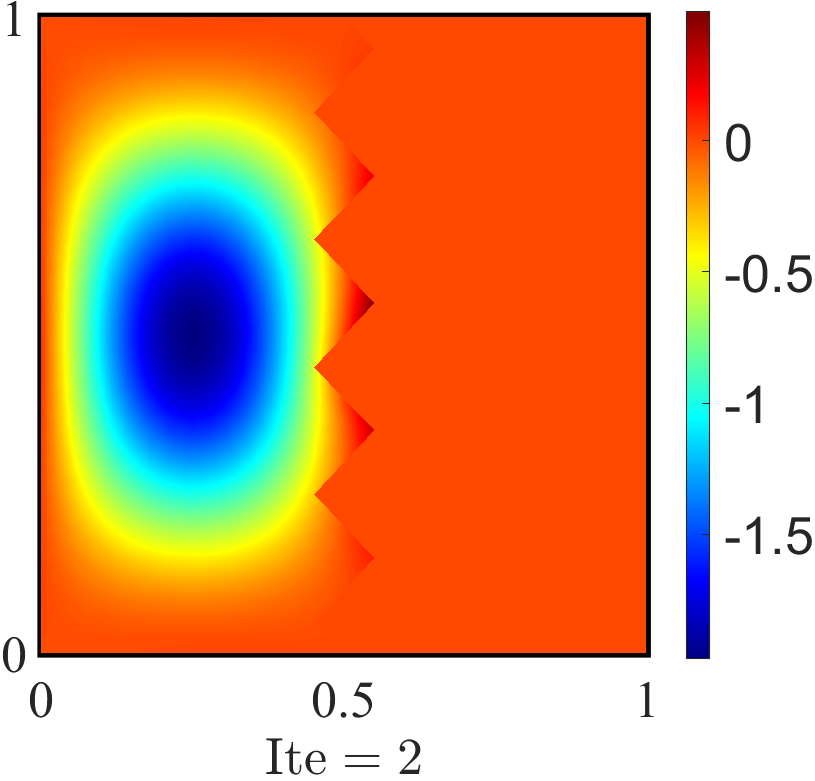}\vspace*{-0.08cm}
\end{minipage}
& 
\begin{minipage}{.23\textwidth}
\centering
\includegraphics[width=1\linewidth]{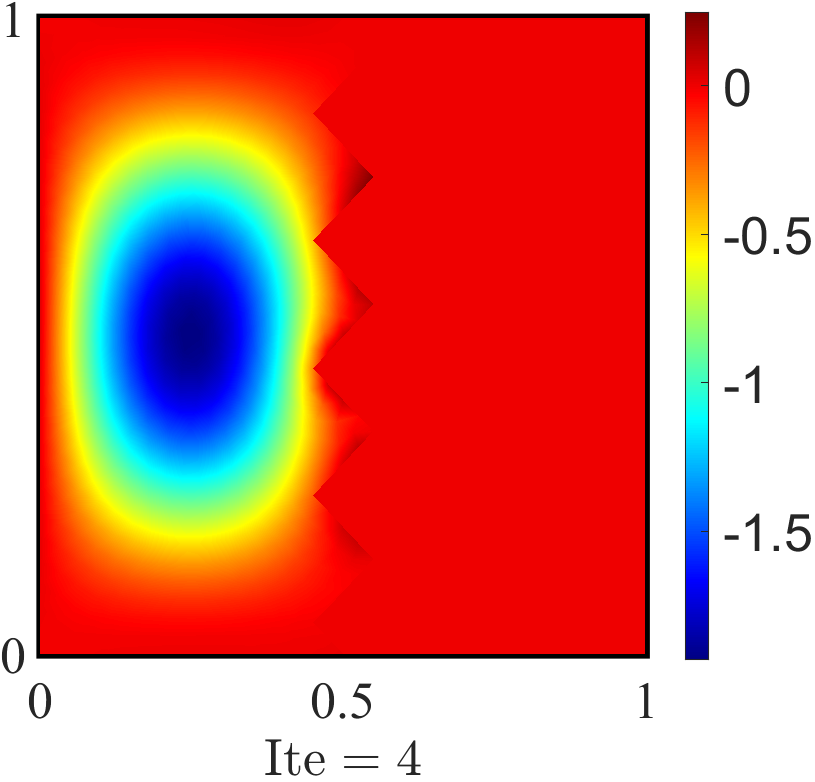}\vspace*{-0.08cm}
\end{minipage}
\\
\\
\makecell{DNLA \\ (deep Ritz)} &
\begin{minipage}{.23\textwidth}
\centering
\includegraphics[width=1\linewidth]{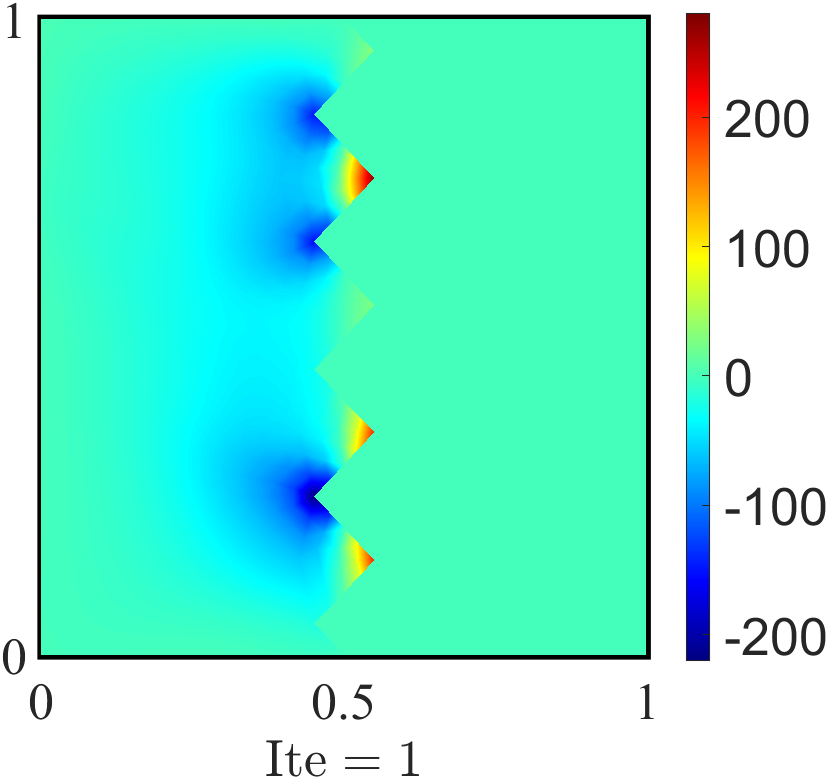}\vspace*{-0.08cm}
\end{minipage}
&
\begin{minipage}{.23\textwidth}
\centering
\includegraphics[width=1\linewidth]{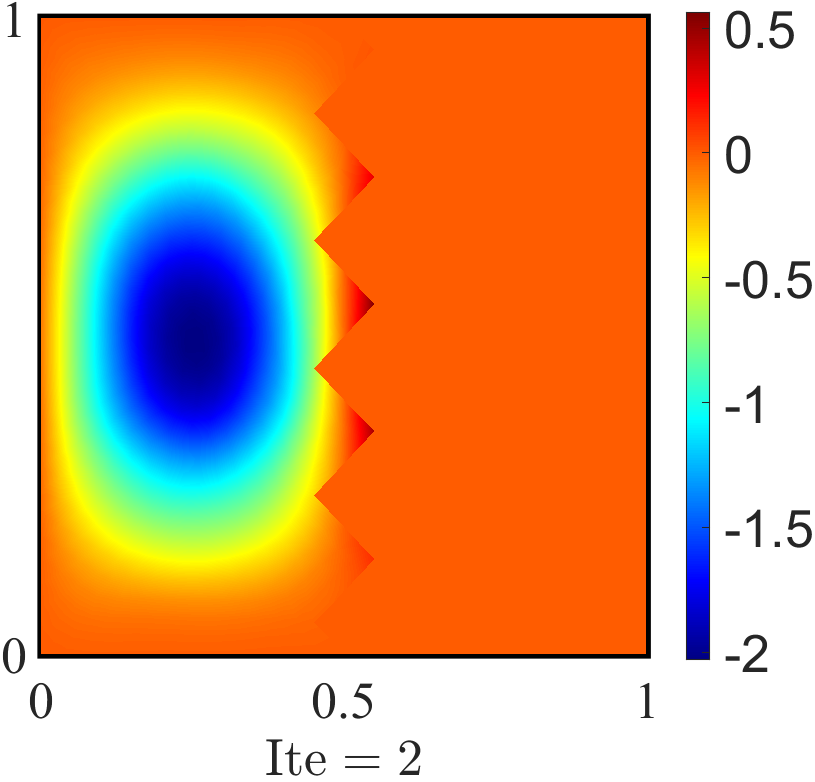}\vspace*{-0.08cm}
\end{minipage}
& 
\begin{minipage}{.23\textwidth}
\centering
\includegraphics[width=1\linewidth]{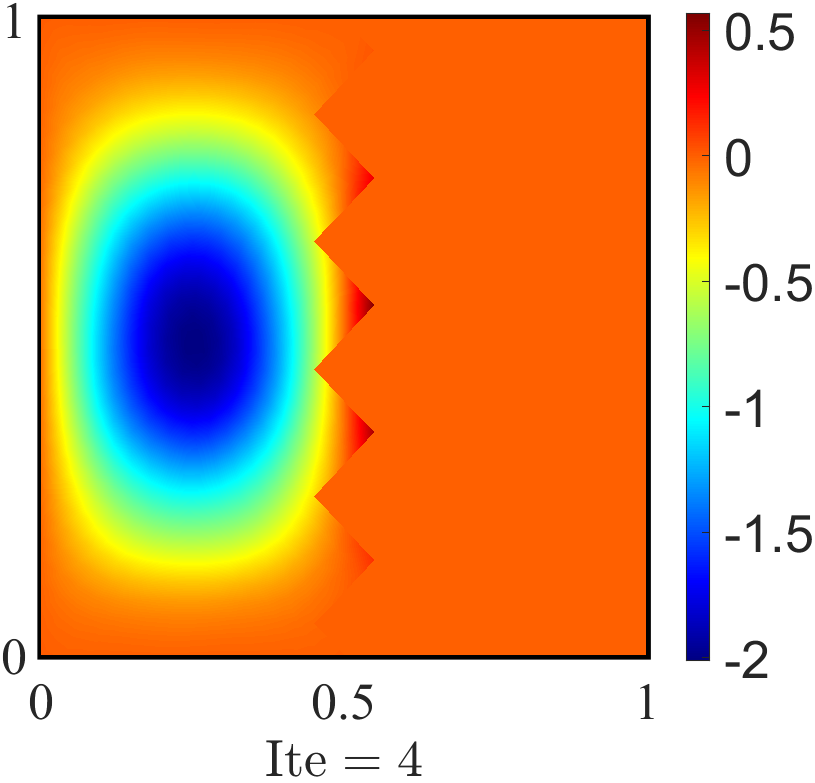}\vspace*{-0.08cm}
\end{minipage}
\end{tabular}}
\end{table}

\hfill

\begin{table}[!htb]
\caption{The iterative solutions $\hat{u}^{[k]}(x,y;\theta)$ for numerical example \eqref{NumXmp2-Zigzag} with $(c_1,c_2)=(1,1)$.}
\centering
\adjustbox{max width=\textwidth}{
\centering
\begin{tabular}{ c c  c  c  c }
\makecell{DeepDDM} &
\begin{minipage}{.23\textwidth}
\centering
\includegraphics[width=1\linewidth]{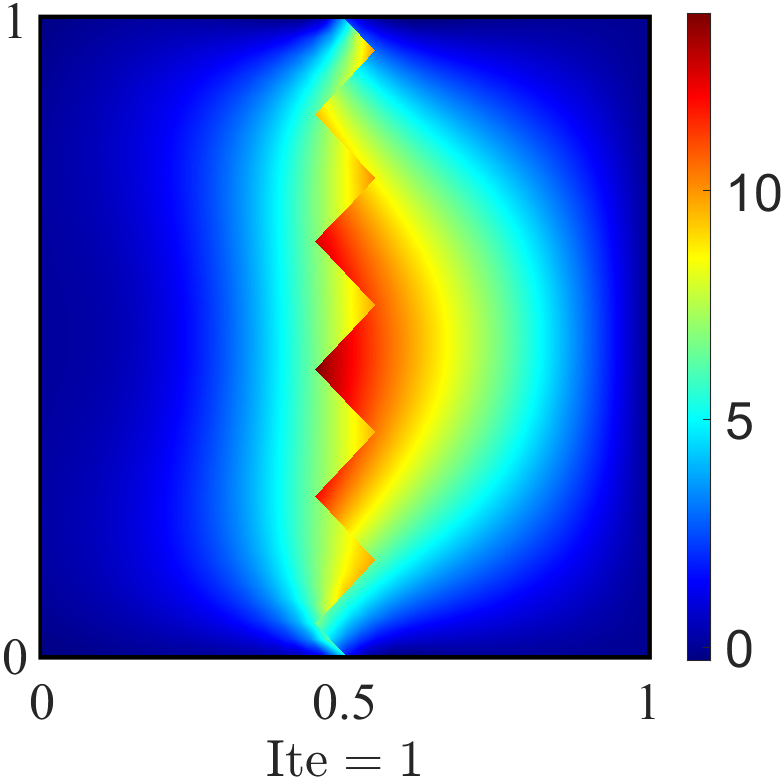}\vspace*{-0.08cm}
\end{minipage}
&
\begin{minipage}{.23\textwidth}
\centering
\includegraphics[width=1\linewidth]{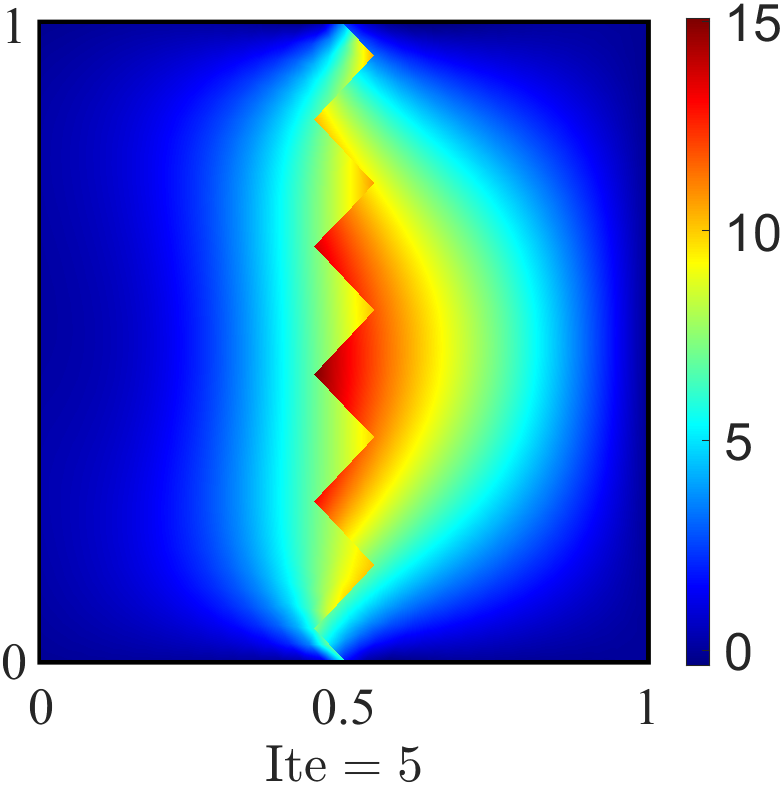}\vspace*{-0.08cm}
\end{minipage}
& 
\begin{minipage}{.23\textwidth}
\centering
\includegraphics[width=1\linewidth]{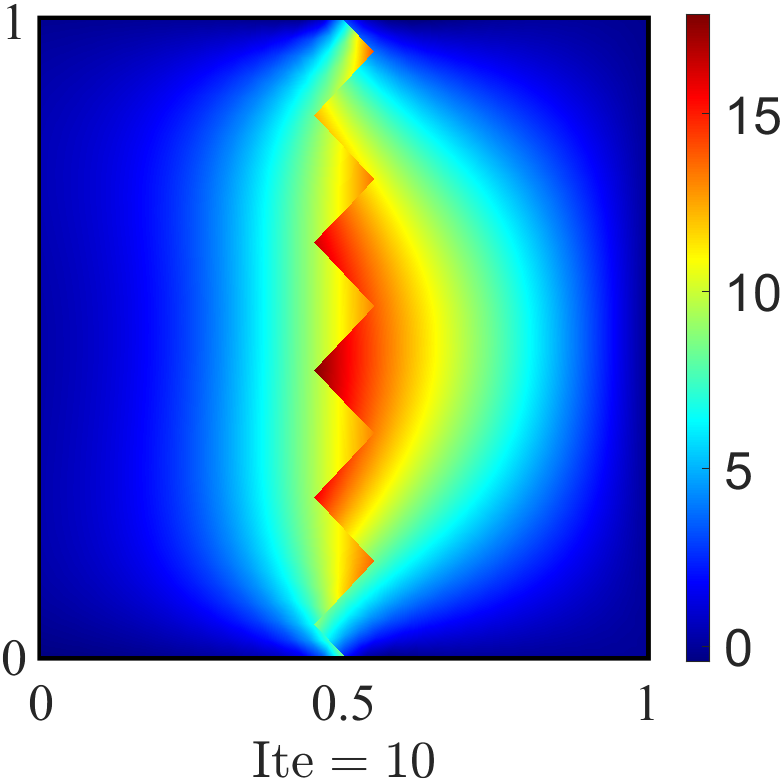}\vspace*{-0.08cm}
\end{minipage}
& 
\begin{minipage}{.23\textwidth}
\centering
\includegraphics[width=1\linewidth]{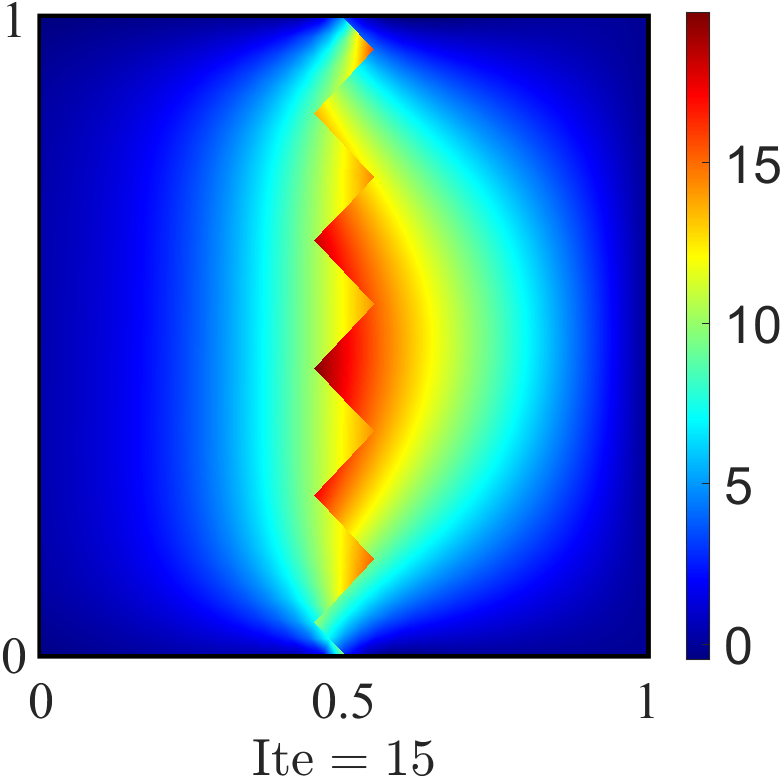}\vspace*{-0.08cm}
\end{minipage}
\\ 
\\
\makecell{DNLA \\ (PINNs)} &
\begin{minipage}{.23\textwidth}
\centering
\includegraphics[width=1\linewidth]{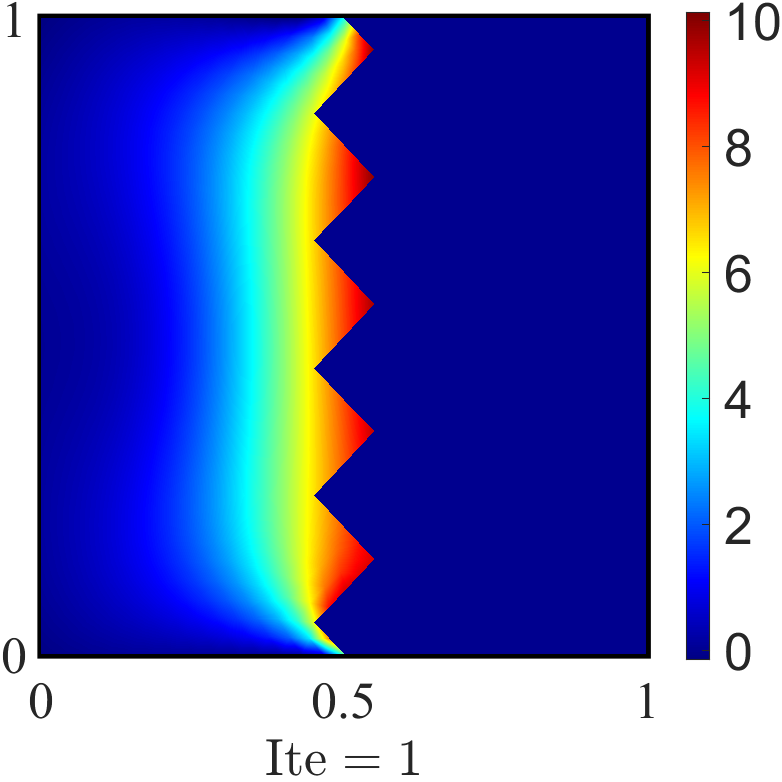}\vspace*{-0.08cm}
\end{minipage}
&
\begin{minipage}{.23\textwidth}
\centering
\includegraphics[width=1\linewidth]{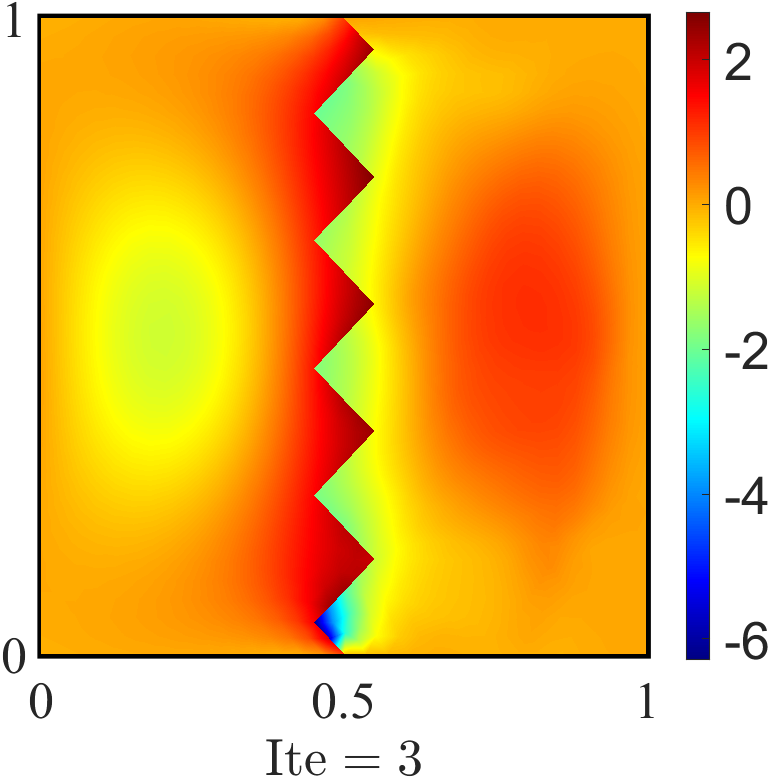}\vspace*{-0.08cm}
\end{minipage}
& 
\begin{minipage}{.23\textwidth}
\centering
\includegraphics[width=1\linewidth]{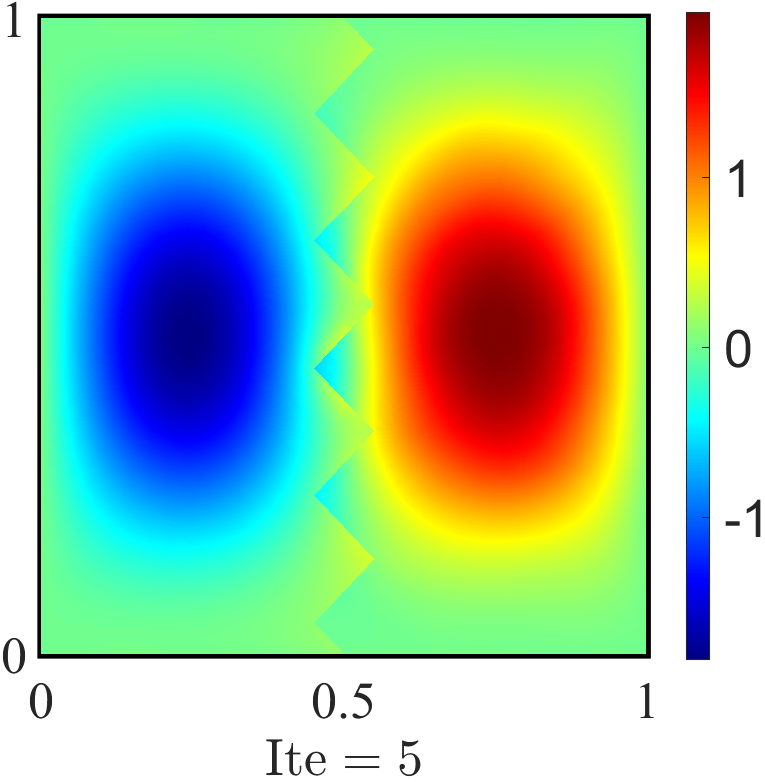}\vspace*{-0.08cm}
\end{minipage}
& 
\begin{minipage}{.23\textwidth}
\centering
\includegraphics[width=1\linewidth]{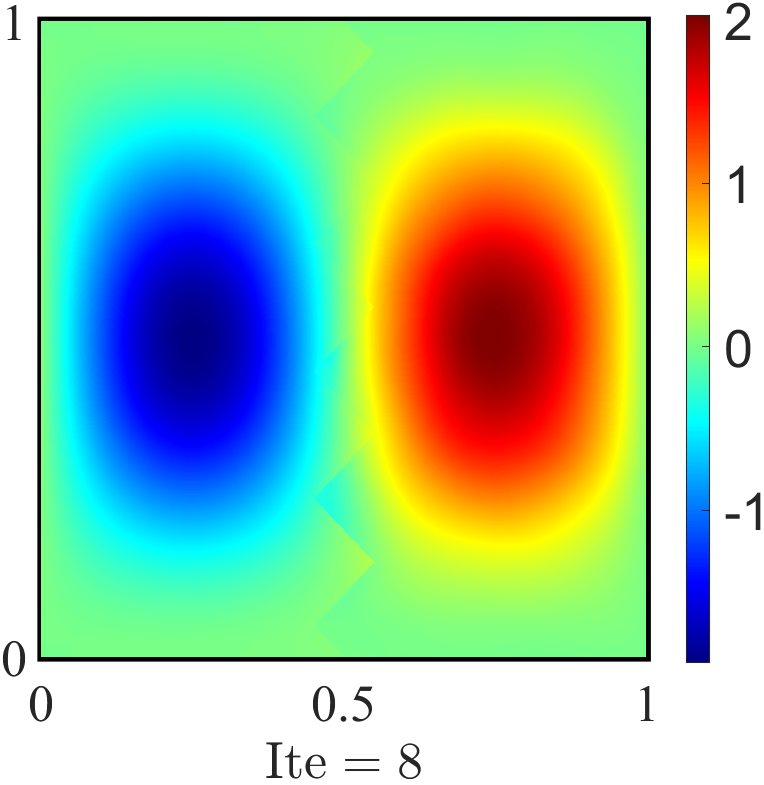}\vspace*{-0.08cm}
\end{minipage}
\\
\\
\makecell{DNLA \\ (deep Ritz)} &
\begin{minipage}{.23\textwidth}
\centering
\includegraphics[width=1\linewidth]{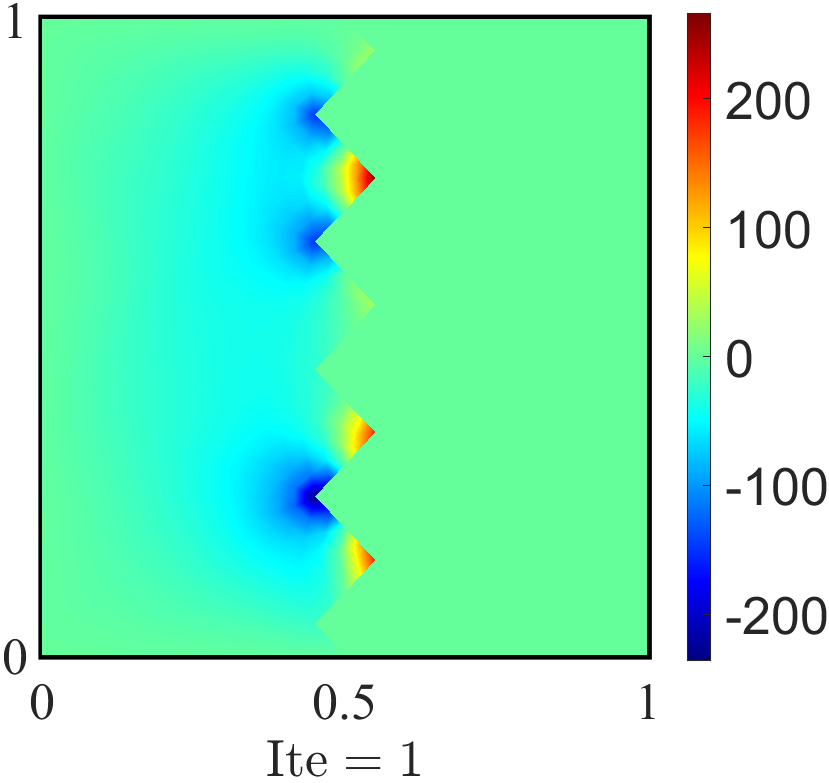}\vspace*{-0.08cm}
\end{minipage}
&
\begin{minipage}{.23\textwidth}
\centering
\includegraphics[width=1\linewidth]{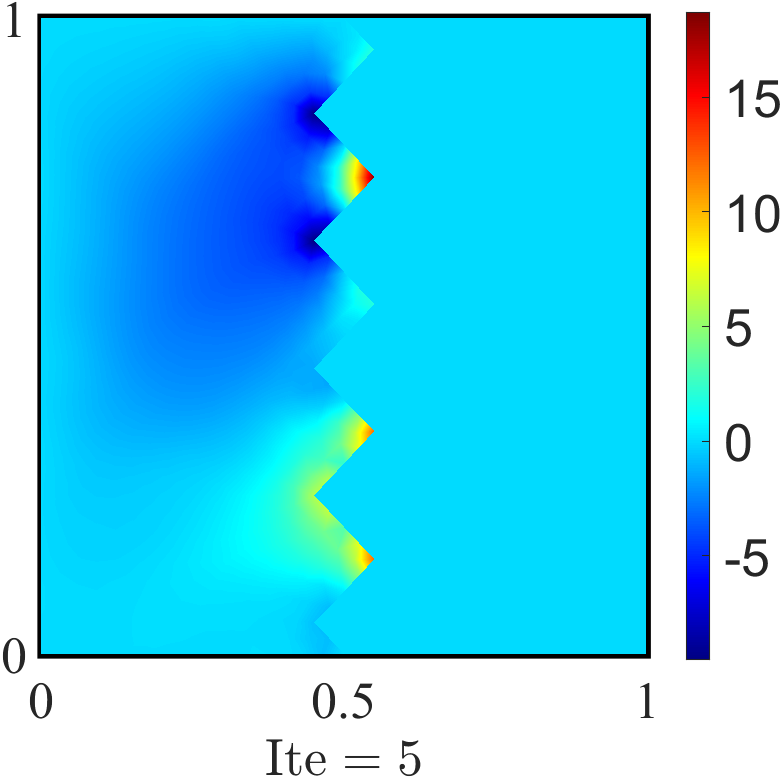}\vspace*{-0.08cm}
\end{minipage}
& 
\begin{minipage}{.23\textwidth}
\centering
\includegraphics[width=1\linewidth]{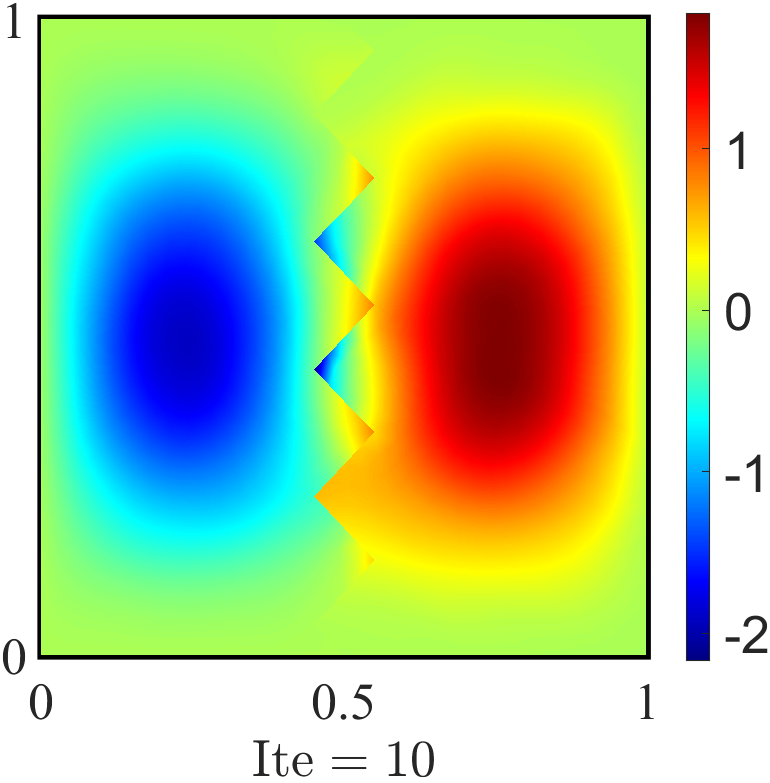}\vspace*{-0.08cm}
\end{minipage}
& 
\begin{minipage}{.23\textwidth}
\centering
\includegraphics[width=1\linewidth]{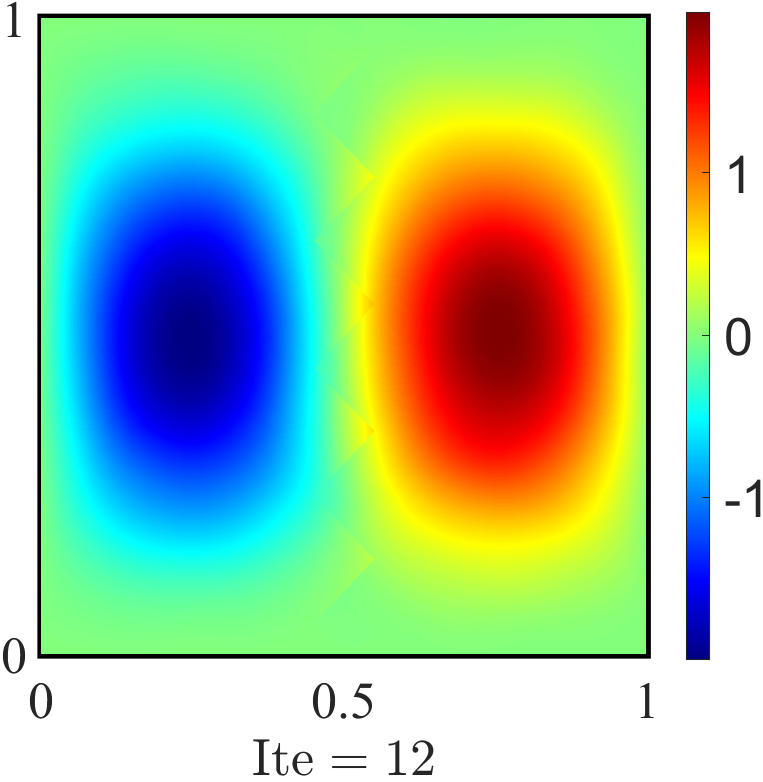}\vspace*{-0.08cm}
\end{minipage}
\end{tabular}}
\end{table}

\vspace{0.5cm}

\begin{table}[!htb]
\caption{The iterative solutions $\hat{u}^{[k]}(x,y;\theta)$ for numerical example \eqref{NumXmp3-Checkerboard} with $(c_1,c_2)=(1,10^3)$.}
\centering
\adjustbox{max width=0.85\textwidth}{
\centering
\begin{tabular}{ c c  c  c }
\makecell{DeepDDM} &
\begin{minipage}{.23\textwidth}
\centering
\includegraphics[width=1\linewidth]{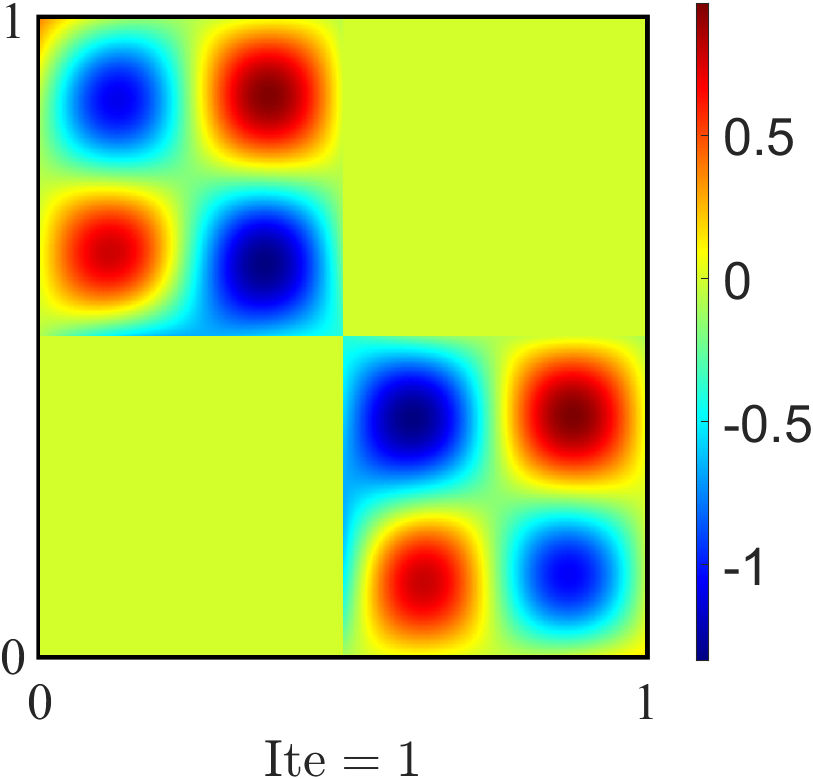}\vspace*{-0.08cm}
\end{minipage}
&
\begin{minipage}{.23\textwidth}
\centering
\includegraphics[width=1\linewidth]{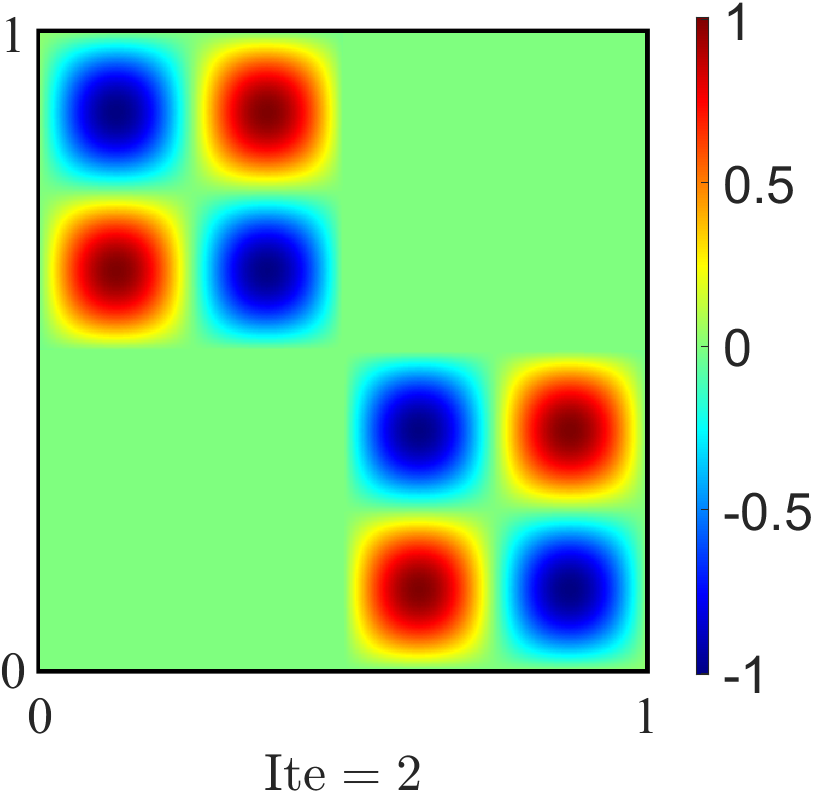}\vspace*{-0.08cm}
\end{minipage}
& 
\begin{minipage}{.23\textwidth}
\centering
\includegraphics[width=1\linewidth]{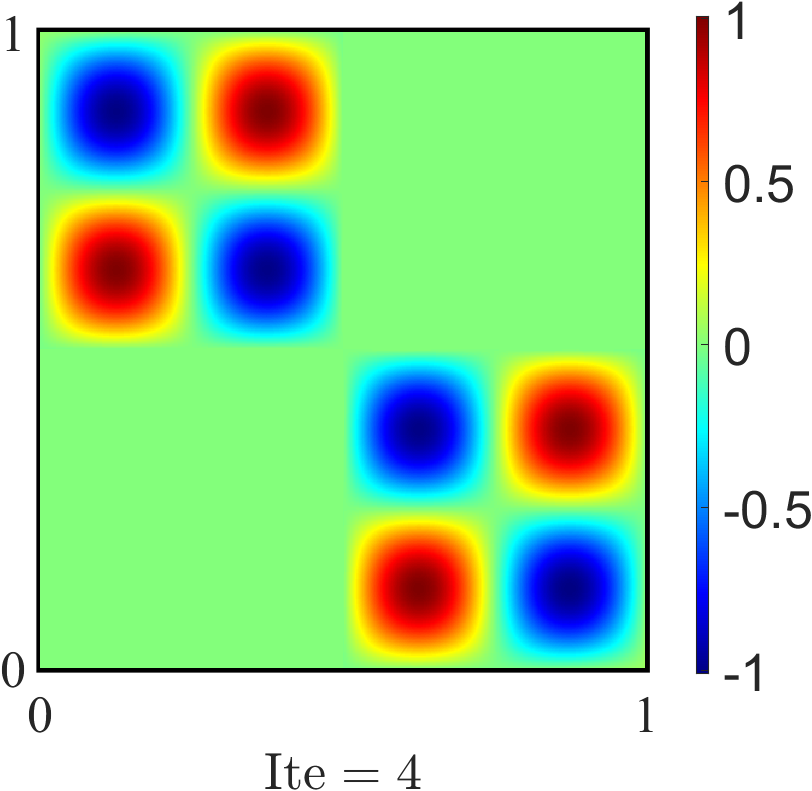}\vspace*{-0.08cm}
\end{minipage}
\\ 
\\
\makecell{DNLA \\ (PINNs)} &
\begin{minipage}{.23\textwidth}
\centering
\includegraphics[width=1\linewidth]{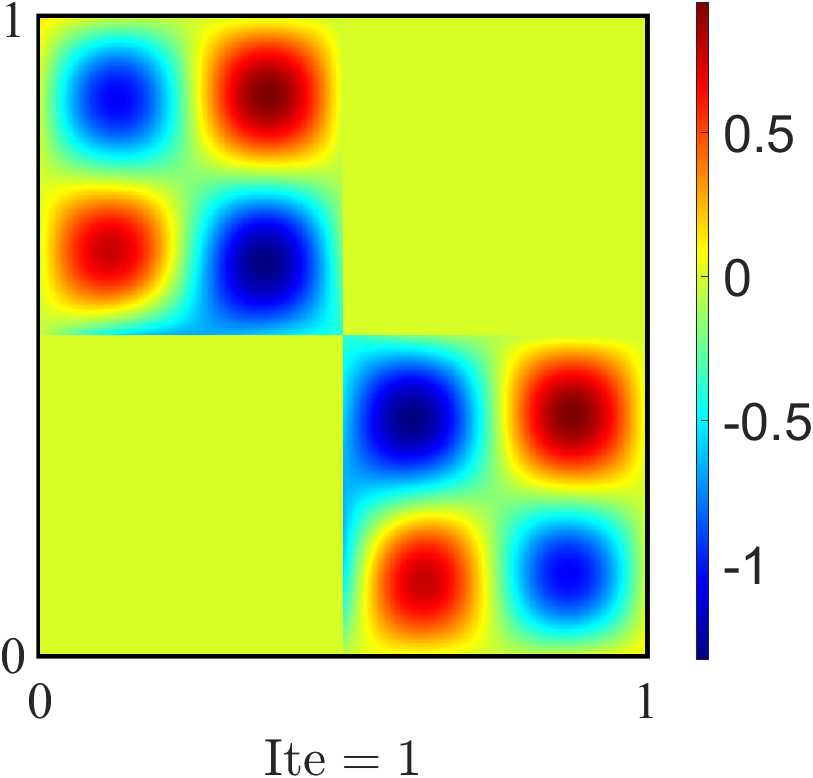}\vspace*{-0.08cm}
\end{minipage}
&
\begin{minipage}{.23\textwidth}
\centering
\includegraphics[width=1\linewidth]{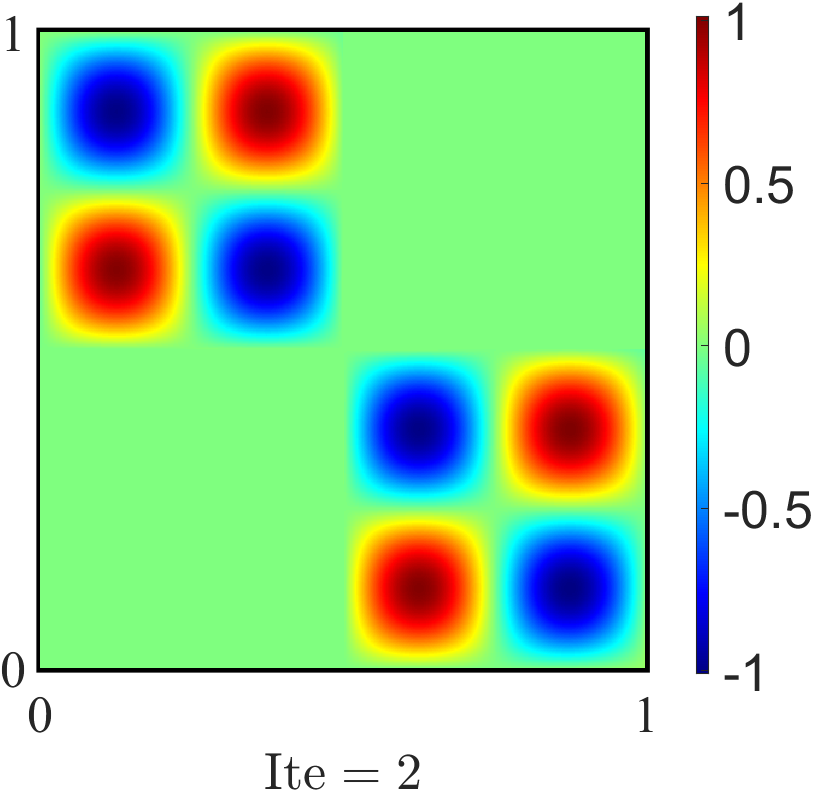}\vspace*{-0.08cm}
\end{minipage}
& 
\begin{minipage}{.23\textwidth}
\centering
\includegraphics[width=1\linewidth]{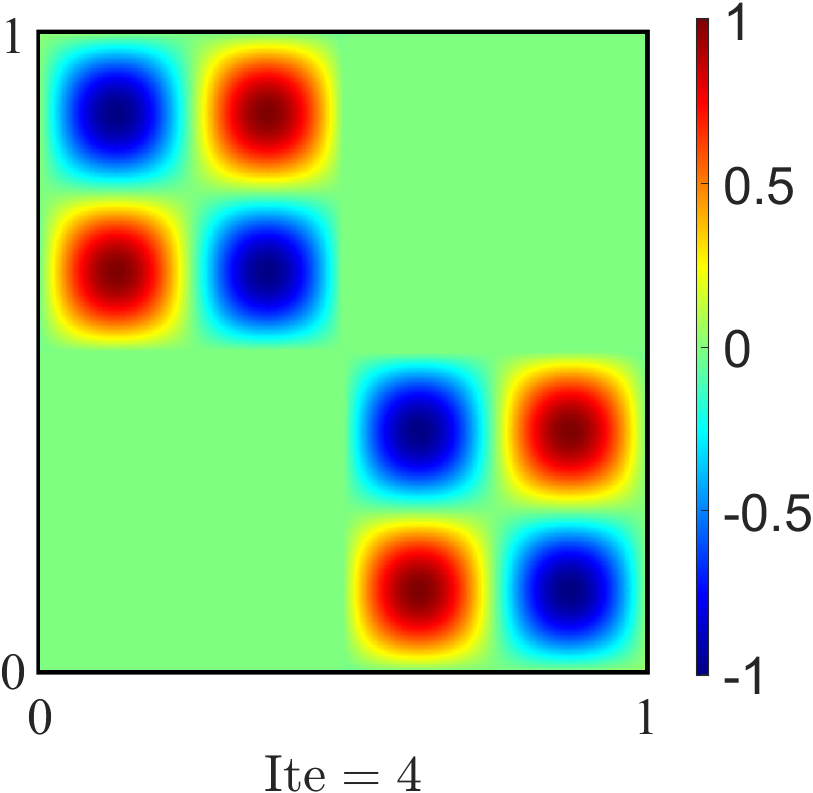}\vspace*{-0.08cm}
\end{minipage}
\\
\\
\makecell{DNLA \\ (deep Ritz)} &
\begin{minipage}{.23\textwidth}
\centering
\includegraphics[width=1.\linewidth]{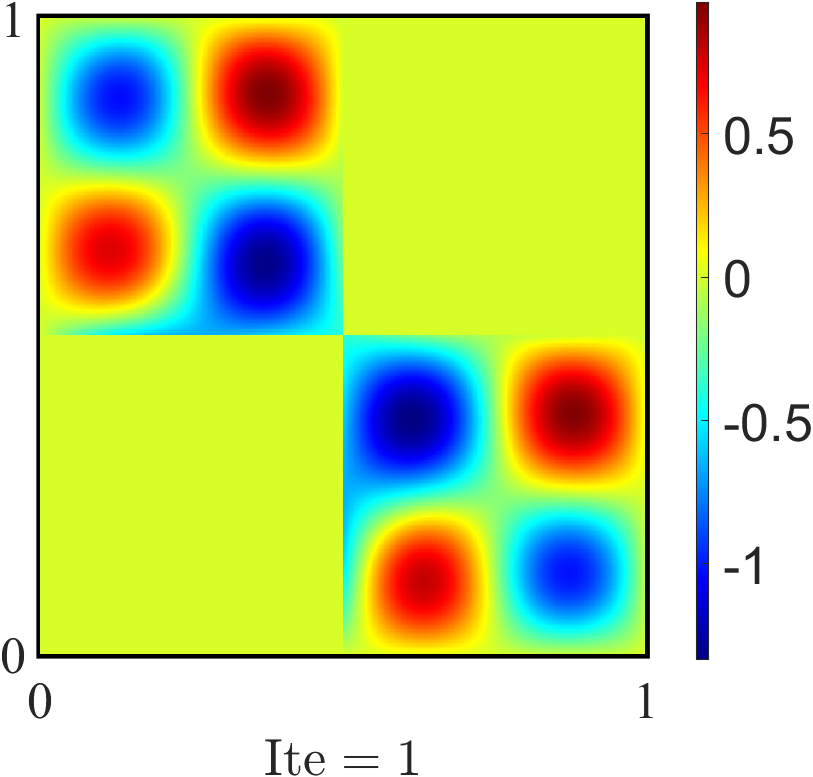}\vspace*{-0.08cm}
\end{minipage}
&
\begin{minipage}{.23\textwidth}
\centering
\includegraphics[width=1\linewidth]{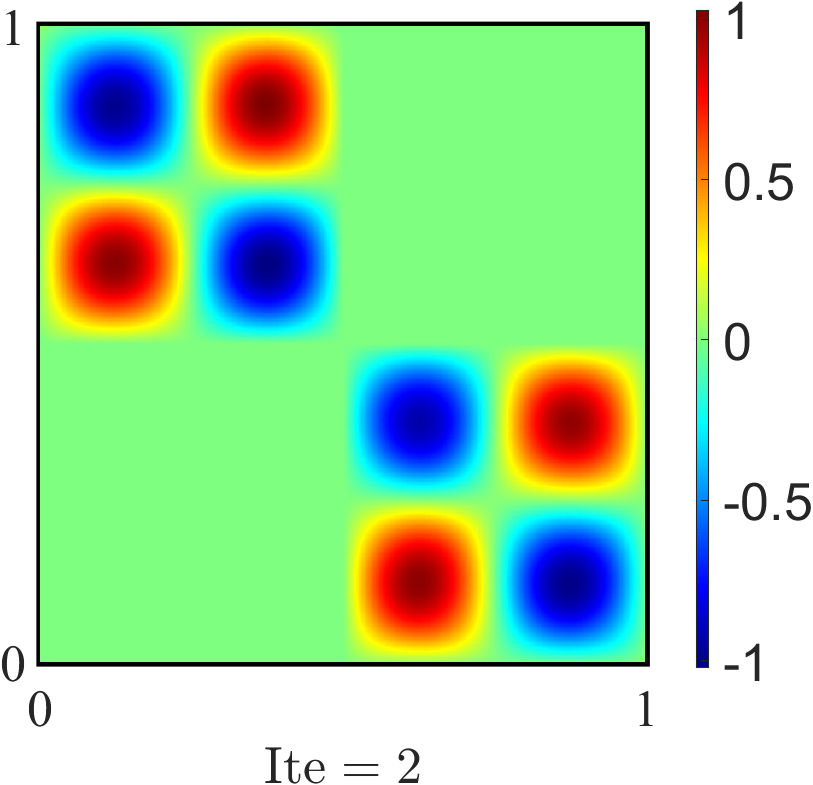}\vspace*{-0.08cm}
\end{minipage}
& 
\begin{minipage}{.23\textwidth}
\centering
\includegraphics[width=1\linewidth]{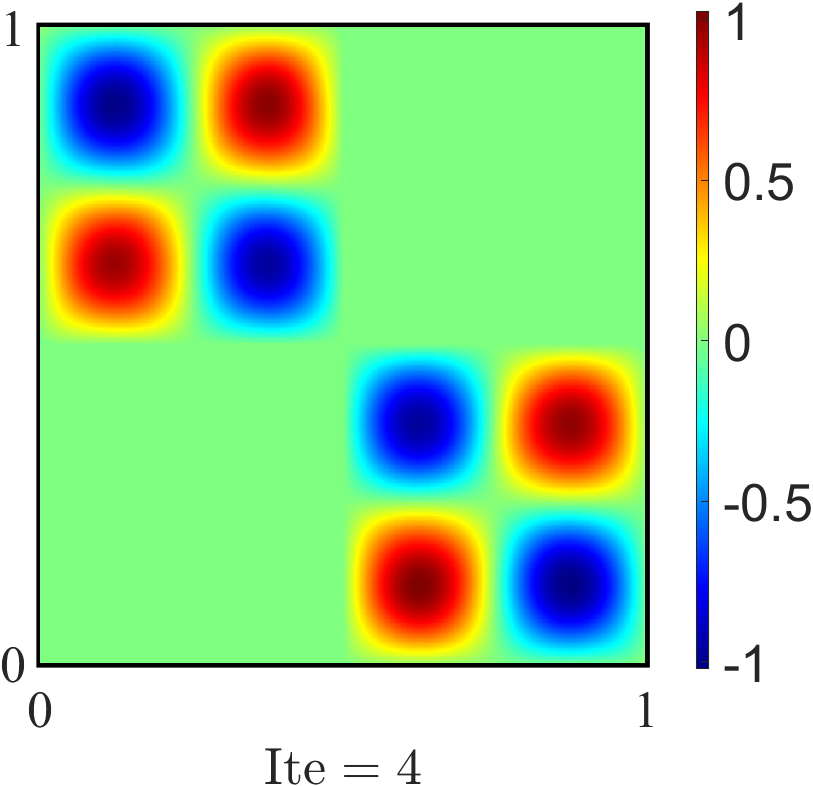}\vspace*{-0.08cm}
\end{minipage}
\end{tabular}}
\end{table}

\hfill

\begin{table}[!htb]
\caption{The iterative solutions $\hat{u}^{[k]}(x,y;\theta)$  for numerical example \eqref{NumXmp3-Checkerboard} with $(c_1,c_2)=(1,1)$.}
\centering
\adjustbox{max width=\textwidth}{
\centering
\begin{tabular}{ c c  c  c  c }
\makecell{DeepDDM} &
\begin{minipage}{.23\textwidth}
\centering
\includegraphics[width=\linewidth]{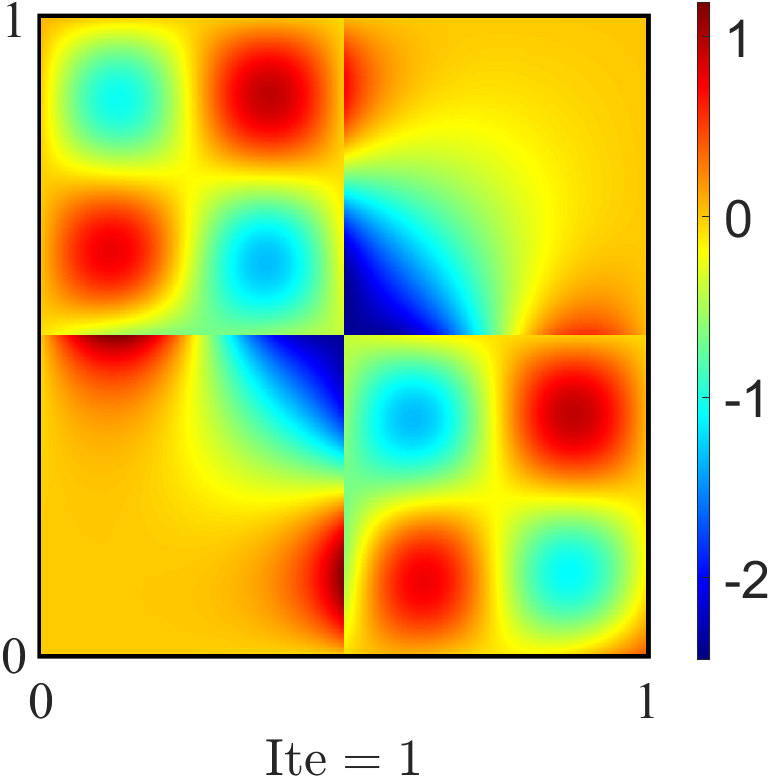}\vspace*{-0.08cm}
\end{minipage}
&
\begin{minipage}{.23\textwidth}
\centering
\includegraphics[width=\linewidth]{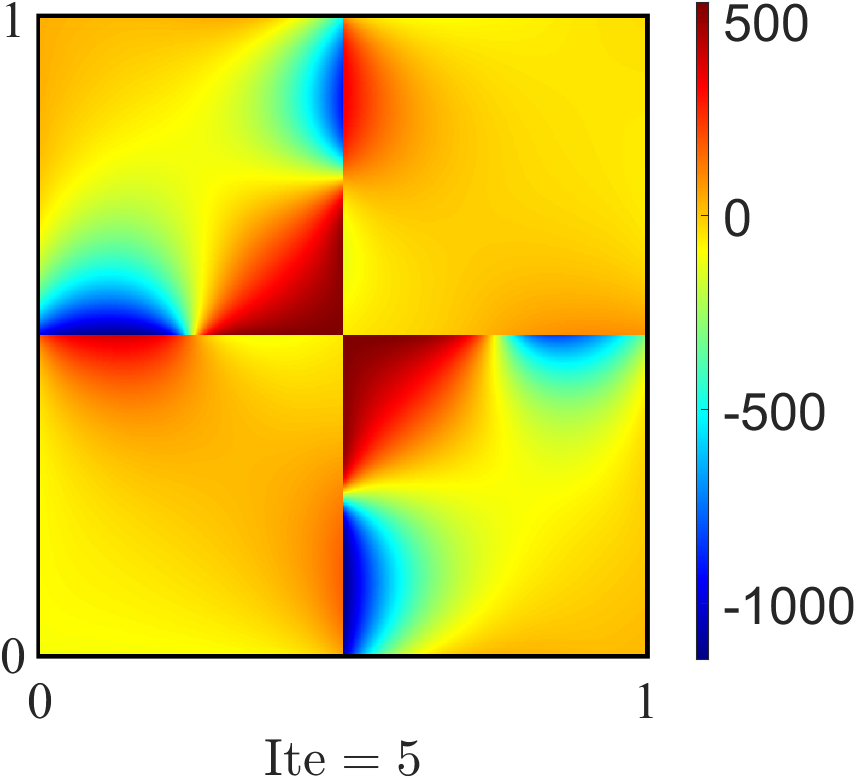}\vspace*{-0.08cm}
\end{minipage}
& 
\begin{minipage}{.23\textwidth}
\centering
\includegraphics[width=\linewidth]{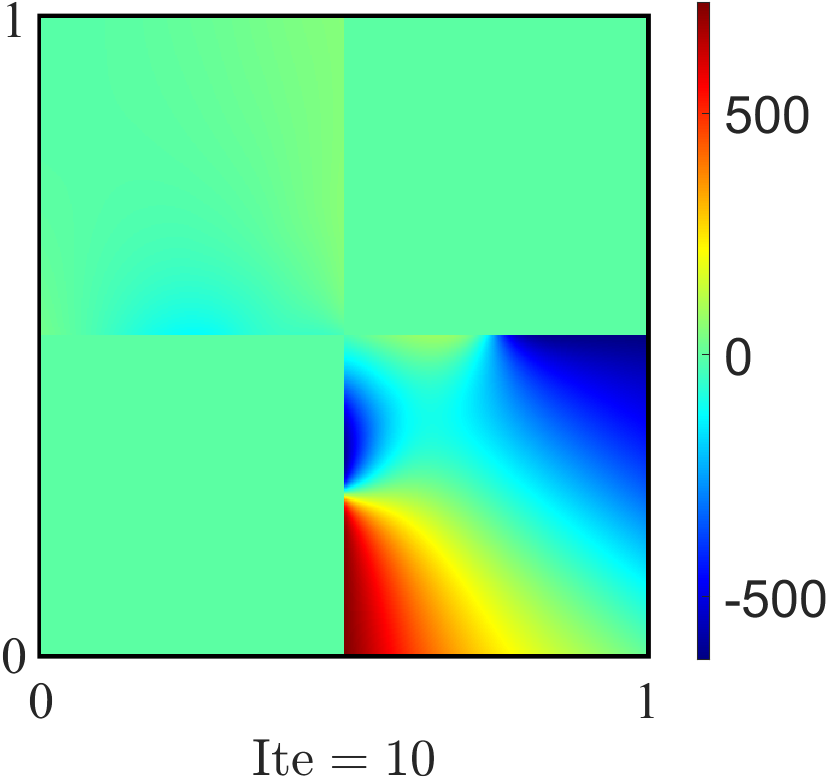}\vspace*{-0.08cm}
\end{minipage}
& 
\begin{minipage}{.23\textwidth}
\centering
\includegraphics[width=\linewidth]{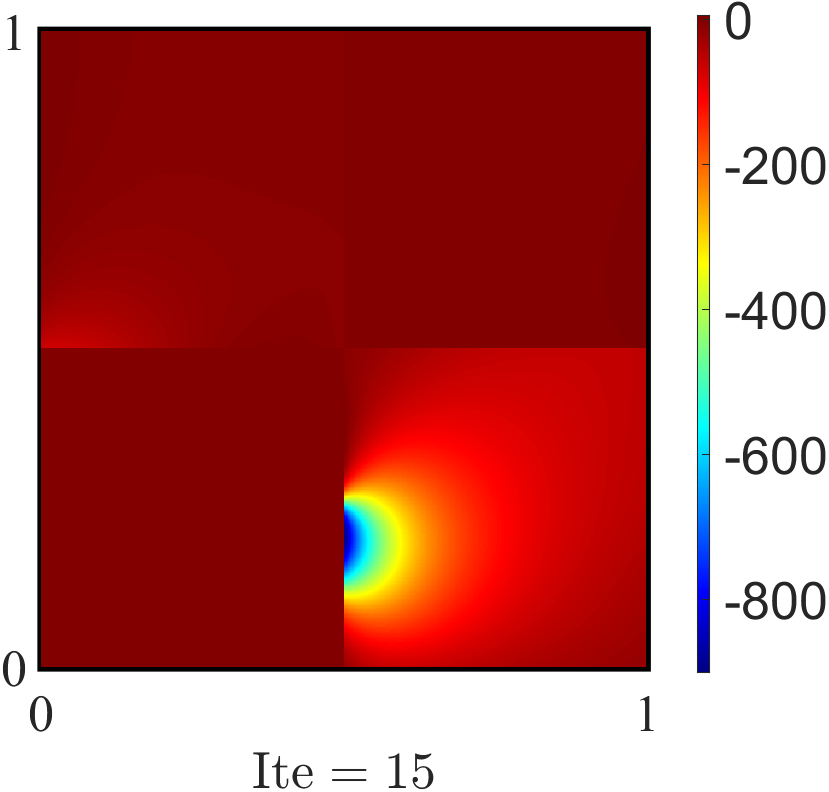}\vspace*{-0.08cm}
\end{minipage}
\\ 
\\
\makecell{DNLA \\ (PINNs)} &
\begin{minipage}{.23\textwidth}
\centering
\includegraphics[width=1\linewidth]{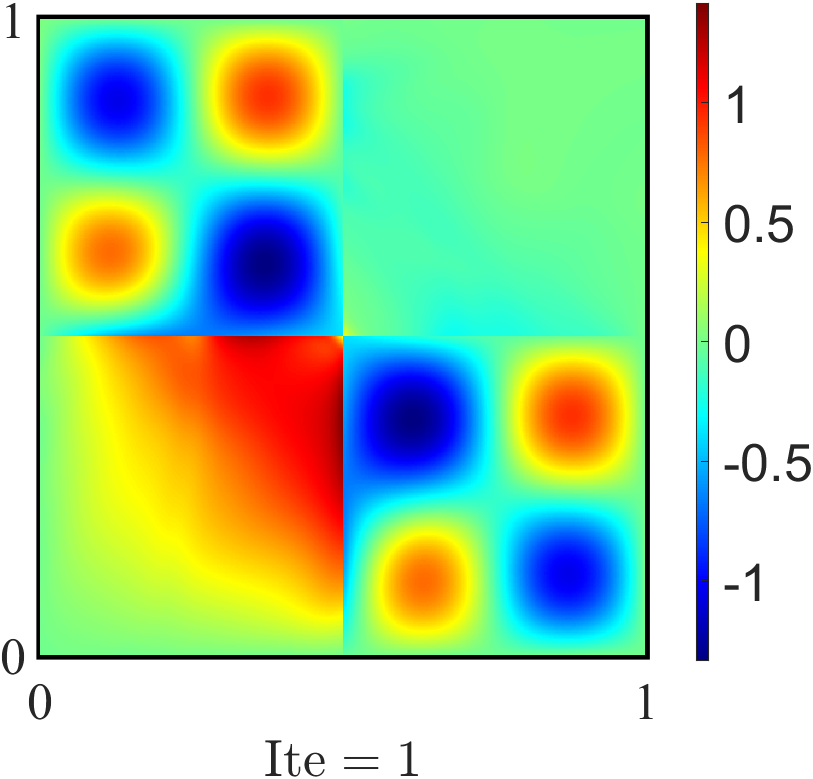}\vspace*{-0.08cm}
\end{minipage}
&
\begin{minipage}{.23\textwidth}
\centering
\includegraphics[width=1\linewidth]{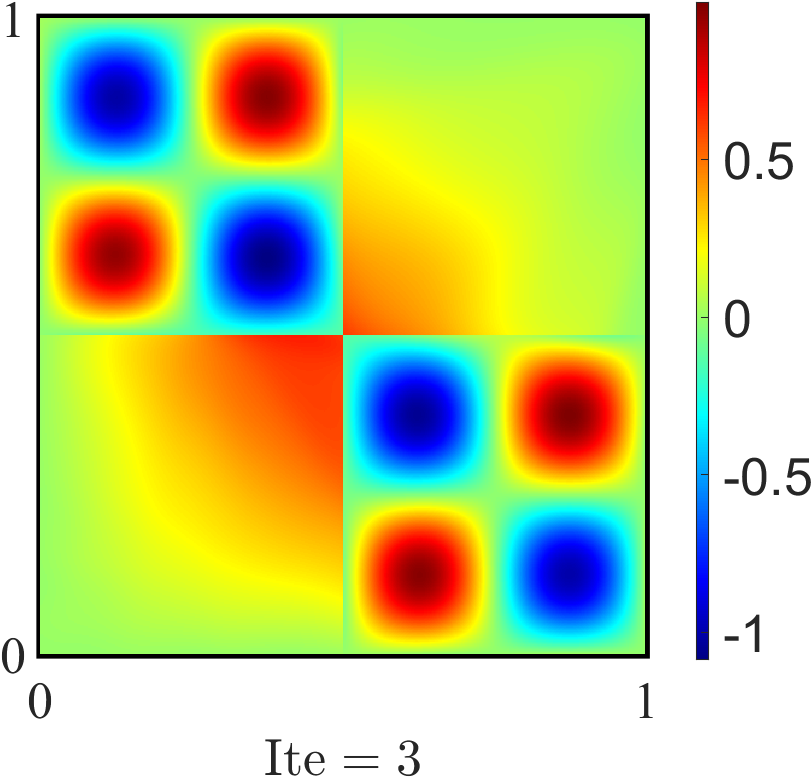}\vspace*{-0.08cm}
\end{minipage}
& 
\begin{minipage}{.23\textwidth}
\centering
\includegraphics[width=1\linewidth]{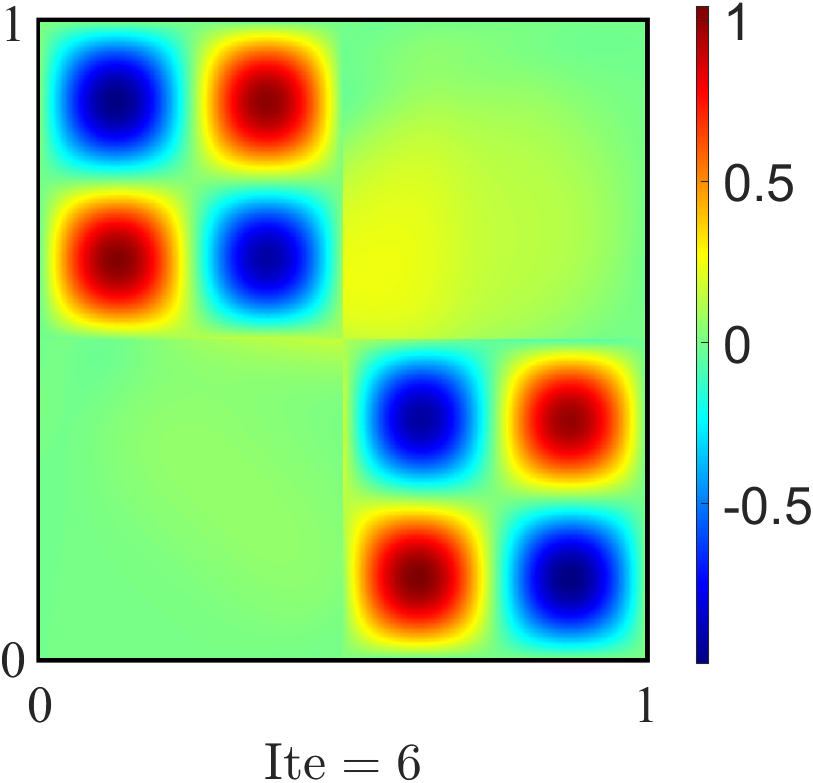}\vspace*{-0.08cm}
\end{minipage}
& 
\begin{minipage}{.23\textwidth}
\centering
\includegraphics[width=1\linewidth]{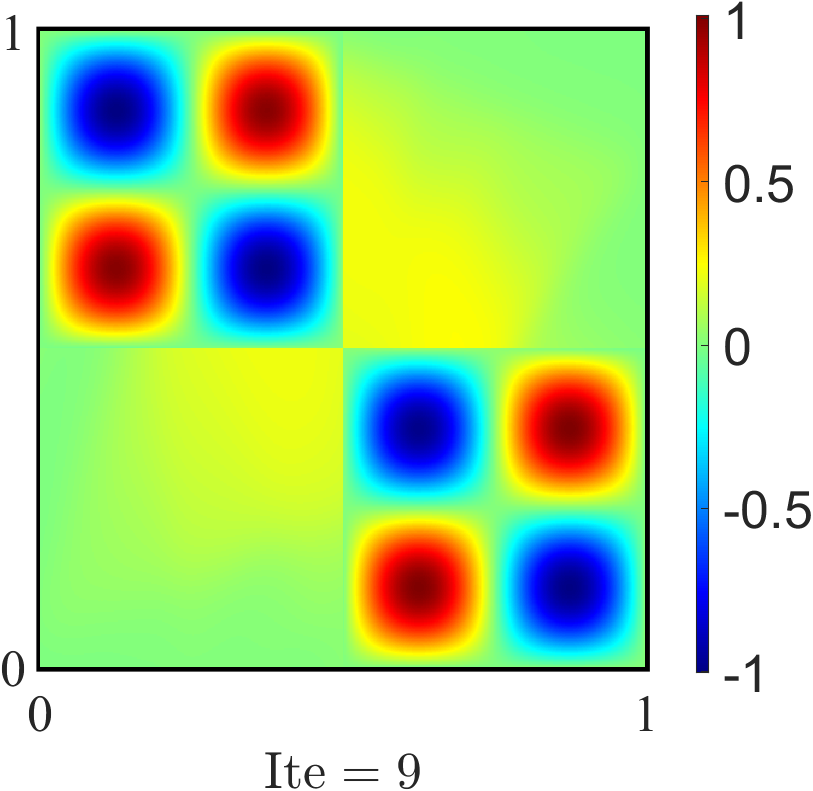}\vspace*{-0.08cm}
\end{minipage}
\\
\\
\makecell{DNLA \\ (deep Ritz)} &
\begin{minipage}{.23\textwidth}
\centering
\includegraphics[width=1.\linewidth]{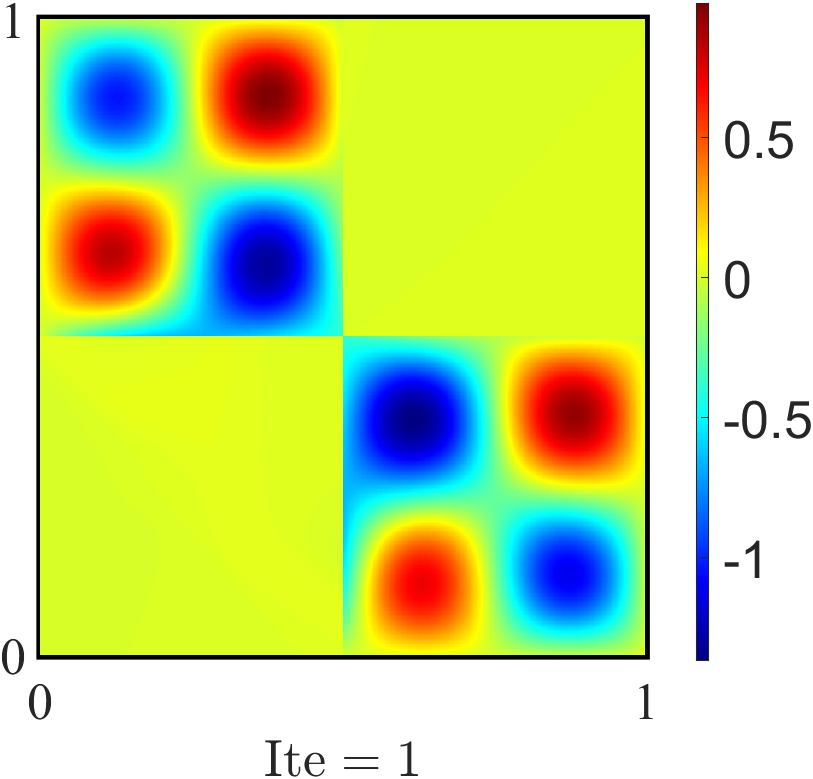}\vspace*{-0.08cm}
\end{minipage}
&
\begin{minipage}{.23\textwidth}
\centering
\includegraphics[width=1\linewidth]{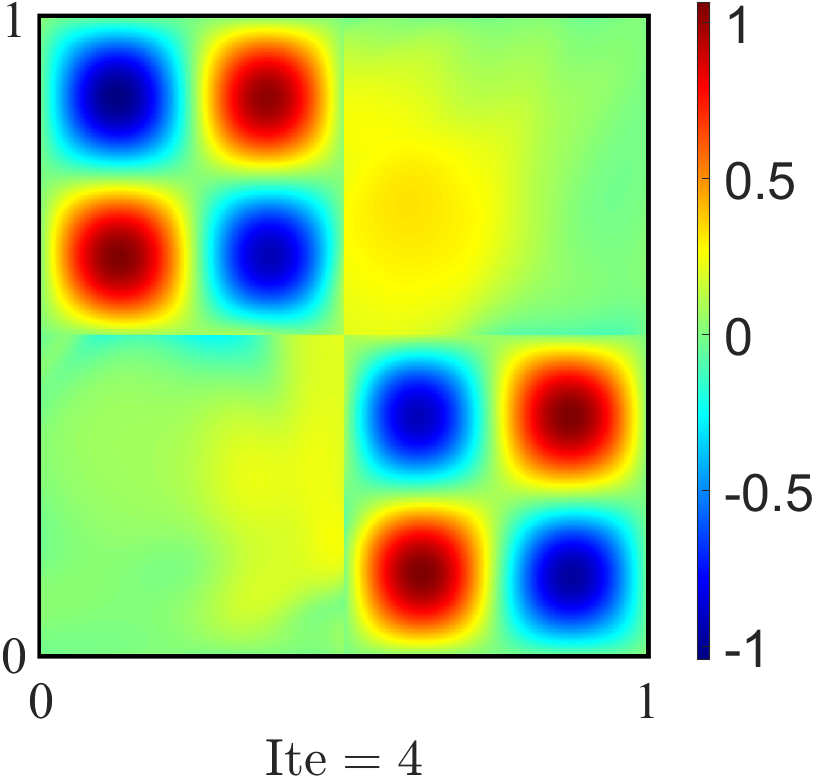}\vspace*{-0.08cm}
\end{minipage}
& 
\begin{minipage}{.23\textwidth}
\centering
\includegraphics[width=1\linewidth]{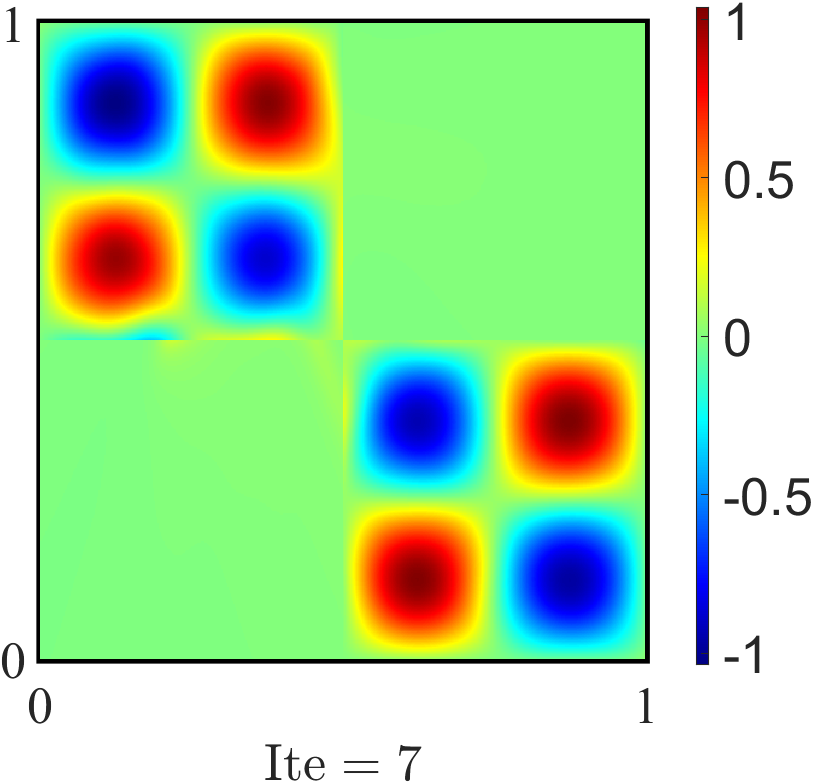}\vspace*{-0.08cm}
\end{minipage}
& 
\begin{minipage}{.23\textwidth}
\centering
\includegraphics[width=1\linewidth]{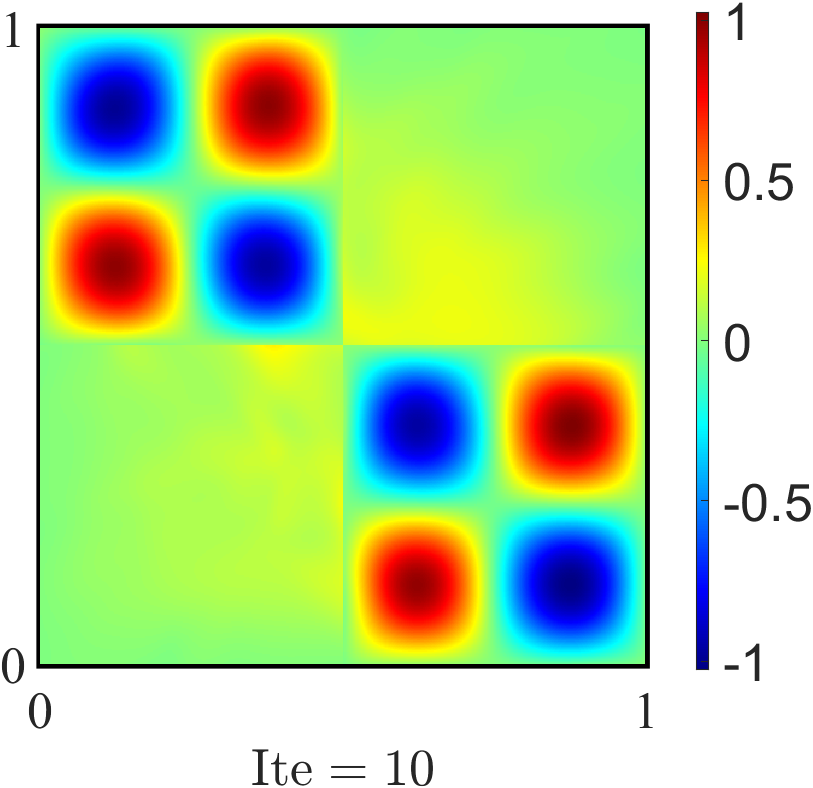}\vspace*{-0.08cm}
\end{minipage}
\end{tabular}}
\end{table}


\newpage

\section*{Acknowledgments}
The computations were done on the high performance computers of School of Mathematical Sciences, Tongji University.

\bibliographystyle{siamplain}
\bibliography{references}
\end{document}